\documentclass[3p,10pt]{elsarticle}
\usepackage{amssymb}
\usepackage{amsmath}
\usepackage{amsthm}
\usepackage{hyperref}
\usepackage{newtx}
\newcommand{\brI}{\mathbf{I}}

\newcommand{\bbC}{\mathbb{C}}
\newcommand{\bbF}{\mathbb{F}}
\newcommand{\bbR}{\mathbb{R}}
\newcommand{\bbZ}{\mathbb{Z}}

\newcommand{\bfA}{{\boldsymbol{A}}}
\newcommand{\bfc}{{\boldsymbol{c}}}
\newcommand{\bfG}{{\boldsymbol{G}}}
\newcommand{\bfL}{{\boldsymbol{L}}}
\newcommand{\bfP}{{\boldsymbol{P}}}
\newcommand{\bfu}{{\boldsymbol{u}}}
\newcommand{\bfU}{{\boldsymbol{U}}}
\newcommand{\bfv}{{\boldsymbol{v}}}
\newcommand{\bfV}{{\boldsymbol{V}}}
\newcommand{\bfx}{{\boldsymbol{x}}}
\newcommand{\bfX}{{\boldsymbol{X}}}
\newcommand{\bfy}{{\boldsymbol{y}}}

\newcommand{\bfone}{{\boldsymbol{1}}}
\newcommand{\bfzero}{{\boldsymbol{0}}}
\newcommand{\bfdelta}{\boldsymbol{\delta}}
\newcommand{\bfDelta}{\boldsymbol{\varDelta}}
\newcommand{\bfGamma}{\boldsymbol{\varGamma}}
\newcommand{\bfchi}{\boldsymbol{\chi}}
\newcommand{\bfphi}{\boldsymbol{\varphi}}
\newcommand{\bfPhi}{\boldsymbol{\varPhi}}
\newcommand{\bfpsi}{\boldsymbol{\psi}}
\newcommand{\bfPsi}{\boldsymbol{\varPsi}}
\newcommand{\bfzeta}{\boldsymbol{\zeta}}

\newcommand{\calA}{\mathcal{A}}
\newcommand{\calB}{\mathcal{B}}
\newcommand{\calD}{\mathcal{D}}
\newcommand{\calE}{\mathcal{E}}
\newcommand{\calG}{\mathcal{G}}
\newcommand{\calK}{\mathcal{K}}
\newcommand{\calL}{\mathcal{L}}
\newcommand{\calM}{\mathcal{M}}
\newcommand{\calN}{\mathcal{N}}
\newcommand{\calQ}{\mathcal{Q}}
\newcommand{\calS}{\mathcal{S}}
\newcommand{\calU}{\mathcal{U}}
\newcommand{\calV}{\mathcal{V}}
\newcommand{\calX}{\mathcal{X}}

\newcommand{\rmB}{\mathrm{B}}
\newcommand{\rmc}{\mathrm{c}}
\newcommand{\rmi}{\mathrm{i}}
\newcommand{\rmQ}{\mathrm{Q}}
\newcommand{\rmT}{\mathrm{T}}

\newcommand{\Aut}{\operatorname{Aut}}
\newcommand{\bin}{{\operatorname{bin}}}
\newcommand{\coh}{{\operatorname{coh}}}
\newcommand{\Fro}{\mathrm{Fro}}
\newcommand{\GL}{\operatorname{GL}}
\newcommand{\im}{\operatorname{im}}
\newcommand{\Orb}{{\operatorname{Orb}}}
\newcommand{\Orth}{\operatorname{O}}
\newcommand{\rank}{\operatorname{rank}}
\newcommand{\sign}{\operatorname{sign}}
\newcommand{\Sp}{\operatorname{Sp}}
\newcommand{\Span}{\operatorname{span}}
\newcommand{\spark}{{\operatorname{spark}}}
\newcommand{\Stab}{{\operatorname{Stab}}}
\newcommand{\Sym}{\operatorname{Sym}}
\newcommand{\TP}{\operatorname{TP}}
\newcommand{\tr}{\operatorname{tr}}
\newcommand{\Tr}{\operatorname{Tr}}

\newcommand{\gen}[1]{\langle{#1}\rangle}
\newcommand{\abs}[1]{|{#1}|}

\newcommand{\biggparen}[1]{\biggl({#1}\biggr)}

\newcommand{\bigbracket}[1]{\bigl[{#1}\bigr]}
\newcommand{\biggbracket}[1]{\biggl[{#1}\biggr]}
\newcommand{\set}[1]{\{{#1}\}}
\newcommand{\bigset}[1]{\bigl\{{#1}\bigr\}}
\newcommand{\norm}[1]{\|{#1}\|}
\newcommand{\biggnorm}[1]{\biggl\|{#1}\biggr\|}
\newcommand{\ip}[2]{\langle{#1},{#2}\rangle}

\theoremstyle{definition}
\newtheorem{theorem}{Theorem}[section]
\newtheorem{lemma}[theorem]{Lemma}
\newtheorem{definition}[theorem]{Definition}
\newtheorem{example}[theorem]{Example}
\newtheorem{remark}[theorem]{Remark}

\begin{document}
\begin{frontmatter}
\title{Doubly transitive equiangular tight frames that contain regular simplices}

\author{Matthew Fickus}
\ead{Matthew.Fickus@afit.edu, Matthew.Fickus@gmail.com}
\author{Evan C.\ Lake}

\address{Department of Mathematics and Statistics, Air Force Institute of Technology, Wright--Patterson AFB, OH 45433}

\begin{abstract}
An equiangular tight frame (ETF) is a finite sequence of equal norm vectors in a Hilbert space that achieves equality in the Welch bound, and so has minimal coherence. The binder of an ETF is the set of all subsets of its indices whose corresponding vectors form a regular simplex. An ETF achieves equality in Donoho and Elad's spark bound if and only if its binder is nonempty. When this occurs, its binder is the set of all linearly dependent subsets of it of minimal size. Moreover, if members of the binder form a balanced incomplete block design (BIBD) then its incidence matrix can be phased to produce a sparse representation of its dual (Naimark complement). A few infinite families of ETFs are known to have this remarkable property. In this paper, we relate this property to the recently introduced concept of a doubly transitive equiangular tight frame (DTETF), namely an ETF for which the natural action of its symmetry group is doubly transitive. In particular, we show that the binder of any DTETF is either empty or forms a BIBD, and moreover that when the latter occurs, any member of the binder of its dual is an oval of this BIBD. We then apply this general theory to certain known infinite families of DTETFs. Specifically, any symplectic form on a finite vector space yields a DTETF, and we compute the binder of it and its dual, showing that the former is empty except in a single notable case, and that the latter consists of affine Lagrangian subspaces. This unifies and generalizes several results from the existing literature. We then consider the binders of four infinite families of DTETFs that arise from quadratic forms over the field of two elements, showing that two of these are empty except in a finite number of cases, whereas the other two form BIBDs that relate to each other, and to Lagrangian subspaces, in nonobvious ways.
\end{abstract}

\end{frontmatter}

%%%%%%%%%%%%%%%%%%%%%%%%%%%%%%%%%%%%%%%%%%%%%%%%%%%%%%%%%%%%%%%%
\section{Introduction}
%%%%%%%%%%%%%%%%%%%%%%%%%%%%%%%%%%%%%%%%%%%%%%%%%%%%%%%%%%%%%%%%

Let \smash{$\bfPhi:=(\bfphi_n)_{n\in\calN}$} be a finite sequence of $N$ nonzero vectors in a Hilbert space $\calE$ over $\bbF\in\set{\bbR,\bbC}$ of dimension $D<N$.
Broadly speaking, the goal of compressed sensing is to uniquely recover, in a practical way, a vector $\bfx\in\bbF^\calN$ from $\bfy=\bfPhi\bfx:=\sum_{n\in\calN}\bfx(n)\bfphi_n$ subject to the assumption that $\bfx$ is sufficiently \textit{sparse},
namely that $\norm{\bfx}_0:=\#\set{n\in\calN: \bfx(n)\neq 0}$ is sufficiently small.
Since two sparse solutions of $\bfPhi\bfx=\bfy$ differ by a (possibly less) sparse member of $\ker(\bfPhi):=\set{\bfx\in\bbF^\calN: \bfPhi\bfx=\bfzero}$,
this can only happen if $\spark(\bfPhi)
:=\min\bigset{\norm{\bfx}_0: \bfzero\neq\bfx\in\ker(\bfPhi)}$ is sufficiently large,
that is, if every $K$-vector subsequence $\bfPhi_\calK:=(\bfphi_n)_{n\in\calK}$ of $\bfPhi$ is linearly independent for some sufficiently large $K$.
Remarkably,
a sufficiently sparse $\bfx$ can be recovered from $\bfy=\bfPhi\bfx$ using convex optimization when $\bfPhi$ has the \textit{restricted isometry property} (RIP),
namely when every such subsequence $\bfPhi_\calK$ is approximately orthogonal in a particular way~\cite{Candes08}.
Though randomly chosen $\bfPhi$ are likely to have the RIP to a remarkable degree,
it is computationally hard to certify this fact for a general $\bfPhi$~\cite{BandeiraDMS13}.
Computing the spark of a general $\bfPhi$ is also hard~\cite{AlexeevCM12}.
The problem of designing a deterministic $\bfPhi$ that is guaranteed to have the RIP to a degree similar to that of its random cousins remains open~\cite{BourgainDFKK11,BandeiraFMW13}.
Much of the recent work on this problem is motivated by two bounds involving the \textit{coherence} of $\bfPhi$,
the \textit{Welch bound}~\cite{Welch74} and Donoho and Elad's \textit{spark bound}~\cite{DonohoE03}, namely
\begin{equation}
\label{eq.Welch and spark bounds}
\coh(\bfPhi)
:=\max_{n_1\neq n_2}\frac{\abs{\ip{\bfphi_{n_1}}{\bfphi_{n_2}}}}{\norm{\bfphi_{n_1}}\norm{\bfphi_{n_2}}}
\geq\biggbracket{\frac{N-D}{D(N-1)}}^{\frac12},
\qquad
\spark(\bfPhi)
\geq\frac1{\coh(\bfPhi)}+1,
\end{equation}
respectively.
The latter suggests that $\bfPhi$ is well-suited for compressed sensing if $\coh(\bfPhi)$ is small.
This idea is reinforced in~\cite{Tropp04}:
any $\bfx\in\bbF^\calN$ with $\norm{\bfx}_0\leq K$ can be uniquely recovered from $\bfPhi\bfx$,
using orthogonal matching pursuit for example,
if $K<\tfrac12\bigset{[\coh(\bfPhi)]^{-1}+1}$.
Moreover, assuming without loss of generality that $\bfPhi$ is equal norm,
it is well known~\cite{StrohmerH03} that it achieves equality in the Welch bound---and thus has minimal coherence---if and only if it is an \textit{equiangular tight frame} (ETF) for $\calE$,
namely if and only if there exists both $A>0$ such that $\sum_{n\in\calN}\ip{\bfphi_n}{\bfy}\bfphi_n=A\bfy$ for all $\bfy\in\calE$ and $C>0$ such that $\abs{\ip{\bfphi_{n_1}}{\bfphi_{n_2}}}=C$ when $n_1\neq n_2$.
Any such $\bfPhi$ has a \textit{dual (Naimark complement)} $\bfPsi=(\bfpsi_n)_{n\in\calN}$ that is an ETF for a space of dimension $N-D$ with $\ip{\bfpsi_{n_1}}{\bfpsi_{n_2}}=-\ip{\bfphi_{n_1}}{\bfphi_{n_2}}$ whenever $n_1\neq n_2$, and that is unique up to unitary transformations.
Infinite families of ETFs are known, arising for example from algebro-combinatorial designs.
See~\cite{FickusM16} for a survey of ETFs.

The \textit{binder} of an ETF $\bfPhi$ is the set $\bin(\bfPhi)$ of all subsets $\calK$ of $\calN$ for which $\bfPhi_\calK$ is a \textit{regular simplex},
that is, a $K$-vector ETF for some $(K-1)$-dimensional subspace of $\calE$.
This set $\bin(\bfPhi)$ gives insight into the efficacy of $\bfPhi$ for compressed sensing:
$\bin(\bfPhi)$ is nonempty if and only if $\bfPhi$ achieves equality in the spark bound,
and when this occurs, $\bin(\bfPhi)$ consists of the smallest subsets $\calK$ of $\calN$ for which $\bfPhi_\calK$ is linearly dependent~\cite{FickusJKM18}.
Moreover, if members of $\bin(\bfPhi)$ form a \textit{balanced incomplete block design} (BIBD), the incidence matrix of this BIBD can be phased to produce a sparse representation of its dual $\bfPsi$~\cite{FickusJMPW19,FickusJKM18}.
A few infinite families of ETFs are known to have this remarkable property~\cite{FickusJMPW19,FickusJKM18,BodmannK20},
and numerical experimentation gives various other small examples that do so as well~\cite{FickusJKM18}.

In this paper, we better explain why so many ETFs have binders that form BIBDs.
The key idea is the recent realization~\cite{IversonM22,IversonM24,DempwolffK23} that many ETFs are \textit{doubly transitive}, that is, have the property that any distinct pair of its indices can be mapped to any other such pair via the natural action of its symmetry group.
In the next section, we establish notation and review known facts that we will need later on.
In Section~3,
we show that the binder of any \textit{doubly transitive equiangular tight frame} (DTETF) is either empty or forms a BIBD,
and moreover that when the latter occurs,
any member of the binder of its dual is an \textit{oval}~\cite{Andriamanalimanana79} of this BIBD (Theorem~\ref{theorem.binders of DTETFs}).
The remainder of the paper concerns the application of this general theory to known infinite families of DTETFs.
In particular, in Section~4, we show that the binder of the dual of any \textit{symplectic ETF} consists of all \textit{affine Lagrangian subspaces} of the underlying finite symplectic space, and compute the parameters of the corresponding BIBD (Theorem~\ref{theorem.binder of dual symplectic}).
This unifies some results of~\cite{FickusJMPW19,FickusJKM18,BodmannK20} and resolves Conjecture~5.12 of~\cite{BodmannK20} as well as another open problem posed there.
We moreover show that the binder of any symplectic ETF is empty except in a single nontrivial case (Theorem~\ref{theorem.binder of symplectic}).
In Section~5, we perform a similar analysis to certain sub-ETFs of symplectic ETFs and their duals that arise from quadratic forms over the field of two elements~\cite{FickusIJK21}.
We use techniques from~\cite{Kantor75,Wertheimer86} to verify that these ETFs are doubly transitive (Theorem~\ref{theorem.sub-ETFs are DT}).
This leads to the new realization that two known infinite families of ETFs have binders that form BIBDs.
These BIBDs are related to the binder of the dual symplectic ETF in nonobvious ways (Theorem~\ref{theorem.binders of sub-ETFs}).
The binders of their duals are empty except in a finite number of exceptional cases (Theorems~\ref{theorem.binder of sub-ETF of symplectic} and~\ref{theorem.binder Tremain}).
A preliminary version of much of the material of Sections~3 and~4 is given in~\cite{Lake23}.
See~\cite{King25} for recent related work.

%%%%%%%%%%%%%%%%%%%%%%%%%%%%%%%%%%%%%%%%%%%%%%%%%%%%%%%%%%%%%%%%
\section{Preliminaries}
%%%%%%%%%%%%%%%%%%%%%%%%%%%%%%%%%%%%%%%%%%%%%%%%%%%%%%%%%%%%%%%%

We use the phrase ``$(x_n)_{n\in\calN}$ is a finite sequence in $\calX$'' as an abbreviation of ``$n\mapsto x_n$ is a function from some finite nonempty set $\calN$ into some set $\calX$.''
For any such $\calN$ and field $\bbF$ we denote the standard basis of the vector space $\bbF^{\calN}:=\set{\bfx:\calN\rightarrow\bbF}$ as $(\bfdelta_n)_{n\in\calN}$, $\bfdelta_n(n'):=1$ when $n'=n$ and $\bfdelta_n(n'):=0$ when $n'\neq n$.
Given two such sets $\calM$ and $\calN$,
we identify a linear mapping $\bfA$ from $\bbF^\calN$ to $\bbF^\calM$ with a member of $\bbF^{\calM\times\calN}$ (an ``$\calM\times\calN$ matrix'') in the standard way,
namely so that $\bfA(\bfdelta_n)=\sum_{m\in\calM}\bfA(m,n)\bfdelta_m$ for all $n\in\calN$.
The \textit{synthesis map} of a finite sequence $(\bfphi_n)_{n\in\calN}$ of vectors in a vector space $\calV$ over some field $\bbF$ is $\bfPhi:\bbF^\calN\rightarrow\calV$, $\bfPhi\bfx:=\sum_{n\in\calN}\bfx(n)\bfphi_n$.
Its image is the span of $(\bfphi_n)_{n\in\calN}$,
that is, $\im(\bfPhi)=\Span(\bfphi_n)_{n\in\calN}$.
Since $\ker(\bfPhi)=\set{\bfzero}$ if and only if $(\bfphi_n)_{n\in\calN}$ is linearly independent,
$\bfPhi$ is thus invertible if and only if $(\bfphi_n)_{n\in\calN}$ is a basis for $\calV$.
Every linear mapping $\bfPhi:\bbF^\calN\rightarrow\calV$ is the synthesis map of the unique finite sequence $(\bfphi_n)_{n\in\calN}$ defined by $\bfphi_n:=\bfPhi\bfdelta_n$ for each $n\in\calN$.
As such, we abuse notation by sometimes denoting $(\bfphi_n)_{n\in\calN}$ itself as $\bfPhi$.

When $\bbF$ is either $\bbR$ or $\bbC$,
we equip $\bbF^\calN$ with the inner product
$\ip{\bfx_1}{\bfx_2}:=\sum_{n\in\calN}\overline{\bfx_1(n)}\bfx_2(n)$.
Each complex inner product in this paper is conjugate-linear in its first argument.
The adjoint of the synthesis map $\bfPhi$ of a finite sequence of vectors $(\bfphi_n)_{n\in\calN}$ in a Hilbert space $\calE$ over $\bbF$ is the \textit{analysis map} $\bfPhi^*:\calE\rightarrow\bbF^\calN$, $(\bfPhi^*\bfy)(n)=\ip{\bfphi_n}{\bfy}$.
As a special case of this,
we sometimes identify a single vector $\bfphi\in\calE$ with the mapping $\bfphi:\bbF\rightarrow\calE$, $\bfphi(z):=z\bfphi$ whose adjoint is the linear functional $\bfphi^*:\calE\rightarrow\bbF$, $\bfphi^*\bfy=\ip{\bfphi}{\bfy}$.
In general,
the corresponding \textit{frame operator} and \textit{Gram matrix} are
$\bfPhi\bfPhi^*:\calE\rightarrow\calE$,
$\bfPhi\bfPhi^*=\sum_{n\in\calN}\bfphi_n^{}\bfphi_n^*$ and $\bfPhi^*\bfPhi\in\bbF^{\calN\times\calN}$, $(\bfPhi^*\bfPhi)(n_1,n_2)=\ip{\bfphi_{n_1}}{\bfphi_{n_2}}$,
respectively.
These operators are positive semidefinite and have $\rank(\bfPhi)=\dim(\Span(\bfphi_n)_{n\in\calN})$ as their rank.
Conversely, by the spectral theorem,
any positive semidefinite $\bfG\in\bbF^{\calN\times\calN}$ factors as $\bfG=\bfPhi^*\bfPhi$ where $\bfPhi$ is the synthesis map of some finite sequence of vectors $(\bfphi_n)_{n\in\calN}$ that lies in a Hilbert space $\calE$ over $\bbF$ of dimension $\rank(\bfG)$.
Such $(\bfphi_n)_{n\in\calN}$ and $\calE$ are only unique up to unitary transformations.
When $\calE=\bbF^\calM$, $\bfPhi$ is the $\calM\times\calN$ matrix whose $n$th column is $\bfphi_n$,
$\bfPhi^*$ is its $\calN\times\calM$ conjugate-transpose,
and $\bfPhi\bfPhi^*$ and $\bfPhi^*\bfPhi$ are their $\calM\times\calM$ and $\calN\times\calN$ products, respectively.

We say that $(\bfphi_n)_{n\in\calN}$ is a \textit{tight frame for $\calE$} if there exists $A>0$ such that $\bfPhi\bfPhi^*=A\brI$,
namely such that $\bfy=\tfrac1A\sum_{n\in\calN}\ip{\bfphi_n}{\bfy}\bfphi_n$ for all $\bfy\in\calE$.
Since this requires $(\bfphi_n)_{n\in\calN}$ to span $\calE$,
we more generally say that $(\bfphi_n)_{n\in\calN}$ is a \textit{tight frame}---without specifying for what space---if it is a tight frame for the subspace $\im(\bfPhi)$ of $\calE$ that it spans,
namely if there exists $A>0$ such that $\bfPhi\bfPhi^*\bfy=A\bfy$ for all $\bfy\in\im(\bfPhi)$, or equivalently, such that $\bfPhi\bfPhi^*\bfPhi=A\bfPhi$.
As we now explain,
this equates to $\bfPhi^*\bfPhi$ being a positive multiple of a projection (orthogonal projection operator).
Indeed, if $\bfPhi\bfPhi^*\bfPhi=A\bfPhi$ for some $A>0$ then $\bfP:=\tfrac1A\bfPhi^*\bfPhi$ satisfies $\bfP^*\bfP=\bfP$.
Conversely, if $\bfPhi^*\bfPhi=A\bfP$ where $A>0$ and $\bfP^*\bfP=\bfP$ then since $\ker(\bfL^*\bfL)=\ker(\bfL)$ in general,
$\ker[\bfPhi(\brI-\bfP)]
=\ker[(\brI-\bfP)A\bfP(\brI-\bfP)]
=\ker(\bfzero)
=\bbF^\calN$
implying
$\bfzero
=\bfPhi(\brI-\bfP)
=\bfPhi-\tfrac1{A}\bfPhi\bfPhi^*\bfPhi$
and so $\bfPhi\bfPhi^*\bfPhi=A\bfPhi$.
Thus, $(\bfphi_n)_{n\in\calN}$ is a tight frame if and only if there exists $A>0$ such that $A\bfPhi^*\bfPhi=(\bfPhi^*\bfPhi)^2$,
that is, such that
\begin{equation}
\label{eq.tight frame condition}
A\ip{\bfphi_{n_1}}{\bfphi_{n_3}}
=\sum_{n_2\in\calN}\ip{\bfphi_{n_1}}{\bfphi_{n_2}}\ip{\bfphi_{n_2}}{\bfphi_{n_3}}
\end{equation}
for all $n_1,n_3\in\calN$.
Since any positive semidefinite matrix is a Gram matrix,
this moreover implies that a matrix $\bfG\in\bbF^{\calN\times\calN}$ is the Gram matrix of a tight frame $(\bfphi_n)_{n\in\calN}$ if and only if $\bfG$ is a positive multiple of a projection,
that is, $\bfG^*=\bfG$ and $\bfG^2=A\bfG$ for some $A>0$.
When this occurs,
\begin{equation}
\label{eq.tightness constant relation}
\bfG=\bfPhi^*\bfPhi,
\qquad
\dim(\Span(\bfphi_n)_{n\in\calN})
=\rank(\bfPhi)
=\rank(\bfG)
=\tfrac1A\Tr(\bfG).
\end{equation}
Since $\brI-\bfP$ is a projection whenever $\bfP$ is as well,
any tight frame $(\bfphi_n)_{n\in\calN}$ has a \textit{dual (Naimark complement)} $(\bfpsi_n)_{n\in\calN}$ that has $\bfPsi^*\bfPsi=A\brI-\bfPhi^*\bfPhi$ as its Gram matrix, and so is a tight frame for a space of dimension $\#(\calN)-\dim(\Span(\bfphi_n)_{n\in\calN})$ that is unique up to unitary transformations.

\subsection{Equiangular tight frames}

The members of any finite sequence $(\bfphi_n)_{n\in\calN}$ of $N:=\#(\calN)$ nonzero vectors in $\calE$ can be normalized,
and this preserves its spark and coherence~\eqref{eq.Welch and spark bounds}.
A finite sequence $(\bfphi_n)_{n\in\calN}$ of unit norm vectors in $\calE$ has the $(K,\delta)$-RIP for some integer $2\leq K\leq N$ and real number $0\leq\delta<1$ if every $K$-vector subsequence $(\bfphi_n)_{n\in\calK}$ of $(\bfphi_n)_{n\in\calN}$ is approximately orthonormal in the sense that its synthesis map $\bfPhi_\calK$ satisfies
$\norm{\bfPhi_\calK^*\bfPhi_\calK^{}-\brI}_2\leq\delta$,
that is, if every eigenvalue of $\bfPhi_\calK^*\bfPhi_{\calK}^{}$ lies in the interval $[1-\delta,1+\delta]$.
Since $\delta<1$ and $\ker(\bfPhi_\calK)=\ker(\bfPhi_\calK^*\bfPhi_\calK^{})$,
this is only possible if every $K$-vector subsequence of $(\bfphi_n)_{n\in\calN}$ is linearly independent, namely only if $\spark(\bfPhi)\geq K+1$.
In the special case when $K=2$,
the optimal (minimal) such $\delta$ is the coherence of $(\bfphi_n)_{n\in\calN}$:
\begin{equation*}
\max_{\#(\calK)=2}\norm{\bfPhi_\calK^*\bfPhi_\calK^{}-\brI}_2
=\max_{n_1\neq n_2}\biggnorm{\left[\begin{array}{cc}
0&\ip{\bfphi_{n_1}}{\bfphi_{n_2}}\\
\ip{\bfphi_{n_2}}{\bfphi_{n_1}}&0\end{array}\right]}_2
=\max_{n_1\neq n_2}\abs{\ip{\bfphi_{n_1}}{\bfphi_{n_2}}}
=\coh(\bfPhi).
\end{equation*}
For larger values of $K$,
the Gershgorin circle theorem gives the following (often poor) estimate:
\begin{equation*}
\max_{\#(\calK)=K}\norm{\bfPhi_\calK^*\bfPhi_\calK^{}-\brI}_2
\leq\max_{\#(\calK)=K}\max_{n_1\in\calK}\sum_{n_2\in\calK\backslash\set{n_1}}
\abs{\ip{\bfphi_{n_1}}{\bfphi_{n_2}}}
\leq(K-1)\,\coh(\bfPhi).
\end{equation*}
In particular, we have $\spark(\bfPhi)\geq K+1$ for any integer $K\geq 2$ such that $(K-1)\,\coh(\bfPhi)<1$.
Taking $K=\bigl\lceil[\,\coh(\bfPhi)]^{-1}\bigr\rceil$ here gives the spark bound~\eqref{eq.Welch and spark bounds}.
For these and other reasons~\cite{StrohmerH03,Tropp04},
we search for finite sequences $(\bfphi_n)_{n\in\calN}$ of unit norm vectors for which $\coh(\bfPhi)$ is minimal.

In that search, it is convenient to more generally let $(\bfphi_n)_{n\in\calN}$ be \textit{equal norm},
that is, to assume there exists $S>0$ such that $\norm{\bfphi_n}^2=S$ for all $n$.
If $(\bfphi_n)_{n\in\calN}$ is both equal norm and a tight frame for $\calE$ then~\eqref{eq.tightness constant relation} gives that \smash{$D=\tfrac{NS}{A}$} where $D:=\dim(\calE)$,
i.e., $A=\tfrac{NS}{D}$.
When this occurs, $(\bfphi_n)_{n\in\calN}$ has a dual $(\bfpsi_n)_{n\in\calN}$ that satisfies $\bfPsi^*\bfPsi=A\brI-\bfPhi^*\bfPhi$, namely
\begin{equation}
\label{eq.dual of ENTF}
\norm{\bfpsi_n}^2=\tfrac{(N-D)S}{D},\ \forall\,n\in\calN,
\qquad
\ip{\bfpsi_{n_1}}{\bfpsi_{n_2}}=-\ip{\bfphi_{n_1}}{\bfphi_{n_2}},\ \forall\,n_1,n_2\in\calN,\,n_1\neq n_2,
\end{equation}
and so is itself an equal norm tight frame for a Hilbert space of dimension $N-D$ when $D<N$.

If $(\bfphi_n)_{n\in\calN}$ is equal norm then it is thus a tight frame for $\calE$ if and only if it achieves equality in the following inequality,
which itself can be proven by cycling traces in a straightforward way:
\begin{equation*}
0\leq\norm{\bfPhi\bfPhi^*-\tfrac{NS}{D}\brI}_\Fro^2
=\sum_{n_1\in\calN}\sum_{n_2\in\calN\backslash\set{n_1}}
\abs{\ip{\bfphi_{n_1}}{\bfphi_{n_2}}}^2-\tfrac{N(N-D)S^2}{D}.
\end{equation*}
Rearranging and then continuing this inequality gives
\begin{equation}
\label{eq.Welch bound derivation}
\tfrac{N-D}{D(N-1)}
\leq\tfrac1{N(N-1)}\sum_{n_1\in\calN}\sum_{n_2\in\calN\backslash\set{n_1}}
\tfrac1{S^2}\abs{\ip{\bfphi_{n_1}}{\bfphi_{n_2}}}^2
\leq\max_{n_1\neq n_2}\tfrac1{S^2}\abs{\ip{\bfphi_{n_1}}{\bfphi_{n_2}}}^2
=[\coh(\bfPhi)]^2.
\end{equation}
When $N>D$, taking a square root of~\eqref{eq.Welch bound derivation} yields the Welch bound~\eqref{eq.Welch and spark bounds}.
Equality holds in it precisely when it holds in both inequalities of~\eqref{eq.Welch bound derivation}.
This equates to $(\bfphi_n)_{n\in\calN}$ being an ETF for $\calE$,
namely a tight frame for $\calE$ that is \textit{equiangular} in the sense that it is equal norm and there exists $C>0$ such that $\abs{\ip{\bfphi_{n_1}}{\bfphi_{n_2}}}^2=C$ whenever $n_1\neq n_2$.

We more generally say that $(\bfphi_n)_{n\in\calN}$ is an ETF---without specifying for what space---if it is both equiangular and a tight frame.
A matrix $\bfG\in\bbF^{\calN\times\calN}$ is thus the Gram matrix of an ETF if and only if it is a positive multiple of a projection (tightness) and its diagonal entries have constant positive value while its off-diagonal entries have constant positive modulus (equiangularity).
Accordingly,
an equiangular finite sequence $(\bfphi_n)_{n\in\calN}$ is an ETF if and only if~\eqref{eq.tight frame condition} holds for all $n_1,n_3\in\calN$.
Here, without loss of generality scaling $(\bfphi_n)_{n\in\calN}$ so that $\abs{\ip{\bfphi_{n_1}}{\bfphi_{n_2}}}=1$ whenever $n_1\neq n_2$ and $\norm{\bfphi_n}^2=S$ for all $n\in\calN$,
note that~\eqref{eq.tight frame condition} holds when $n_1=n_3$ if and only if \smash{$A=\tfrac{N-1}{S}+S$},
and also that~\eqref{eq.tight frame condition} holds when $n_1\neq n_3$ if and only if
$A=\sum_{n_2\in\calN}\ip{\bfphi_{n_1}}{\bfphi_{n_2}}\ip{\bfphi_{n_2}}{\bfphi_{n_3}}\ip{\bfphi_{n_3}}{\bfphi_{n_1}}$.
As such, when scaled in this way,
an equiangular finite sequence $(\bfphi_n)_{n\in\calN}$ is an ETF if and only if
\begin{equation}
\label{eq.triple product tightness}
\tfrac{N-1}{S}-S
%=\sum_{n_2\in\calN}\ip{\bfphi_{n_1}}{\bfphi_{n_2}}\ip{\bfphi_{n_2}}{\bfphi_{n_3}}\ip{\bfphi_{n_3}}{\bfphi_{n_1}}
=\sum_{n_2\in\calN\backslash\set{n_1,n_3}}\ip{\bfphi_{n_1}}{\bfphi_{n_2}}\ip{\bfphi_{n_2}}{\bfphi_{n_3}}\ip{\bfphi_{n_3}}{\bfphi_{n_1}},
\quad
\forall\,n_1,n_3\in\calN,\, n_1\neq n_3.
\end{equation}
The summands of~\eqref{eq.triple product tightness} are known as \textit{triple products}.
When an ETF $(\bfphi_n)_{n\in\calN}$ is scaled in this way,
\eqref{eq.dual of ENTF} implies that its dual $(\bfpsi_n)_{n\in\calN}$ is an ETF that enjoys this same scaling, that is,
$\abs{\ip{\bfpsi_{n_1}}{\bfpsi_{n_2}}}=1$ whenever $n_1\neq n_2$.

\subsection{Equiangular tight frames that contain regular simplices}

Let $\bfPhi=(\bfphi_n)_{n\in\calN}$ be an $N$-vector ETF,
scaled so that $\abs{\ip{\bfphi_{n_1}}{\bfphi_{n_2}}}=1$ whenever $n_1\neq n_2$ and so we also have
\smash{$\norm{\bfphi_n}^2=S=[\tfrac{D(N-1)}{N-D}]^{\frac12}$} for all $n$,
where $D=\dim(\Span(\bfphi_n)_{n\in\calN})$.
Since $\coh(\bfPhi)=\frac1S$, the spark bound~\eqref{eq.Welch and spark bounds} gives $\spark(\bfPhi)\geq S+1$,
namely that any fewer than $S+1$ of these vectors are linearly independent.
Any $K$-vector subsequence $\bfPhi_\calK=(\bfphi_n)_{n\in\calK}$ of $\bfPhi$ is equiangular and so is itself an ETF if it happens to be a tight frame.
Though seemingly rare, this remarkable phenomenon nevertheless happens infinitely often in various nontrivial ways~\cite{FickusMT12,FickusMJ16,FickusJMP18,FickusJKM18,FickusS20,FickusIJK21}.
As a special case of this, we say that $\bfPhi_\calK$ is a \textit{regular simplex} if it is an ETF with $\rank(\bfPhi_\calK)=K-1$.
The \textit{binder} of an ETF $\bfPhi$ is the set $\bin(\bfPhi)$ of all subsets $\calK$ of $\calN$ for which $\bfPhi_\calK$ is a regular simplex.

Though seemingly hard to compute in general,
this binder helps us gauge the ETF's suitability for compressed sensing:
by Theorem~3.1 of~\cite{FickusJKM18},
$\bin(\bfPhi)$ is nonempty if and only if $\spark(\bfPhi)=S+1$,
and moreover when this occurs, $\calK\in\bin(\bfPhi)$ if and only if $\bfPhi_\calK$ is a linearly dependent subsequence of $\bfPhi$ of minimal cardinality $K=S+1$.
To explain, if $\bin(\bfPhi)$ is nonempty then $\bfPhi_\calK$ is a regular simplex for some $K$-element subset $\calK$ of $\calN$.
Since both $\bfPhi$ and $\bfPhi_\calK$ achieve equality in the Welch bound,
$\frac1S
=\coh(\bfPhi)
=\coh(\bfPhi_\calK)
=\frac1{K-1}$ and so $K=S+1$.
Thus, $\bfPhi_\calK$ is an $(S+1)$-vector ETF for its $S$-dimensional span.
When combined with the fact that $\spark(\bfPhi)\geq S+1$,
this implies that $\spark(\bfPhi)=S+1$ and that $\bfPhi_\calK$ is a linearly dependent subsequence of $\bfPhi$ of minimal cardinality $K=S+1$.
Conversely, if $\spark(\bfPhi)=S+1$ then there exists an $(S+1)$-vector linearly dependent subsequence $\bfPhi_\calK$ of $\bfPhi$ that spans a subspace of dimension $S$.
Since $\coh(\bfPhi_\calK)=\coh(\bfPhi)=\tfrac1S$,
any such $\bfPhi_\calK$ moreover achieves equality in the Welch bound for $S+1$ vectors in this subspace,
and so is an ETF for this subspace, and thus also a regular simplex.
Any such $\calK$ thus lies in $\bin(\bfPhi)$.

In particular, $\bfPhi$ achieves equality in both the Welch and spark bounds~\eqref{eq.Welch and spark bounds} if and only if it is an ETF that contains a regular simplex.
Such $\bfPhi$ are the ``worst of the best'' with respect to spark,
having the least (worst) spark subject to the greatest (best) spark bound.
From a compressed sensing perspective, it is nevertheless useful to know the binder $\bin(\bfPhi)$ of such an ETF $\bfPhi$ since, for example,
the ratio of the cardinality of $\bin(\bfPhi)$ to $\binom{N}{S+1}$ quantifies how frequently the mapping $\bfx\mapsto\bfPhi\bfx$ confuses minimally sparse nonzero signals $\bfx$ with $\bfzero$.

By the arguments above,
a given subset $\calK$ of $\calN$ belongs to $\bin(\bfPhi)$ if and only if $(\bfphi_n)_{n\in\calK}$ is an $(S+1)$-vector tight frame.
Considering~\eqref{eq.triple product tightness} when ``$\calN$'' is $\calK$ and ``$N$'' is $S+1$,
this occurs if and only if $\#(\calK)=S+1$ and
$-(S-1)=\sum_{n_2\in\calK\backslash\set{n_1,n_3}}\ip{\bfphi_{n_1}}{\bfphi_{n_2}}\ip{\bfphi_{n_2}}{\bfphi_{n_3}}\ip{\bfphi_{n_3}}{\bfphi_{n_1}}$ for all $n_1,n_3\in\calK$ with $n_1\neq n_3$.
Since $S-1$ unimodular numbers sum to $-(S-1)$ if and only if each of these numbers is $-1$, this recovers the triple-product-based characterization of the binder given in Theorem~4.2 of~\cite{FickusJKM18},
namely that $\calK\in\bin(\bfPhi)$ if and only if
\begin{equation}
\label{eq.triple product simplex}
\#(\calK)=S+1,
\quad \ip{\bfphi_{n_1}}{\bfphi_{n_2}}\ip{\bfphi_{n_2}}{\bfphi_{n_3}}\ip{\bfphi_{n_3}}{\bfphi_{n_1}}=-1,
\ \forall\,n_1,n_2,n_3\in\calK,\,n_1\neq n_2\neq n_3\neq n_1.
\end{equation}
In particular,
$\bin(\bfPhi)$ is a set of certain $(S+1)$-element subsets $\calK$ of $\calN$.
As we now explain, objects of this type arise in combinatorial design,
and this has ramifications for the study of ETFs.

In general, a finite \textit{incidence structure} is a pair $(\calV,(\calK_b)_{b\in\calB})$ of a finite nonempty (vertex) set $\calV$ and a finite sequence $(\calK_b)_{b\in\calB}$ of subsets of $\calV$, called \textit{blocks}.
In this paper, we use the term ``incidence structure'' to exclusively refer to instances of such objects for which, for some integer $K\geq 2$, we have $\#(\calK_b)=K$ for all $b\in\calB$.
Such an incidence structure is said to be a BIBD if, for some integer $\Lambda$, any two distinct vertices are contained in exactly $\Lambda$ blocks.
For any $v\in\calV$,
counting $\set{(v',b)\in\calV\times\calB: v,v'\in\calK_b,\, v'\neq v}$ in two ways gives that the number $R$ of blocks that contain $v$ satisfies $(V-1)\Lambda=R(K-1)$.
This implies that $\Lambda\geq 1$ and also that $R$ is independent of $v$.
Similarly, $B=\#(\calB)$ satisfies $BK=\#\set{(v,b)\in\calV\times\calB: v\in\calK_b}=VR$.
The \textit{incidence matrix} of $(\calV,(\calK_b)_{b\in\calB})$ is $\bfX\in\bbR^{\calB\times\calV}$ where $\bfX(b,v):=1$ if $v\in\calK_b$ and $\bfX(b,v):=0$ otherwise.
Letting $\bfone_\calV$, $\bfone_\calB$ and $\bfone_{\calV\times\calV}$ be the appropriate all-ones vectors and matrix, respectively, $\bfX$ satisfies
\begin{equation}
\label{eq.BIBD incidence matrix properties}
\bfX\bfone_\calV=K\bfone_\calB,
\quad
\bfX^\rmT\bfone_\calB=R\bfone_\calV,
\quad
\bfX^\rmT\bfX=(R-\Lambda)\brI+\Lambda\bfone_{\calV\times\calV}.
\end{equation}
Excluding degenerate cases where $K=V$ we moreover have \smash{$\Lambda<\tfrac{\Lambda(V-1)}{K-1}=R$},
implying that $\bfX^\rmT\bfX$ is positive definite and so $V=\rank(\bfX^\rmT\bfX)=\rank(\bfX)\leq B$.
Altogether, a BIBD can only exist if
\begin{equation}
\label{eq.BIBD parameter relations}
R=\tfrac{\Lambda(V-1)}{K-1}\in\bbZ,
\qquad
B=\tfrac{VR}{K}=\tfrac{\Lambda V(V-1)}{K(K-1)}\in\bbZ,
\qquad
K=V\ \text{or}\ V\leq B.
\end{equation}

We thus have that $(\calN,(\calK_b)_{b\in\calB})$ is a BIBD for some finite sequence $(\calK_b)_{b\in\calB}$ of members of $\bin(\bfPhi)$ if and only if the number of corresponding regular simplices that contain any pair of distinct ETF vectors is the same as that of any other such pair.
When this occurs, for any $b\in\calB$,
the dual of the regular simplex $\bfPhi_{\calK_b}$ is an ETF for a one-dimensional space.
Without loss of generality taking it to be $\bbF$,
\eqref{eq.dual of ENTF} gives that such a dual equates to a finite sequence of unimodular scalars $(z_{b,n})_{n\in\calK_b}$ such that
$\overline{z_{b,n_1}}z_{b,n_2}=-\ip{\bfphi_{n_1}}{\bfphi_{n_2}}$ for all $n_1,n_2\in\calK_b$ with $n_1\neq n_2$.
We can let $z_{b,n_0}=1$ for some $n_0\in\calK_b$ and $z_{b,n}=-\ip{\bfphi_{n_0}}{\bfphi_n}$ for all $n\in\calK_b\backslash\set{n_0}$, for example.
Repopulating $\bfX$ with these scalars yields
\begin{equation}
\label{eq.phased incidence}
\bfPsi\in\bbF^{\calB\times\calV},
\quad
\bfPsi(b,n):=\frac1{\sqrt{\Lambda}}\left\{\begin{array}{cl}
z_{b,n},&\ n\in\calK_b,\\
0,&\ n\notin\calK_b.
\end{array}\right.
\end{equation}
For any $n\in\calN$,
\smash{$(\bfPsi^*\bfPsi)(n,n)
=\frac1{\Lambda}\sum_{\set{b\in\calB\, :\, n\in\calK_b}}\abs{z_{b,n}}^2
=\tfrac{R}{\Lambda}
=\tfrac{V-1}{K-1}
=\tfrac{N-1}{S}
=\tfrac{(N-D)S}{D}$}.
When instead $n_1\neq n_2$,
$(\bfPsi^*\bfPsi)(n_1,n_2)
=\tfrac1{\Lambda}\sum_{\set{b\in\calB\, :\, n_1,n_2\in\calK_b}}
\overline{z_{b,n_1}}z_{b,n_2}
=-\ip{\bfphi_{n_1}}{\bfphi_{n_2}}$.
Comparing this against~\eqref{eq.dual of ENTF} reveals that $\bfPsi$ is (the synthesis map of) a dual of $\bfPhi$.
Theorems~5.1 and 5.2 of~\cite{FickusJKM18} give this result and its converse:
if $(\calN,(\calK_b)_{b\in\calB})$ is a BIBD and~\eqref{eq.phased incidence} defines an ETF $\bfPsi$ for some choice of unimodular scalars $((z_{b,n})_{n\in\calK_b})_{b\in\calB}$ such that $\overline{z_{b_1,n_1}}z_{b_1,n_2}=\overline{z_{b_2,n_1}}z_{b_2,n_2}$ whenever $n_1,n_2\in\calK_{b_1}\cap\calK_{b_2}$,
then every block $\calK_b$ belongs to the binder of its dual $\bfPhi$.

In general, we say that a finite sequence $(\bfpsi_n)_{n\in\calN}$ of vectors in a Hilbert space is \textit{equivalent} to another such sequence $(\bfphi_n)_{n\in\calN}$ if there exists a finite sequence $(z_n)_{n\in\calN}$ of unimodular scalars such that $\ip{\bfpsi_{n_1}}{\bfpsi_{n_2}}=\overline{z_{n_1}}z_{n_2}\ip{\bfphi_{n_1}}{\bfphi_{n_2}}$ for all $n_1,n_2\in\calN$,
namely such that $\bfpsi_n=z_n\bfU\bfphi_n$ for all $n$, where $\bfU$ is some unitary transformation from the span of $(\bfphi_n)_{n\in\calN}$ onto that of $(\bfpsi_n)_{n\in\calN}$.
Any \textit{cycle} of inner products is preserved by this notion of equivalence,
including $\norm{\bfphi_n}^2=\ip{\bfphi_n}{\bfphi_n}$ for any $n$,
$\abs{\ip{\bfphi_{n_1}}{\bfphi_{n_2}}}^2
=\ip{\bfphi_{n_1}}{\bfphi_{n_2}}\ip{\bfphi_{n_2}}{\bfphi_{n_1}}$ for any $n_1$ and $n_2$,
and triple products:
\begin{equation}
\label{eq.triple product invariance}
\ip{z_{n_1}\bfphi_{n_1}}{z_{n_2}\bfphi_{n_2}}\ip{z_{n_2}\bfphi_{n_2}}{z_{n_3}\bfphi_{n_3}}\ip{z_{n_3}\bfphi_{n_3}}{z_{n_1}\bfphi_{n_1}}\\
=\ip{\bfphi_{n_1}}{\bfphi_{n_2}}\ip{\bfphi_{n_2}}{\bfphi_{n_3}}\ip{\bfphi_{n_3}}{\bfphi_{n_1}}.
\end{equation}
In particular, by~\eqref{eq.triple product tightness} and~\eqref{eq.triple product simplex}, any finite sequence of vectors that is equivalent to an ETF $(\bfphi_n)_{n\in\calN}$ is itself an ETF and the two ETFs have the same binder.
This notion of equivalence moreover preserves tightness~\eqref{eq.tight frame condition}, duals, the RIP, coherence, and spark.
We caution however that it does not preserve certain other properties of note such as
\textit{centroidal symmetry}~\cite{FickusJMPW18} and \textit{flatness}~\cite{FickusJMP21}.

%%%%%%%%%%%%%%%%%%%%%%%%%%%%%%%%%%%%%%%%%%%%%%%%%%%%%%%%%%%%%%%%
\section{The binders of doubly transitive equiangular tight frames}
%%%%%%%%%%%%%%%%%%%%%%%%%%%%%%%%%%%%%%%%%%%%%%%%%%%%%%%%%%%%%%%%

As detailed above,
the binder of an ETF $\bfPhi=(\bfphi_n)_{n\in\calN}$ is the set $\bin(\bfPhi)$ of all subsets $\calK$ of $\calN$ for which $\bfPhi_\calK=(\bfphi_n)_{n\in\calK}$ is a regular simplex.
Many ETFs have empty binders~\cite{FickusJKM18}.
Indeed by~\eqref{eq.triple product simplex},
this occurs if \smash{$S:=[\frac{D(N-1)}{N-D}]^{\frac12}$} is not an integer,
or if $\ip{\bfphi_{n_1}}{\bfphi_{n_2}}$ is an odd root of unity for any distinct $n_1,n_2\in\calN$.
Other ETFs, including some \textit{harmonic} ETFs, have binders that seemingly consist of a single partition of $\calN$~\cite{FickusJKM18,FickusS20}.
Experimentation reveals that yet other ETFs have nontrivial binders with little discernible combinatorial structure.
That said, a remarkable number of ETFs $\bfPhi$ have binders that are rich enough for there to exist a BIBD of the form $(\calN,(\calK_b)_{b\in\calB})$ where $\calK_b\in\bin(\bfPhi)$ for all $b\in\calB$.
As reviewed above,
this in turn permits a (sparse but high-dimensional) \textit{phased incidence} representation~\eqref{eq.BIBD parameter relations} of its dual $\bfPsi$~\cite{FickusJKM18}.
Infinite families of such ETFs are known for example,
including---via Theorem~5.2 of~\cite{FickusJKM18}---the \textit{phased BIBD} ETFs of Theorems~5.1 and~5.2 of~\cite{FickusJMPW19},
and the \textit{Gabor--Steiner} ETFs of Theorems~5.6 and~5.8 of~\cite{BodmannK20}.
In fact, these results from~\cite{BodmannK20} show that certain ETFs have the following stronger property:

\begin{definition}
We say the binder $\bin(\bfPhi)$ of an ETF $\bfPhi$ \textit{forms a BIBD} if $(\calN,(\calK_b)_{b\in\calB})$ is a BIBD for some enumeration $(\calK_b)_{b\in\calB}$ of $\bin(\bfPhi)$, i.e., for some bijection $b\mapsto\calK_b$ from $\calB$ onto $\bin(\bfPhi)$.
\end{definition}

In particular, when the binder of ETF $\bfPhi$ forms a BIBD we can represent its dual $\bfPsi$ as a phased incidence matrix~\eqref{eq.phased incidence} that has no repeated blocks, nor any ``missing'' blocks.
Numerical experimentation provides a number of small-parameter ETFs with this property~\cite{FickusJKM18} that are not explained by the results of~\cite{BodmannK20}.
In this section, we exploit the recently introduced concept of a \textit{doubly transitive equiangular tight frame} (DTETF) to better explain this phenomenon.
Many ETFs are DTETFs~\cite{IversonM22,IversonM24,DempwolffK23},
and we show that the binder of any DTETF is either empty or forms a BIBD.
In the next section, we combine this result with a key proof technique of~\cite{BodmannK20} to resolve some problems posed there,
showing that the binder of the dual of any \textit{symplectic ETF} forms a BIBD that consists of all affine Lagrangian subspaces.
In Section~5, we use a similar approach to calculate the binder of any member of two other infinite families of DTETFs and their duals.

In general,
a group $\calG$ is a \textit{permutation group} of a finite set $\calN$ if it is a subgroup of the symmetric group $\Sym(\calN)$ of all permutations on $\calN$.
For any such $\calG$, the mapping $(\sigma,n)\mapsto\sigma(n)$ is its \textit{natural action} on $\calN$.
Such an action is \textit{doubly transitive} if, for any $n_1,n_2,n_3,n_4\in\calN$ with $n_1\neq n_2$ and $n_3\neq n_4$, there exists $\sigma\in\calG$ such that $\sigma(n_1)=n_3$ and $\sigma(n_2)=n_4$.
Any such action is necessarily transitive, that is, for any $n_1,n_2\in\calN$ there exists $\sigma\in\calG$ such that $\sigma(n_1)=n_2$.
We consider how such actions arise in the classical theory of BIBDs as well as their connection to ETFs.

Let $V\geq K\geq 2$, and let $(\calV,(\calK_b)_{b\in\calB})$ be an incidence structure with $\#(\calV)=V$ and $\#(\calK_b)=K$ for all $b\in\calB$.
A subset $\calL$ of $\calV$ is an \textit{arc} of $(\calV,(\calK_b)_{b\in\calB})$ if $\#(\calK_b\cap\calL)\leq 2$ for all $b\in\calB$.
A permutation $\sigma$ of $\calV$ is an \textit{automorphism} of $(\calV,(\calK_b)_{b\in\calB})$ if there exists a permutation $\tau$ of $\calB$ such that, for any $v\in\calV$ and $b\in\calB$,
$v\in\calK_b$ if and only if $\sigma(v)\in\calK_{\tau(b)}$.
The corresponding \textit{automorphism group} is the set $\Aut(\calV,(\calK_b)_{b\in\calB})$ of all such $\sigma$,
which is a subgroup of $\Sym(\calV)$.
If $\sigma\in\Aut(\calV,(\calK_b)_{b\in\calB})$ then
$\sigma(\calK_b)
:=\set{\sigma(v): v\in\calK_b}
=\set{\sigma(v): \sigma(v)\in\calK_{\tau(b)}}
=\calK_{\tau(b)}$ for all $b\in\calB$.
In particular, any such $\sigma$ maps any block to another.
The converse is true if we further assume that $(\calK_b)_{b\in\calB}$ contains no repeated blocks, that is, if $\calK_{b_1}\neq\calK_{b_2}$ when $b_1\neq b_2$:
in that case, if $\sigma$ is any permutation of $\calV$ that maps any block to another,
then for each $b\in\calB$, there exists a unique $\tau(b)\in\calB$ such that $\sigma(\calK_b)=\calK_{\tau(b)}$,
and the mapping $b\mapsto\tau(b)$ is moreover a permutation of $\calB$ since if $\tau(b_1)=\tau(b_2)$ then
$\calK_{b_1}
=\sigma^{-1}(\calK_{\tau(b_1)})
=\sigma^{-1}(\calK_{\tau(b_2)})
=\calK_{b_2}$ and so $b_1=b_2$.
In general, we say that the incidence structure $(\calV,(\calK_b)_{b\in\calB})$ is ``doubly transitive'' when the natural action of $\Aut(\calV,(\calK_b)_{b\in\calB})$ on $\calV$ is doubly transitive.
The next result combines combinatorial design folklore with some results about arcs from~\cite{Andriamanalimanana79}.
For the sake of completeness, a proof of it is given in the appendix%
%\ of the arXiv version of this paper~\cite{FickusL23}% COMMENT
.

\begin{lemma}[\cite{Andriamanalimanana79}]
\label{lemma.DT BIBD}
Let $V\geq K\geq 2$ be integers and let $(\calK_b)_{b\in\calB}$ be a finite sequence of $K$-element subsets of the $V$-element set $\calV$.
Then:
\begin{enumerate}
\renewcommand{\labelenumi}{(\alph{enumi})}
\item
If $(\calV,(\calK_b)_{b\in\calB})$ is doubly transitive then it is a BIBD.\smallskip

\item
If $(\calV,(\calK_b)_{b\in\calB})$ is a BIBD,
then any $L$-element arc $\calL$ of it satisfies
$L\leq\frac{V-1}{K-1}+1$.\smallskip

Moreover, $L=\frac{V-1}{K-1}+1$ if and only if $\calL$ is nonempty and $\#(\calK_b\cap\calL)\in\set{0,2}$ for all $b\in\calB$.\smallskip

When this occurs, we say that $\calL$ is an \textit{oval} of $(\calV,(\calK_b)_{b\in\calB})$.\smallskip

\item
If $(\calV,(\calK_b)_{b\in\calB})$ is doubly transitive and has an oval then letting $(\calL_a)_{a\in\calA}$ be an enumeration of its ovals,
$(\calV,(\calL_a)_{a\in\calA})$ is itself a doubly transitive BIBD and has each $\calK_b$ as an oval.
\end{enumerate}
\end{lemma}

The next result is a new application of these classical ideas to the study of ETFs:

\begin{theorem}
\label{theorem.ovals of B-BIBD ETF}
If $\bfPhi$ is an ETF and $(\calN,(\calK_b)_{b\in\calB})$ is a BIBD where $\calK_b\in\bin(\bfPhi)$ for all $b\in\calB$,
then every member of the binder of its dual $\bfPsi$ is an oval of this BIBD.
\end{theorem}

\begin{proof}
Recall that if $\bfPhi=(\bfphi_n)_{n\in\calN}$ is an $N$-vector ETF with $D:=\dim(\Span(\bfphi_n)_{n\in\calN})$ that is without loss of generality scaled so that $\abs{\ip{\bfphi_{n_1}}{\bfphi_{n_2}}}=1$ when $n_1\neq n_2$,
then defining \smash{$S:=[\tfrac{D(N-1)}{N-D}]^{\frac12}$},
we have $\norm{\bfphi_n}^2=S$ for all $n$, and moreover that every $\calK\in\bin(\bfPhi)$ has $\#(\calK)=K:=S+1$.
By~\eqref{eq.dual of ENTF},
its dual $\bfPsi=(\bfpsi_n)_{n\in\calN}$ is an $N$-vector ETF with \smash{$\dim(\Span(\bfpsi_n)_{n\in\calN})=N-D$} and has $\abs{\ip{\bfpsi_{n_1}}{\bfpsi_{n_2}}}=1$ when $n_1\neq n_2$.
Letting \smash{$T:=[\tfrac{(N-D)(N-1)}{D}]^{\frac12}$},
we thus have that $\norm{\bfpsi_n}^2=T$ for all $n$,
and moreover that every $\calL\in\bin(\bfPsi)$ has $\#(\calL)=L:=T+1$.
(We caution that such a set $\calL$ does not necessarily exist.
If one does not, this result is vacuously true.)
In particular,
\begin{equation*}
(K-1)(L-1)
=ST
=[\tfrac{D(N-1)}{N-D}]^{\frac12}[\tfrac{(N-D)(N-1)}{D}]^{\frac12}
=N-1.
\end{equation*}
Now assume that $(\calN,(\calK_b)_{b\in\calB})$ is a BIBD where $\calK_b\in\bin(\bfPhi)$ for all $b\in\calB$.
Applying Lemma~\ref{lemma.DT BIBD} where ``$\calV$'' is $\calN$ gives that a subset $\calL$ of $\calN$ is an oval of this BIBD if and only if it is an arc of it of cardinality $\frac{N-1}{K-1}+1=L$.
In particular, to show that any $\calL\in\bin(\bfPsi)$ is an oval of $(\calN,(\calK_b)_{b\in\calB})$, it suffices to show that $\#(\calK\cap\calL)\leq 2$ for all $\calK\in\bin(\bfPhi)$,
or equivalently, that if $\set{n_1,n_2,n_3}\subseteq\calK\in\bin(\bfPhi)$ and $n_1\neq n_2\neq n_3\neq n_1$ then $\set{n_1,n_2,n_3}\nsubseteq\calL$.
Here, note that by~\eqref{eq.triple product simplex}, having the former gives $\ip{\bfphi_{n_1}}{\bfphi_{n_2}}\ip{\bfphi_{n_2}}{\bfphi_{n_3}}\ip{\bfphi_{n_3}}{\bfphi_{n_1}}=-1$.
But since $\ip{\bfpsi_{n_1}}{\bfpsi_{n_2}}=-\ip{\bfphi_{n_1}}{\bfphi_{n_2}}$ whenever $n_1\neq n_2$, this implies that
$\ip{\bfpsi_{n_1}}{\bfpsi_{n_2}}\ip{\bfpsi_{n_2}}{\bfpsi_{n_3}}\ip{\bfpsi_{n_3}}{\bfpsi_{n_1}}=-(-1)^3=1\neq -1$,
implying via~\eqref{eq.triple product simplex} that indeed $\set{n_1,n_2,n_3}\nsubseteq\calL$.
\end{proof}

One nonobvious consequence of Lemma~\ref{lemma.DT BIBD} and Theorem~\ref{theorem.ovals of B-BIBD ETF} is that, under the hypotheses of Theorem~\ref{theorem.ovals of B-BIBD ETF},
no member of $\bin(\bfPsi)$ can intersect any member of $\bin(\bfPhi)$ in exactly one index $n$.
In the special case where $(\calN,(\calL_a)_{a\in\calA})$ is itself a BIBD where $\calL_a\in\bin(\bfPsi)$ for all $a\in\calA$, there is an alternative way to observe this:
since $\bfPsi^*\bfPsi=A\brI-\bfPhi^*\bfPhi$ where $(\bfPhi^*\bfPhi)^2=A\bfPhi^*\bfPhi$ we have $\norm{\bfPsi\bfPhi^*}_\Fro^2
=\Tr(\bfPhi\bfPsi^*\bfPsi\bfPhi^*)
=\Tr(\bfPhi^*\bfPhi(A\brI-\bfPhi^*\bfPhi))
=\Tr(\bfzero)
=0$,
and so $\bfPsi\bfPhi^*=\bfzero$.
As such, the rows of the particular representation of $\bfPsi$ given in~\eqref{eq.phased incidence} are necessarily orthogonal to the rows of the analogous representation of $\bfPhi$.
Such rows are supported on $\calK\in\bin(\bfPhi)$ and $\calL\in\bin(\bfPsi)$, respectively, and so cannot be orthogonal if $\#(\calK\cap\calL)=1$.

In order to relate these concepts to DTETFs we now review some ideas from~\cite{IversonM22,IversonM24}.
Here let $\bfPhi=(\bfphi_n)_{n\in\calN}$ be an equal-norm finite sequence of vectors in a Hilbert space $\calE$.
A \textit{symmetry} of $\bfPhi$ is a permutation $\sigma$ of $\calN$ for which $(\bfphi_{\sigma(n)})_{n\in\calN}$ is equivalent to $(\bfphi_n)_{n\in\calN}$,
that is, for which there exist unimodular scalars $(z_n)_{n\in\calN}$ such that
$\ip{\bfphi_{\sigma(n_1)}}{\bfphi_{\sigma(n_2)}}
=\overline{z_{n_1}}z_{n_2}\ip{\bfphi_{n_1}}{\bfphi_{n_2}}$ whenever $n_1\neq n_2$.
It is straightforward to verify that the set $\Sym(\bfPhi)$ of all such symmetries $\sigma$ is a subgroup of $\Sym(\calN)$.
If $\bfPhi$ is an equal norm tight frame with $\dim(\Span(\bfphi_n)_{n\in\calN})<\#(\calN)$ then its dual $\bfPsi$ is also an equal norm tight frame, and~\eqref{eq.dual of ENTF} immediately implies that $\Sym(\bfPhi)=\Sym(\bfPsi)$.
Though we refer to $\Sym(\bfPhi)$ as the \textit{symmetry group} of $\bfPhi$,
it is more accurate to call it the symmetry group of the corresponding finite sequence of \textit{lines} $(\gen{\bfphi_n})_{n\in\calN}$,
where $\gen{\bfphi_n}:=\set{z\bfphi_n: z\in\bbF}$ is the one-dimensional subspace of $\calE$ that contains any particular vector $\bfphi_n$.
Indeed, though not necessary for our work here,
it is straightforward to verify~\cite{IversonM22} that $\sigma\in\Sym(\bfPhi)$ if and only if there exists a unitary operator $\bfU$ on $\calE$ such that $\gen{\bfphi_{\sigma(n)}}=\bfU\gen{\bfphi_n}$ for each $n$.
We now consider a streamlined version of a result from~\cite{IversonM22} that relates ETFs to doubly transitive actions.
Our proof of it relies on triple products, and is given in the appendix%
%\ of~\cite{FickusL23}% COMMENT
.

\begin{lemma}[Lemma~1.1 of~\cite{IversonM22}]
\label{lemma.DT implies ETF}
Let $\bfPhi=(\bfphi_n)_{n\in\calN}$ be an equal norm finite sequence of $N$ vectors in a Hilbert space.
Then if $\dim(\Span(\bfphi_n)_{n\in\calN})<N$ and the natural action of its symmetry group $\Sym(\bfPhi)$ on $\calN$ is doubly transitive then $\bfPhi$ is an ETF.
We call any such $\bfPhi$ a DTETF.
\end{lemma}

The dual $\bfPsi$ of any DTETF $\bfPhi$ has $\Sym(\bfPsi)=\Sym(\bfPhi)$ and so is itself a DTETF.
With these ideas in hand, we now state and prove our most fundamental result:

\begin{theorem}
\label{theorem.binders of DTETFs}
If $\bfPhi$ is a DTETF then its binder is either empty or forms a BIBD,
and in the latter case, any member of the binder of its dual $\bfPsi$ is an oval of this BIBD.
\end{theorem}

\begin{proof}
Let $\bfPhi=(\bfphi_n)_{n\in\calN}$ be an ETF, scaled without loss of generality so that \smash{$\abs{\ip{\bfphi_{n_1}}{\bfphi_{n_2}}}=1$} when $n_1\neq n_2$.
Let $\TP(n_1,n_2,n_3):=
\ip{\bfphi_{n_1}}{\bfphi_{n_2}}\ip{\bfphi_{n_2}}{\bfphi_{n_3}}\ip{\bfphi_{n_3}}{\bfphi_{n_1}}$
be the triple product that corresponds to a given $n_1,n_2,n_3\in\calN$.
For any $\sigma\in\Sym(\bfPhi)$, there exists a finite sequence of unimodular scalars $(z_n)_{n\in\calN}$ such that $\ip{\bfphi_{\sigma(n_1)}}{\bfphi_{\sigma(n_2)}}
=\overline{z_{n_1}}z_{n_2}\ip{\bfphi_{n_1}}{\bfphi_{n_2}}$ when $n_1\neq n_2$.
For any $n_1,n_2,n_3\in\calN$,
this implies via~\eqref{eq.triple product invariance} that $\TP(\sigma(n_1),\sigma(n_2),\sigma(n_3))=\TP(n_1,n_2,n_3)$.
In particular, if $\calK$ is a subset of $\calN$ of cardinality $S+1$ and $\TP(n_1,n_2,n_3)=-1$ for all pairwise distinct $n_1,n_2,n_3\in\calK$ then $\sigma(\calK):=\set{\sigma(n): n\in\calK}$ has these same properties.
By~\eqref{eq.triple product simplex}, this implies that any symmetry $\sigma$ of $\bfPhi$ maps any member $\calK$ of the binder of $\bfPhi$ to another such member.
That is, if $\bin(\bfPhi)$ is nonempty then $\Sym(\bfPhi)\leq\Aut(\calN,(\calK_b)_{b\in\calB})$ where $(\calK_b)_{b\in\calB}$ is any enumeration of $\bin(\bfPhi)$.

As such, if $\bfPhi$ is a DTETF with a nonempty binder then
$\Sym(\bfPhi)\leq \Aut(\calN,(\calK_b)_{b\in\calB})$ where the natural action of $\Sym(\bfPhi)$ on $\calN$ is doubly transitive,
implying that the same is true for $\Aut(\calN,(\calK_b)_{b\in\calB})$.
By Lemma~\ref{lemma.DT BIBD},
$(\calN,(\calK_b)_{b\in\calB})$ is a BIBD with $V=N$ and $K=S+1$.
By Theorem~\ref{theorem.ovals of B-BIBD ETF}, any member of the binder of the dual $\bfPsi$ of $\bfPhi$ is an oval of this BIBD.
Here in fact, $\bfPsi$ is itself a DTETF since $\Sym(\bfPsi)=\Sym(\bfPhi)$,
and so when its binder is also nonempty,
any member of either of the two resulting BIBDs is an oval of the other.
\end{proof}

In Theorem II.6 of~\cite{King19}, a similar triple-product-based argument is used to show that every triply transitive ETF is either a regular simplex or an ETF for a space of dimension one.
In fact, such ETFs are totally symmetric.
See also Lemma~6.2 of~\cite{IversonM24}.
Theorem~\ref{theorem.binders of DTETFs} still applies in such cases,
but the resulting BIBDs are degenerate:

\begin{example}
\label{example.binder of simplex}
For any positive integer $S$,
let $\bfPhi=(\bfphi_n)_{n\in\calN}$ be an $(S+1)$-vector regular simplex,
namely an ETF whose span has dimension $S$.
Its dual $\bfPsi$ is thus an $(S+1)$-vector ETF for a space of dimension one.
Without loss of generality scaling $\bfPhi$ and $\bfPsi$ so that $\abs{\ip{\bfphi_{n_1}}{\bfphi_{n_2}}}=\abs{\ip{\bfpsi_{n_1}}{\bfpsi_{n_2}}}=1$ when $n_1\neq n_2$,
any such ETF $\bfPsi$ consists, up to equivalence, of $S+1$ copies of the scalar $1$.
Choosing $\bfPsi$ in this way and modifying its dual $\bfPhi$ accordingly gives $\bfPhi^*\bfPhi=(S+1)\brI-\bfPsi^*\bfPsi=(S+1)\brI-\bfone_{\calN\times\calN}$,
that is, $\norm{\bfphi_n}^2=S$ for all $n$ and
$\ip{\bfphi_{n_1}}{\bfphi_{n_2}}=-1$ when $n_1\neq n_2$.
The synthesis map of such a $\bfPhi$ can be constructed, for example, by taking a possibly-complex Hadamard matrix of size $S+1$, scaling its columns as necessary so that it has an all-ones row,
and then removing that row.

Clearly, every permutation $\sigma$ of $\calN$ is a symmetry of $\bfPsi$, and so $\Sym(\bfPhi)=\Sym(\bfPsi)=\Sym(\calN)$.
Since the natural action of $\Sym(\calN)$ on $\calN$ is doubly transitive---in fact it is totally transitive---this implies that $\bfPhi$ and $\bfPsi$ are DTETFs.
As such, Theorem~\ref{theorem.binders of DTETFs} applies to them.
The binder of the regular simplex $\bfPhi$ consists of the single set $\calN$.
Indeed since $\ip{\bfphi_{n_1}}{\bfphi_{n_2}}\ip{\bfphi_{n_2}}{\bfphi_{n_3}}\ip{\bfphi_{n_3}}{\bfphi_{n_1}}=(-1)^3=-1$
for any $n_1,n_2,n_3\in\calN$ with $n_1\neq n_2\neq n_3\neq n_1$,
\eqref{eq.triple product simplex} gives that a subset $\calK$ of $\calN$ belongs to $\bin(\bfPhi)$ if and only if $\#(\calK)=S+1$, namely if and only if $\calK=\calN$.
Since $\bin(\bfPhi)=\set{\calN}$ is nonempty,
Theorem~\ref{theorem.binders of DTETFs} guarantees that it forms a BIBD.
In fact, it is a degenerate BIBD with $V=K=S+1$ and $\Lambda=1$ whose $\calB\times\calN$ incidence matrix $\bfX$ consists of a single all-ones row.
Since $\ip{\bfphi_{n_1}}{\bfphi_{n_2}}=-1$ when $n_1\neq n_2$,
we can further choose $z_{b,n}=1$ for all $b$ and $n$ when representing its dual $\bfPsi$ as in~\eqref{eq.phased incidence}, yielding an ``all-ones row vector.''

In contrast, the binder of its dual $\bfPsi$ consists of all two-element subsets of $\calN$: when ``$(\bfphi_n)_{n\in\calN}$'' is $(\bfpsi_n)_{n\in\calN}$,
``$N$'', ``$D$'' and ``$S$'' become $S+1$, $1$ and $1$, respectively,
and so~\eqref{eq.triple product simplex} gives that a subset $\calK$ of $\calN$ belongs to $\bin(\bfPsi)$ if and only if $\#(\calK)=2$ and $-1=\ip{\bfpsi_{n_1}}{\bfpsi_{n_2}}\ip{\bfpsi_{n_2}}{\bfpsi_{n_3}}\ip{\bfpsi_{n_3}}{\bfpsi_{n_1}}=1$ for all $n_1,n_2,n_3\in\calK$ with $n_1\neq n_2\neq n_3\neq n_1$,
and though this second property is impossible when $\#(\calK)\geq 3$,
it is vacuously true when $\#(\calK)=2$.
Since $\bin(\bfPsi)=\bigset{\set{n_1,n_2}: n_1,n_2\in\calN, n_1\neq n_2}$ is nonempty, Theorem~\ref{theorem.binders of DTETFs} guarantees that it forms a BIBD.
In fact, it is clearly a BIBD with $V=N$, $K=2$ and $\Lambda=1$,
implying via~\eqref{eq.BIBD parameter relations} that $R=S$ and $B=\binom{S+1}{2}=\tfrac12 S(S+1)$.
As such, applying~\eqref{eq.phased incidence} when ``$\bfPhi$'' is $\bfPsi$ yields a phased incidence representation of its dual $\bfPhi$.
Here since $\ip{\bfpsi_{n_1}}{\bfpsi_{n_2}}=1$ when $n_1\neq n_2$,
we can simply choose ``$z_{b,n_1}$'' and ``$z_{b,n_2}$'' to be $1$ and $-1$, respectively, when ``$\calK_b$'' is $\set{n_1,n_2}$.
This yields a nonobvious result: up to equivalence, any $(S+1)$-vector regular simplex arises as the columns of a $\tfrac12 S(S+1)\times(S+1)$ matrix,
each of whose rows is supported on two opposite nonzero entries,
and whose pairs of columns have exactly one nonzero entry in common.
As guaranteed by Theorem~\ref{theorem.binders of DTETFs},
every member of $\bin(\bfPsi)$ is an oval of $\bin(\bfPhi)$, and vice versa.
Here in fact, the sole member of $\bin(\bfPhi)$ intersects every member of $\bin(\bfPsi)$ in exactly two vertices.

For example, when $S=3$, we can form a real such simplex $\bfPhi$ by removing an all-ones row---the synthesis map of its dual $\bfPsi$---from a standard Hadamard matrix of size $4$:
\begin{equation*}
\setlength{\arraycolsep}{0pt}
\begin{array}{ccccc}
\bfPsi
&=&
\begin{tiny}\left[\begin{array}{rrrr}+&+&+&+\end{array}\right]\end{tiny}
&\cong&
\begin{tiny}\left[\begin{array}{rrrr}+&+&+&+\end{array}\right]\end{tiny},\\
\bfPhi
&=&
\begin{tiny}\left[\begin{array}{rrrr}
+&+&-&-\\
+&-&+&-\\
+&-&-&+
\end{array}\right]\end{tiny}
&\cong&
\begin{tiny}\left[\begin{array}{rrrr}
+&-&0&0\\
0&0&+&-\\
+&0&-&0\\
0&+&0&-\\
+&0&0&-\\
0&+&-&0\\
\end{array}\right]\end{tiny}.
\end{array}
\end{equation*}
Here, ``$+$'' and ``$-$'' denote $1$ and $-1$, respectively, and we in general write ``$\bfPhi_1\cong\bfPhi_2$'' when $\bfPhi_1:\calV\rightarrow\calU_1$ and $\bfPhi_2:\calV\rightarrow\calU_2$ are linear where $\calV$, $\calU_1$ and $\calU_2$ are finite-dimensional Hilbert spaces and there exists a unitary transformation $\bfU$ from the image of $\bfPhi_1$ to that of $\bfPhi_2$ such that $\bfPhi_2=\bfU\bfPhi_1$.
\end{example}

%%%%%%%%%%%%%%%%%%%%%%%%%%%%%%%%%%%%%%%%%%%%%%%%%%%%%%%%%%%%%%%%
\section{The binders of symplectic equiangular tight frames and their duals}
%%%%%%%%%%%%%%%%%%%%%%%%%%%%%%%%%%%%%%%%%%%%%%%%%%%%%%%%%%%%%%%%

In the previous section, we discussed when a finite sequence $\bfPhi=(\bfphi_n)_{n\in\calN}$ of vectors in a Hilbert space is a DTETF (Lemma~\ref{lemma.DT implies ETF}), and showed that the binder of any such $\bfPhi$ is either empty or forms a BIBD (Theorem~\ref{theorem.binders of DTETFs}).
To date, we have not found a general method for determining when the binder of a DTETF is nonempty, nor if so, for computing the ``$\Lambda$'' parameter of the resulting BIBD.
That said, \cite{IversonM24,DempwolffK23} provides a complete classification of~DTETFs,
and analyzing each class of them individually seems like a surmountable task for the community.
In support of such a program,
we devote this section and the next to performing such analyses for specific classes of DTETFs that arise from symplectic and quadratic forms on finite vector spaces, respectively.

As elegantly summarized in Example~5.10 of~\cite{IversonM22},
symplectic ETFs have been independently discovered several times, and---up to equivalence---in several guises.
Some facts about the binders of their duals are already known:
combining Theorem~5.1 of~\cite{FickusJMPW19} and Theorem~5.2 of~\cite{FickusJKM18} gives that any such binder contains a (Desarguesian) affine plane.
Theorem~5.6 of~\cite{BodmannK20} gives that it is in fact an affine plane when the underlying finite vector space has dimension~$2$.
When this space instead has dimension~$4$,
Theorem~5.8 of~\cite{BodmannK20} gives that it forms a more exotic BIBD.
Conjecture~5.12 of~\cite{BodmannK20} moreover predicts that the binder of the dual of any symplectic ETF forms a BIBD with a particular $\Lambda$ parameter.
The approach of~\cite{BodmannK20} relied on their novel use of the triple-product-based characterization of regular simplices~\eqref{eq.triple product simplex} as a symbolic tool, rather than simply a numerical aid~\cite{FickusJKM18}.
In this section, we combine that trick with both Theorem~\ref{theorem.binders of DTETFs} and a clean expression for the Gram matrix of the symplectic ETF to prove that Conjecture~5.12 of~\cite{BodmannK20} is partially true---the binder of its dual does indeed form a BIBD---but is also partially false,
as its prediction for the value of $\Lambda$ is incorrect in general.
We in fact show that, in general, this binder consists of all affine Lagrangian subspaces of the underlying finite symplectic space.
We begin here by reviewing symplectic ETFs.

\subsection{Symplectic forms}

A function $\rmB:\calV\times\calV\rightarrow\bbF$ is a \textit{bilinear form} on a vector space $\calV$ if,
for any $\bfv_1,\bfv_2\in\calV$, the mappings $\bfv\mapsto\rmB(\bfv_1,\bfv)$ and $\bfv\mapsto\rmB(\bfv,\bfv_2)$ are linear.
Such a function $\rmB$ is \textit{symmetric} if $\rmB(\bfv_1,\bfv_2)=\rmB(\bfv_2,\bfv_1)$ for all $\bfv_1,\bfv_2\in\calV$, and is \textit{alternating} if $\rmB(\bfv,\bfv)=0$ for all $\bfv\in\calV$.
For any $\bfv_1,\bfv_2\in\calV$,
the latter implies $0=\rmB(\bfv_1+\bfv_2,\bfv_1+\bfv_2)=\rmB(\bfv_1,\bfv_2)+\rmB(\bfv_2,\bfv_1)$,
that is, $\rmB(\bfv_2,\bfv_1)=-\rmB(\bfv_1,\bfv_2)$.
A \textit{symplectic form} on $\calV$ is an alternating bilinear form $\rmB$ that is \textit{nondegenerate} in the sense that for any $\bfv_1\neq0$ there exists $\bfv_2\in\calV$ such that $\rmB(\bfv_1,\bfv_2)\neq0$.
When this occurs, we can scale such a vector $\bfv_2$ so as to assume without loss of generality that $\rmB(\bfv_1,\bfv_2)=1$.

For the remainder of this section, let $\rmB$ be a symplectic form on a finite-dimensional vector space $\calV$ over some field $\bbF$.
Vectors $\bfv_1,\bfv_2\in\calV$ are \textit{orthogonal} if $\rmB(\bfv_1,\bfv_2)=0$.
Since $\rmB$ is alternating, this occurs if and only if $\rmB(\bfv_2,\bfv_1)=0$, and so we need not distinguish left- and right-orthogonality.
The \textit{orthogonal complement} of any subset $\calS$ of $\calV$ is $\calS^\perp:=\set{\bfv_2\in\calV: \rmB(\bfv_1,\bfv_2)=0,\ \forall\,\bfv_1\in\calS}$.
The algebraic properties of symplectic forms differ from those of real inner products in some key respects: for example,
since every vector is orthogonal to itself,
orthogonality does not guarantee linear independence,
and orthonormal bases do not exist.
That said, some nontrivial properties of real inner products generalize to the symplectic setting.
We now briefly review those that we need here.

For example,
we claim that $\dim(\calV_0^{})+\dim(\calV_0^\perp)=\dim(\calV)$
for any subspace $\calV_0$ of $\calV$.
To see this,
extend a basis $(\bfv_m)_{m=1}^{M_0}$ for $\calV_0$ to a basis $(\bfv_m)_{m=1}^{M}$ for $\calV$.
The nondegeneracy of $\rmB$ implies that
$\bfL:\calV\rightarrow\bbF^M$, $(\bfL\bfv)(m):=\rmB(\bfv_m,\bfv)$
has $\ker(\bfL)=\set{\bfzero}$.
By the rank-nullity theorem, $\bfL$ is invertible.
Composing $\bfL$ with the truncated identity matrix $\bfA\in\bbF^{M_0\times M}$, $\bfA(m_1,m_2):=\bfdelta_{m_2}(m_1)$ yields
$\bfA\bfL:\calV\rightarrow\bbF^{M_0}$, $(\bfA\bfL)(m)=\rmB(\bfv_m,\bfv)$ for all $m\in[M_0]$.
Thus,
$\dim(\calV)-\dim(\calV_0^\perp)
=\dim(\calV)-\dim(\ker(\bfA\bfL))
=\rank(\bfA\bfL)
=\rank(\bfA)
=M_0
=\dim(\calV_0)$,
giving the claim.
This in turn implies that $(\calS^\perp)^\perp=\Span(\calS)$ for any subset $\calS$ of $\calV$:
letting $\calV_0=\Span(\calS)$, we have \smash{$\calV_0^\perp=\calS^\perp$} where clearly \smash{$\calV_0\subseteq(\calV_0^\perp)^\perp$} and
$\dim(\calV_0)
=\dim(\calV)-\dim(\calV_0^\perp)
=\dim((\calV_0^\perp)^\perp)$.
We moreover claim that any linear functional $\bfL:\calV\rightarrow\bbF$ has a Riesz representation,
that is, there exists $\bfv_\bfL\in\calV$ such that $\bfL(\bfv)=\rmB(\bfv_\bfL,\bfv)$ for all $\bfv\in\calV$.
Indeed, take $\bfv_\bfL=\bfzero$ when $\bfL=\bfzero$.
When instead $\bfL\neq\bfzero$,
$\ker(\bfL)^\perp$ has dimension one and so contains some $\bfv_1\neq\bfzero$.
Take $\bfv_2$ such that $\rmB(\bfv_1,\bfv_2)=1$ and so $\bfv_2\notin\calV_0$.
For any $\bfv\in\calV$,
there thus exists $c\in\bbF$ and $\bfv_0\in\calV_0$ such that $\bfv=c\bfv_2+\bfv_0$.
Here, $\rmB(\bfv_1,\bfv)=c1+0=c$ and so $\bfL\bfv=c\bfL\bfv_2+\bfzero=\rmB(\bfv_\bfL,\bfv)$ where $\bfv_\bfL:=(\bfL\bfv_2)\bfv_1$.
Having the claim,
we moreover note that since $\rmB$ is nondegenerate, this $\bfv_\bfL$ is unique.

A pair $(\bfv_1,\bfv_2)$ of vectors in $\calV$ is \textit{symplectic} if $\rmB(\bfv_1,\bfv_2)=1$.
Such a pair exists if and only if $\calV\neq\set{\bfzero}$.
In fact, as noted above, for any nonzero $\bfv_1\in\calV$ there exists $\bfv_2\in\calV$ such that $(\bfv_1,\bfv_2)$ is symplectic.
For any symplectic pair $(\bfv_1,\bfv_2)$,
$\Span\set{\bfv_1,\bfv_2}$ has dimension $2$ since otherwise $\bfv_1$ and $\bfv_2$ are collinear and $\rmB(\bfv_1,\bfv_2)=0$, a contradiction.
Thus, $\calV_0:=\set{\bfv_1,\bfv_2}^\perp$ has $\dim(\calV_0)=M-2$ where $M=\dim(\calV)$.
We moreover claim that any $\bfv\in\calV$ can be uniquely written as $\bfv=x_1\bfv_1+x_2\bfv_2+\bfv_0$ where $x_1,x_2\in\bbF$ and $\bfv_0\in\calV_0$.
Indeed, taking the $\rmB$-form of this equation with $\bfv_1$ and $\bfv_2$ implies that necessarily $x_1=\rmB(\bfv,\bfv_2)$ and $x_2=\rmB(\bfv_1,\bfv)$,
and so $\bfv_0$ is necessarily $\bfv-x_1\bfv_1-x_2\bfv_2$.
Conversely, choosing $x_1$, $x_2$ and $\bfv_0$ in this way gives $\bfv=x_1\bfv_1+x_2\bfv_2+\bfv_0$ where $\bfv_0\in\calV_0$, yielding the claim.
In particular, $\calV_0^{}\cap\calV_0^\perp=\set{\bfzero}$ and $\calV$ is the internal direct sum of $\calV_0^{}$ and $\calV_0^\perp=\Span\set{\bfv_1,\bfv_2}$.
The restriction $\rmB_0:\calV_0\times\calV_0\rightarrow\bbF$ of $\rmB$ to $\calV_0$ is itself symplectic:
it is immediately bilinear and alternating,
and is nondegenerate since if $\bfv_0\in\calV_0$ satisfies $\rmB(\bfv_0,\bfv)=0$ for all $\bfv\in\calV_0$ then $\bfv\in\calV_0^\perp$, implying
$\bfv_0\in\calV_0^{}\cap\calV_0^\perp=\set{\bfzero}$,
and so $\bfv_0=\bfzero$.
Provided $\calV_0\neq\set{\bfzero}$,
this allows one to repeat this procedure with ``$\calV$'' being $\calV_0$,
that is, to choose a symplectic pair $(\bfv_3,\bfv_4)$ from $\calV_0$,
which is a space of dimension $M_0=M-2$.
Since we have assumed here that the dimension $M$ of $\calV$ is finite,
proceeding in this manner by induction reveals that $M$ is necessarily even---since no alternating bilinear form on a one-dimensional space is nondegenerate---and produces a \textit{symplectic basis} for $\calV$, namely a finite sequence $(\bfv_m)_{m=1}^{2J}$ of $2J$ vectors in the $2J$-dimensional space $\calV$ such that
\begin{equation}
\label{eq.symplectic basis}
\rmB(\bfv_{m_1},\bfv_{m_2})
=\left\{\begin{array}{rl}
 1,&\ 2\mid m_2=m_1+1,\\
-1,&\ 2\mid m_1=m_2+1,\\
 0,&\ \text{else}.
\end{array}\right.
\end{equation}
Any such $(\bfv_m)_{m=1}^{2J}$ is a basis for $\calV$:
letting $\bfV:\bbF^{2J}\rightarrow\calV$,
$\bfV\bfx:=\sum_{m=1}^{2J}\bfx(m)\bfv_m$ be its synthesis map,
we have $\bfV\bfx=\bfv$ only if
$\bfx(2j-1)=\rmB(\bfv,\bfv_{2j})$ and $\bfx(2j)=\rmB(\bfv_{2j-1},\bfv)$ for all $j\in[J]$,
implying $\bfV$ is invertible.
Moreover, writing any $\bfv_1,\bfv_2\in\calV$ as $\bfv_1=\bfV\bfx_1$ and $\bfv_2=\bfV\bfx_2$ for some $\bfx_1,\bfx_2\in\bbF^{2J}$,
distributing $\rmB$ over linear combinations gives $\rmB(\bfv_1,\bfv_2)=\rmB_J(\bfx_1,\bfx_2)$ where
\begin{equation}
\label{eq.canonical symplectic form}
\rmB_J:\bbF^{2J}\times\bbF^{2J}\rightarrow\bbF,
\qquad
\rmB_J(\bfx_1,\bfx_2)=\sum_{\smash{j=1}}^{\smash{J}}[\bfx_1(2j-1)\bfx_2(2j)-\bfx_1(2j)\bfx_2(2j-1)],
\end{equation}
is the \textit{canonical} symplectic form on $\bbF^{2J}$.
Conversely,
letting $\bfV$ be the synthesis map of any basis $(\bfv_m)_{m=1}^{2J}$ for any vector space over $\bbF$ of dimension $2J$,
it is straightforward to verify that $\rmB(\bfv_1,\bfv_2):=\rmB_J(\bfV^{-1}\bfv_1,\bfV^{-1}\bfv_2)$ defines a symplectic form on $\calV$.
Up to isomorphism, $\rmB_J$ is thus the unique symplectic form on a vector space of dimension $2J$.

\subsection{Symplectic equiangular tight frames}

As we now explain, in the special case where $\bbF$ is the finite field $\bbF_P=\set{0,\dotsc,P-1}$ of prime order $P$, the symplectic form $\rmB$ on a vector space $\calV$ of dimension $2J$ over $\bbF_P$ gives rise to a DTETF whose vectors are indexed by $\calV$,
which has order $N=P^{2J}$.
When $\rmB$ is more generally a symplectic form on a finite-dimensional vector space $\calV$ over a finite field $\bbF_Q$ whose order $Q$ is a power of this prime $P$,
we claim that the requisite $\bbF_P$-valued form can be obtained by composing $\rmB$ with any nontrivial linear functional $\bfL:\bbF_Q\rightarrow\bbF_P$.
Indeed, it is clear that $\calV$ is a finite-dimensional space over $\bbF_P$ whose dimension is the product of that of $\calV$ over $\bbF_Q$ with that of $\bbF_Q$ over $\bbF_P$,
and moreover that $\bfL\circ\rmB:\calV\times\calV\rightarrow\bbF_P$ is an alternating bilinear form.
To see that $\bfL\circ\rmB$ is nondegenerate,
for any nonzero $\bfv_1\in\calV$, take $\bfv_2\in\calV$ such that $\rmB(\bfv_1,\bfv_2)=1$ and any $x\in\bbF_Q$ such that $\bfL(x)\neq0$,
and note that
$(\bfL\circ\rmB)(\bfv_1,x\bfv_2)
=\bfL(\rmB(\bfv_1,x\bfv_2))
=\bfL(x\rmB(\bfv_1,\bfv_2))
=\bfL(x)\neq0$,
yielding the claim.
For example, one can always choose $\bfL$ to be the field trace,
that is, $\bfL(x)=\tr(x):=\sum_{i=0}^{I-1}x^{P^i}$ where $Q=P^I$.

The \textit{character table} of an $\bbF_P$-valued symplectic form $\rmB$ is the matrix $\bfGamma\in\bbF^{\calV\times\calV}$, $\bfGamma(\bfv_1,\bfv_2):=\exp(\frac{2\pi\rmi}{P}\rmB(\bfv_1,\bfv_2))$.
Here, $\bbF=\bbC$ unless $P=2$ and we can choose $\bbF=\bbR$.
It can be shown that $\bfGamma$ is the classical character table of the (finite abelian) additive group of $\calV$ with respect to a particular enumeration of its Pontryagin dual.
We claim that $\bfGamma$ is a possibly-complex Hadamard matrix.
Indeed,
$(\bfGamma\bfGamma^*)(\bfv_1,\bfv_2)
=\sum_{\bfv\in\calV}\bfGamma(\bfv_1,\bfv)\overline{\bfGamma(\bfv_2,\bfv)}
=\sum_{\bfv\in\calV}\exp(\tfrac{2\pi\rmi}{P}\rmB(\bfv_1-\bfv_2,\bfv))$
for any $\bfv_1,\bfv_2\in\calV$.
When $\bfv_1\neq \bfv_2$,
letting $\bfv_3:=\bfv_2-\bfv_1\neq0$ and taking $\bfv_4\in\calV$ such that $\rmB(\bfv_3,\bfv_4)=1$,
\begin{equation*}
\bigbracket{1-\exp(\tfrac{2\pi\rmi}{P})}(\bfGamma\bfGamma^*)(\bfv_1,\bfv_2)
=\sum_{\bfv\in\calV}\exp(\tfrac{2\pi\rmi}{P}\rmB(\bfv_3,\bfv))
-\sum_{\bfv\in\calV}\exp(\tfrac{2\pi\rmi}{P}\rmB(\bfv_3,\bfv+\bfv_4))
=0,
\end{equation*}
implying $(\bfGamma\bfGamma^*)(\bfv_1,\bfv_2)=0$.
Thus, $\bfGamma\bfGamma^*=N\brI$, as claimed.
Since $\rmB$ is alternating,
$\bfGamma(\bfv,\bfv)=\exp(\frac{2\pi\rmi}{P}0)=1$ for all $\bfv\in\calV$.
Since $\rmB(\bfv_2,\bfv_1)=-\rmB(\bfv_1,\bfv_2)$ for all $\bfv_1,\bfv_2\in\calV$,
we moreover have $\bfGamma^*=\bfGamma$.
Thus, $\bfGamma^2=\bfGamma\bfGamma^*=P^{2J}\brI$, implying that every eigenvalue of $\bfGamma$ is either $P^J$ or $-P^J$.
In fact, the multiplicity $P^J$ is $D=\tfrac12P^J(P^J+1)$ since $P^{2J}=\Tr(\bfGamma)=DP^J+(N-D)(-P^J)=P^J(2D-P^{2J})$.
Altogether,
$\bfGamma$ is a self-adjoint possibly-complex Hadamard matrix of size $N=P^{2J}$ whose diagonal entries are $1$ and whose spectrum consists of $P^J$ and $-P^J$ with multiplicities
$D=\tfrac12 P^J(P^J+1)$ and $N-D=\tfrac12 P^J(P^J-1)$, respectively.
As such, $P^J\brI+\bfGamma$ and $P^J\brI-\bfGamma$ are the Gram matrices of an ETF $\bfPhi$ and its dual $\bfPsi$, respectively.
We refer to these ETFs as follows:

\begin{definition}
\label{def.symplectic ETF}
Let $P$ be prime, $\rmB$ be a symplectic form on a vector space $\calV$ over $\bbF_P$ of dimension $2J$,
and $\bfGamma\in\bbF^{\calV\times\calV}$, $\bfGamma(\bfv_1,\bfv_2):=\exp(\frac{2\pi\rmi}{P}\rmB(\bfv_1,\bfv_2))$.
Here, $\bbF$ is $\bbR$ if $P=2$ and $\bbF=\bbC$ if $P>2$.
A finite sequence $\bfPhi=(\bfphi_{\bfv})_{\bfv\in\calV}$ of vectors in a Hilbert space over $\bbF$ is a corresponding \textit{symplectic ETF} if
\begin{equation*}
\ip{\bfphi_{\bfv_1}}{\bfphi_{\bfv_2}}
=\left\{\begin{array}{cl}
P^J+1,&\ \bfv_1=\bfv_2,\\
\exp(\frac{2\pi\rmi}{P}\rmB(\bfv_1,\bfv_2)),&\ \bfv_1\neq \bfv_2,
 \end{array}\right.
\end{equation*}
for any $\bfv_1,\bfv_2\in\calV$.
Such a $\bfPhi$ exists for any such $\rmB$,
and any such $\bfPhi$ is an $N=P^{2J}$-vector ETF for a space of dimension
$D=\tfrac12P^J(P^J+1)$.
A finite sequence $\bfPsi=(\bfpsi_{\bfv})_{\bfv\in\calV}$ is thus its dual if
\begin{equation*}
\ip{\bfpsi_{\bfv_1}}{\bfpsi_{\bfv_2}}
=\left\{\begin{array}{cl}
P^J-1,&\ \bfv_1=\bfv_2,\\
-\exp(\frac{2\pi\rmi}{P}\rmB(\bfv_1,\bfv_2)),&\ \bfv_1\neq \bfv_2,
\end{array}\right.
\end{equation*}
for any $\bfv_1,\bfv_2\in\calV$.
Any such $\bfPsi$ is an $N$-vector ETF for a space of dimension $N-D=\tfrac12P^J(P^J-1)$.
\end{definition}

\begin{example}
\label{example.symplectic ETFs}
The canonical symplectic form~\eqref{eq.canonical symplectic form} on $\bbF_2^2$ is $\rmB_1((x_1,x_2),(y_1,y_2))=x_1y_2+x_2y_1$.
(Here since the underlying field has characteristic $2$,
we express the subtractions in~\eqref{eq.canonical symplectic form} as additions.)
The corresponding character table $\bfGamma_1$
is a symmetric Hadamard matrix with ones along its diagonal:
\begin{equation*}
\setlength{\arraycolsep}{0pt}
\bfGamma_1((x_1,x_2),(y_1,y_2))=(-1)^{x_1y_2+x_2y_1},
\quad\text{i.e.,}\quad
\bfGamma_1=\begin{tiny}\left[\begin{array}{cccc}
+&+&+&+\\
+&+&-&-\\
+&-&+&-\\
+&-&-&+
\end{array}\right]\ \begin{array}{c}(00)\\(01)\\(10)\\(11)\end{array}\end{tiny}\,.
\end{equation*}
(Here, the elements of the index set $\bbF_2^2$ of the rows and columns of $\bfGamma_1$ are ordered as
$00$, $01$, $10$, $11$.)
Letting $P=2$ and $J=1$ in Definition~\ref{def.symplectic ETF} gives that
the Gram matrices of the corresponding symplectic ETF and its dual are
\begin{equation*}
\setlength{\arraycolsep}{0pt}
2\brI+\bfGamma_1
=\begin{tiny}\left[\begin{array}{cccc}
3&+&+&+\\
+&3&-&-\\
+&-&3&-\\
+&-&-&3
\end{array}\right]\end{tiny},
\qquad
2\brI-\bfGamma_1
=\begin{tiny}\left[\begin{array}{cccc}
+&-&-&-\\
-&+&+&+\\
-&+&+&+\\
-&+&+&+
\end{array}\right]\end{tiny}.
\end{equation*}
These are $N=P^{2J}=4$-vector ETFs for spaces of dimension
$D=\frac12P^J(P^J+1)=3$ and $N-D=1$, respectively,
that are indexed by $\calV=\bbF_2^2$.
Negating the first vector in these ETFs---and so negating all nondiagonal entries in the first row and column of these two Gram matrices---reveals them to be equivalent to the regular simplex and its dual, respectively, that were considered in Example~\ref{example.binder of simplex}.
For more interesting ETFs, consider the analogous form on $\bbF_2^4$, namely
$\rmB_2((x_1,x_2,x_3,x_4),(y_1,y_2,y_3,y_4))=(x_1y_2+x_2y_1)+(x_3y_4+x_4y_3)$.
Its character table $\bfGamma_2$ has
\begin{align*}
\bfGamma_2((x_1,x_2,x_3,x_4),(y_1,y_2,y_3,y_4))
&=(-1)^{x_1y_2+x_2y_1}(-1)^{x_3y_4+x_4y_3}\\
&=\bfGamma_1((x_1,x_2),(y_1,y_2))\bfGamma_1((x_3,x_4),(y_3,y_4)),
\end{align*}
and so is naturally identified with the tensor product of $\bfGamma_1$ with itself:
\begin{equation}
\label{eq.Gamma 16x16}
\setlength{\arraycolsep}{0pt}
\bfGamma_2=\bfGamma_1\otimes\bfGamma_1
=\begin{tiny}\left[\begin{array}{cccc}
+&+&+&+\\
+&+&-&-\\
+&-&+&-\\
+&-&-&+
\end{array}\right]\end{tiny}
\otimes
\begin{tiny}\left[\begin{array}{cccc}
+&+&+&+\\
+&+&-&-\\
+&-&+&-\\
+&-&-&+
\end{array}\right]\end{tiny}
=\begin{tiny}\left[\begin{array}{cccccccccccccccc}
+&+&+&+&+&+&+&+&+&+&+&+&+&+&+&+\\
+&+&-&-&+&+&-&-&+&+&-&-&+&+&-&-\\
+&-&+&-&+&-&+&-&+&-&+&-&+&-&+&-\\
+&-&-&+&+&-&-&+&+&-&-&+&+&-&-&+\\
+&+&+&+&+&+&+&+&-&-&-&-&-&-&-&-\\
+&+&-&-&+&+&-&-&-&-&+&+&-&-&+&+\\
+&-&+&-&+&-&+&-&-&+&-&+&-&+&-&+\\
+&-&-&+&+&-&-&+&-&+&+&-&-&+&+&-\\
+&+&+&+&-&-&-&-&+&+&+&+&-&-&-&-\\
+&+&-&-&-&-&+&+&+&+&-&-&-&-&+&+\\
+&-&+&-&-&+&-&+&+&-&+&-&-&+&-&+\\
+&-&-&+&-&+&+&-&+&-&-&+&-&+&+&-\\
+&+&+&+&-&-&-&-&-&-&-&-&+&+&+&+\\
+&+&-&-&-&-&+&+&-&-&+&+&+&+&-&-\\
+&-&+&-&-&+&-&+&-&+&-&+&+&-&+&-\\
+&-&-&+&-&+&+&-&-&+&+&-&+&-&-&+
\end{array}\right]
\ \begin{array}{c}
(0000)\\
(0001)\\
(0010)\\
(0011)\\
(0100)\\
(0101)\\
(0110)\\
(0111)\\
(1000)\\
(1001)\\
(1010)\\
(1011)\\
(1100)\\
(1101)\\
(1110)\\
(1111)
\end{array}
\end{tiny}\ .
\end{equation}
Being a symmetric Hadamard matrix of size $N=P^{2J}=2^4=16$ and trace $N=16$,
$\bfGamma_2$ has eigenvalues $P^J=2^2=4$ and $-P^J=-4$ with multiplicities
$D=\frac12P^J(P^J+1)=\frac12 2^2(2^2+1)=10$ and
$N-D=\frac12P^J(P^J-1)=\frac12 2^2(2^2-1)=6$, respectively,
implying $4\brI+\bfGamma_2$ and $4\brI-\bfGamma_2$ are Gram matrices of $16$-vector ETFs for spaces of dimension $10$ and $6$, respectively.
Similarly,
$8\brI+\bfGamma_3$ and $8\brI-\bfGamma_3$ are Gram matrices of $64$-vector ETFs for spaces of dimension $36$ and $28$, respectively, where $\bfGamma_3=\bfGamma_1\otimes\bfGamma_1\otimes\bfGamma_1$.
Later on in this section, we calculate the binders of all such ETFs.
\end{example}

Any symplectic ETF $\bfPhi$ and its dual $\bfPsi$ are DTETFs~\cite{IversonM22}.
To explain, in general, if $\rmB$ is any symplectic form on a finite-dimensional vector space $\calV$ over some field $\bbF$,
the corresponding \textit{symplectic group} is the set $\Sp(\rmB)$ of all linear operators $\bfA:\calV\rightarrow\calV$ that preserve $\rmB$, that is,
for which $\rmB(\bfA\bfv_1,\bfA\bfv_2)=\rmB(\bfv_1,\bfv_2)$ for all $\bfv_1,\bfv_2\in\calV$.
In light of~\eqref{eq.symplectic basis},
any such $\bfA$ maps any symplectic basis of $\calV$ to another such basis.
Since any such $\calV$ has a symplectic basis,
and since any symplectic basis for $\calV$ is a basis for it,
any such $\bfA$ is necessarily invertible.
Thus, $\Sp(\rmB)$ is necessarily a subset of the general linear group $\GL(\calV)$ of all invertible linear operators on $\calV$.
It is moreover straightforward to show that it is a subgroup of $\GL(\calV)$.
Thus,
\begin{equation}
\label{eq.symplectic group}
\Sp(\rmB)=\set{\bfA\in\GL(\calV) : \rmB(\bfA\bfv_1,\bfA\bfv_2)=\rmB(\bfv_1,\bfv_2),\ \forall\,\bfv_1,\bfv_2\in\calV}
\leq\GL(\calV).
\end{equation}
Conversely, any linear operator $\bfA:\calV\rightarrow\calV$
that maps even a single symplectic basis $(\bfv_m)_{m=1}^{2J}$ for $\calV$ to another such basis is necessarily a member of $\Sp(\rmB)$ since any $\bfu_1,\bfu_2\in\calV$ can be written as linear combinations of $(\bfv_m)_{m=1}^{2J}$ where $\rmB(\bfA\bfv_{m_1},\bfA\bfv_{m_2})=\rmB(\bfv_{m_1},\bfv_{m_2})$ for all $m_1,m_2\in\calM$.
From this, it follows that the natural action of $\Sp(\rmB)$ on $\calV\backslash\set{\bfzero}$ is transitive,
that is, that for any nonzero $\bfv_1,\bfv_2\in\calV$, there exists $\bfA\in\Sp(\rmB)$ such that $\bfA\bfv_1=\bfv_2$:
since any nonzero vector can serve as the first member of a symplectic basis for $\calV$,
extending $\bfv_1$ to one such basis and $\bfv_2$ to another,
the unique linear operator $\bfA:\calV\rightarrow\calV$ that maps the former onto the latter thus belongs to $\Sp(\rmB)$.

We now more broadly consider affine symplectic mappings, that is,
bijections from $\calV$ onto itself of the form
$\bfv'\mapsto(\bfv,\bfA)\cdot\bfv':=\bfv+\bfA\bfv'$ where $\bfv\in\calV$ and $\bfA\in\Sp(\rmB)$.
Composing two such mappings yields another:
$(\bfv_1,\bfA_1)\cdot[(\bfv_2,\bfA_2)\cdot\bfv']
=\bfv_1+\bfA_1(\bfv_2+\bfA_2\bfv')
%=(\bfv_1+\bfA_1\bfv_2)+(\bfA_1\bfA_2)
=(\bfv_1+\bfA_1\bfv_2,\bfA_1\bfA_2)\cdot\bfv'$.
It quickly follows that $(\bfv,\bfA)\cdot\bfv':=\bfv+\bfA\bfv'$ defines an action on $\calV$ by the semidirect product $\calV\rtimes\Sp(\rmB)$ under the group operation
$(\bfv_1,\bfA_1)(\bfv_2,\bfA_2):=(\bfv_1+\bfA_1\bfv_2,\bfA_1\bfA_2)$.
This action is doubly transitive since, for any $\bfv_1,\bfv_2,\bfv_3,\bfv_4\in\calV$ with $\bfv_1\neq\bfv_2$ and $\bfv_3\neq\bfv_4$,
taking $\bfA\in\Sp(\rmB)$ such that $\bfA(\bfv_1-\bfv_2)=(\bfv_3-\bfv_4)$ and $\bfv=\bfv_3-\bfA\bfv_1$,
we have
$(\bfv,\bfA)\cdot\bfv_1
=\bfv+\bfA\bfv_1
=\bfv_3$
and
\begin{equation*}
(\bfv,\bfA)\cdot\bfv_2
=\bfv+\bfA\bfv_2
=\bfv_3-\bfA\bfv_1+\bfA\bfv_2
=\bfv_3-\bfA(\bfv_1-\bfv_2)
=\bfv_3-(\bfv_3-\bfv_4)
=\bfv_4.
\end{equation*}
Such mappings have a simple effect on symplectic forms:
for any $\bfv_1,\bfv_2\in\calV$,
\begin{align*}
\rmB(\bfv+\bfA\bfv_1,\bfv+\bfA\bfv_2)
&=\rmB(\bfA\bfv_1,\bfv)+\rmB(\bfv,\bfA\bfv_2)+\rmB(\bfA\bfv_1,\bfA\bfv_2)\\
&=-\rmB(\bfv,\bfA\bfv_1)+\rmB(\bfv,\bfA\bfv_2)+\rmB(\bfv_1,\bfv_2).
\end{align*}
When regarded as permutations of $\calV$, these mappings thus belong to the symmetry group of the symplectic ETF (Definition~\ref{def.symplectic ETF}):
for any $\bfv\in\calV$ and $\bfA\in\Sp(\rmB)$,
$(z_{\bfv'})_{\bfv'\in\calV}$, $z_{\bfv'}:=\exp(\frac{2\pi\rmi}{P}\rmB(\bfv,\bfA\bfv'))$
is a finite sequence of unimodular scalars, and the following holds for any $\bfv_1,\bfv_2\in\calV$, $\bfv_1\neq\bfv_2$:
\begin{align*}
\ip{\bfphi_{\bfv+\bfA\bfv_1}}{\bfphi_{\bfv+\bfA\bfv_2}}
&=\exp(\tfrac{2\pi\rmi}{P}\rmB(\bfv+\bfA\bfv_1,\bfv+\bfA\bfv_2))\\
&=\exp(\tfrac{2\pi\rmi}{P}
[-\rmB(\bfv,\bfA\bfv_1)+\rmB(\bfv,\bfA\bfv_2)+\rmB(\bfv_1,\bfv_2)])\\
&=\overline{z_{\bfv_1}}z_{\bfv_2}\ip{\bfphi_{\bfv_1}}{\bfphi_{\bfv_2}}.
\end{align*}
One may more strongly show that these affine symplectic mappings are the only members of the symmetry group of a symplectic ETF $\bfPhi$, that is, $\bfPhi\cong\calV\rtimes\Sp(\rmB)$~\cite{IversonM22}.
Regardless, by the argument above, any symplectic ETF is a DTETF,
and so its dual is as well.
By Theorem~\ref{theorem.binders of DTETFs},
the binders of these DTETFs are thus either empty or form BIBDs.
We now calculate these binders explicitly.

\subsection{Calculating the binders of symplectic equiangular tight frames and their duals}

The process of calculating the binder of a symplectic ETF has much in common with that for its dual:

\begin{lemma}
\label{lemma.binder of symplectic}
For any prime $P$ and positive integer $J$,
let $\bin(\bfPhi)$ be the binder of a $P^{2J}$-vector symplectic ETF $\bfPhi$,
and $\bin(\bfPsi)$ be the binder of its dual $\bfPsi$,
as given in Definition~\ref{def.symplectic ETF}.
Then:
\begin{enumerate}
\renewcommand{\labelenumi}{(\alph{enumi})}
\item
$\bin(\bfPhi)$ is empty if $P>2$.
When $P=2$, $\bin(\bfPhi)$ consists of sets of the form $\bfv+(\set{\bfzero}\cup\calS)$ where $\bfv\in\calV$ and $\calS\subseteq\calV$ satisfies $\#(\calS)=2^J+1$ and $\rmB(\bfv_1,\bfv_2)=1$ for all $\bfv_1,\bfv_2\in\calS$, $\bfv_1\neq\bfv_2$.

\item
$\bin(\bfPsi)$ consists of sets of the form $\bfv+\calL$ where $\bfv\in\calV$ and $\calL\subseteq\calV$ satisfies $\#(\calL)=P^J$ and $\rmB(\bfv_1,\bfv_2)=0$ for all $\bfv_1,\bfv_2\in\calL$.
\end{enumerate}
\end{lemma}

\begin{proof}
(a)
Since $\ip{\bfphi_{\bfv_1}}{\bfphi_{\bfv_2}}
=\exp(\frac{2\pi\rmi}{P}\rmB(\bfv_1,\bfv_2))$ is unimodular when $\bfv_1\neq\bfv_2$ and $\norm{\bfphi_\bfv}^2=P^J+1$ for all $\bfv$,
\eqref{eq.triple product simplex} gives that a subset $\calK$ of $\calV$ belongs to $\bin(\bfPhi)$ if and only if both $\#(\calK)=P^J+2$ and
\begin{align*}
-1=\ip{\bfphi_{\bfv_1}}{\bfphi_{\bfv_2}}\ip{\bfphi_{\bfv_2}}{\bfphi_{\bfv_3}}\ip{\bfphi_{\bfv_3}}{\bfphi_{\bfv_1}}
&=\exp(\tfrac{2\pi\rmi}{P}[\rmB(\bfv_1,\bfv_2)+\rmB(\bfv_2,\bfv_3)+\rmB(\bfv_3,\bfv_1)])\\
&=\exp(\tfrac{2\pi\rmi}{P}\rmB(\bfv_1-\bfv_3,\bfv_2-\bfv_3))
\end{align*}
for all pairwise distinct $\bfv_1,\bfv_2,\bfv_3\in\calK$.
Since $-1$ is only a $P$th root of unity when $P=2$,
$\bin(\bfPhi)$ is thus empty when $P>2$.
When instead $P=2$,
$\calK\in\bin(\bfPhi)$ if and only if $\#(\calK)=2^J+2$ and $\rmB(\bfv_1-\bfv_3,\bfv_2-\bfv_3)=1$ for all pairwise distinct $\bfv_1,\bfv_2,\bfv_3\in\calK$.
These conditions are translation invariant,
and so $\calK\in\bin(\bfPhi)$ if and only if $\calK_0:=-\bfv+\calK\in\bin(\bfPhi)$ for any $\bfv\in\calV$.
In particular,
for any $\calK\in\bin(\bfPhi)$, letting $\bfv\in\calK$ and $\bfv_3=\bfzero\in\calK_0$,
we have $1=\rmB(\bfv_1-\bfzero,\bfv_2-\bfzero)=\rmB(\bfv_1,\bfv_2)$ for any $\bfv_1,\bfv_2\in\calK_0$ with $\bfzero\neq\bfv_1\neq\bfv_2\neq\bfzero$.
Letting $\calS=\calK_0\backslash\set{\bfzero}$,
any $\calK\in\bin(\bfPhi)$ can thus be written as $\calK=\bfv+(\set{\bfzero}\cup\calS)$ where $\bfv\in\calV$, $\#(\calS)=2^J+1$ and $\rmB(\bfv_1,\bfv_2)=1$ for all $\bfv_1,\bfv_2\in\calS$ with $\bfv_1\neq\bfv_2$.
Conversely, for any such $\calS$, we claim that $\calK_0:=\set{\bfzero}\cup\calS$ belongs to $\bin(\bfPhi)$.
Indeed,
$\#(\calK_0)=2^J+2$ since $\#(\calS)=2^J+1$ and $\bfzero\notin\calS$.
Moreover, since $P=2$,
$\rmB(\bfv_1-\bfv_3,\bfv_2-\bfv_3)
=\rmB(\bfv_1,\bfv_2)+\rmB(\bfv_2,\bfv_3)+\rmB(\bfv_3,\bfv_1)=1$
for all pairwise distinct $\bfv_1,\bfv_2,\bfv_3\in\calK_0$,
being either $1+1+1=1$ when such $\bfv_1,\bfv_2,\bfv_3$ are all nonzero,
and being a sum of one $1$ and two zeros when any one of these three vectors is $\bfzero$.
Having the claim,
we further note that since $\bin(\bfPhi)$ is translation invariant,
it contains $\bfv+(\set{\bfzero}\cup\calS)$ for any $\bfv\in\calV$ and any such $\calS$.

(b) Since
$\ip{\bfpsi_{\bfv_1}}{\bfpsi_{\bfv_2}}=-\ip{\bfphi_{\bfv_1}}{\bfphi_{\bfv_2}}$
is unimodular when $\bfv_1\neq\bfv_2$ and has $\norm{\bfpsi_\bfv}^2=P^J-1$ for all $\bfv$,
\eqref{eq.triple product simplex} gives that a subset $\calK$ of $\calV$ belongs to $\bin(\bfPsi)$ if and only if both $\#(\calK)=P^J$ and
\begin{align*}
-1=\ip{\bfpsi_{\bfv_1}}{\bfpsi_{\bfv_2}}\ip{\bfpsi_{\bfv_2}}{\bfpsi_{\bfv_3}}\ip{\bfpsi_{\bfv_3}}{\bfpsi_{\bfv_1}}
&=-\ip{\bfphi_{\bfv_1}}{\bfphi_{\bfv_2}}\ip{\bfphi_{\bfv_2}}{\bfphi_{\bfv_3}}\ip{\bfphi_{\bfv_3}}{\bfphi_{\bfv_1}}\\
&=-\exp(\tfrac{2\pi\rmi}{P}\rmB(\bfv_1-\bfv_3,\bfv_2-\bfv_3)),
\end{align*}
for all pairwise distinct $\bfv_1,\bfv_2,\bfv_3\in\calK$.
As such, $\calK\in\bin(\bfPsi)$ if and only if $\#(\calK)=P^J$ and $0=\rmB(\bfv_1-\bfv_3,\bfv_2-\bfv_3)$ for any pairwise distinct $\bfv_1,\bfv_2,\bfv_3\in\calK$.
As such, $\calK\in\bin(\bfPsi)$ if and only if $\calL:=-\bfv+\calK\in\bin(\bfPsi)$ for any $\bfv\in\calV$.
Thus, if $\calK\in\bin(\bfPsi)$ then taking $\bfv\in\calK$ and $\bfv_3=\bfzero\in\calL$,
we have $0=\rmB(\bfv_1-\bfzero,\bfv_2-\bfzero)=\rmB(\bfv_1,\bfv_2)$ for any distinct nonzero $\bfv_1,\bfv_2\in\calL$.
Since any $\bfv\in\calV$ is orthogonal to $\bfzero$ and to itself,
$\rmB(\bfv_1,\bfv_2)=0$ for all $\bfv_1,\bfv_2\in\calL$.
Conversely, any $\calL\subseteq\calV$ with these properties belongs to $\bin(\bfPsi)$,
having $\#(\calL)=P^J$ and $\rmB(\bfv_1-\bfv_3,\bfv_2-\bfv_3)=0-0-0+0=0$ for any $\bfv_1,\bfv_2,\bfv_3\in\calV$.
As such, for any such $\calL$, we also have $\bfv+\calL\in\bin(\bfPsi)$ for any $\bfv\in\calV$.
\end{proof}

In light of the prior result and the fact that orthogonality is simpler than nonorthogonality in general,
we for the moment focus exclusively on the binder of the dual of a symplectic ETF.
In particular, a subset $\calS$ of $\calV$ is \textit{totally orthogonal} if $\rmB(\bfv_1,\bfv_2)=0$ for all $\bfv_1,\bfv_2\in\calS$, that is, if $\calS\subseteq\calS^\perp$.
When a subset $\calL$ of $\calV$ satisfies $\calL=\calL^\perp$ it is known as a \textit{Lagrangian subspace} of $\calV$.
Any such $\calL$ is necessarily a subspace of $\calV$ since $\calS^\perp$ is a subspace of $\calV$ for any subset $\calS$ of $\calV$.
As this name suggests, such sets are topics of classical interest.
The following result gives the relevant folklore.
For the sake of completeness,
these facts are proven in the appendix%
%\ of~\cite{FickusL23}% COMMENT
.

\begin{lemma}[Folklore]
\label{lemma.totally orthogonal}
Let $\rmB$ be a symplectic form on a vector space $\calV$ of dimension $2J$ over a field $\bbF$.
Let $\calS$ and $\calL$ be a subset and subspace of $\calV$, respectively,
Let $Q$ be a prime power.

\begin{enumerate}
\renewcommand{\labelenumi}{(\alph{enumi})}
\item
If $\calS$ is totally orthogonal then so is $\Span(\calS)$.\smallskip

\item
If $\calL\subseteq\calL^\perp$ then $\dim(\calL)\leq J$.
Here, $\dim(\calL)=J$ if and only if $\calL$ is Lagrangian, that is,
$\calL=\calL^\perp$.\smallskip

\item
A finite sequence $(\bfv_m)_{m=1}^M$ of vectors in $\calV$ is an ordered basis for a totally orthogonal subspace of $\calV$ if and only if \smash{$\bfv_m\in\calV_m^\perp\backslash\calV_m^{}$} for each $m\in[M]$ where \smash{$\calV_m:=\Span(\bfv_{m'})_{m'=1}^{m-1}$}.\smallskip

\item
When $\bbF=\bbF_Q$ has order $Q$,
the number of Lagrangian subspaces of $\calV$ is
$\prod_{j=1}^{J}(Q^j+1)$.
\end{enumerate}
\end{lemma}

Combining these results with our prior observations yields the following main result:

\begin{theorem}
\label{theorem.binder of dual symplectic}
For any prime $P$ and positive integer $J$,
the binder of the dual $\bfPsi$ of a $P^{2J}$-vector symplectic ETF $\bfPhi$
(Definition~\ref{def.symplectic ETF}) consists of all affine Lagrangian subspaces of $\calV$,
that is, all subsets of $\calV$ of the form $\bfv+\calL$ where $\bfv\in\calV$ and $\calL^\perp=\calL$.
This binder forms a BIBD with parameters
\begin{equation*}
V=P^{2J},
\quad
K=P^J,
\quad
\Lambda=\prod_{j=1}^{J-1}(P^j+1),
\quad
R=\prod_{j=1}^{J}(P^j+1),
\quad
B=P^J\prod_{j=1}^{J}(P^j+1).
\end{equation*}
Moreover, a subset of this binder forms an affine plane of order $Q=P^J$.
namely a BIBD with $(V,K,\Lambda,R,B)=(Q^2,Q,1,Q+1,Q(Q+1))$.
\end{theorem}

\begin{proof}
By Lemma~\ref{lemma.binder of symplectic},
$\bin(\bfPsi)$ consists of sets of the form $\bfv+\calL$ where $\bfv\in\calV$ and $\calL$ is a $P^J$-vector totally orthogonal subset of $\calV$.
For any such $\calL$, Lemma~\ref{lemma.totally orthogonal} gives that $\Span(\calL)$ is a totally orthogonal subspace of $\calV$ and so moreover that $\dim(\Span(\calL))\leq J$.
Thus, $P^J=\#(\calL)\leq\#(\Span(\calL))\leq P^J$,
implying that $\calL=\Span(\calL)$ is a totally orthogonal subspace of $\calV$ with $\#(\calL)=P^J$, i.e., $\dim(\calL)=J$.
Lemma~\ref{lemma.totally orthogonal} gives that any such $\calL$ is Lagrangian.
Conversely, any Lagrangian subspace $\calL$ is totally orthogonal and,
by Lemma~\ref{lemma.totally orthogonal}, has $\dim(\calL)=J$ and so $\#(\calL)=P^J$.
Thus indeed,
\begin{equation*}
\bin(\bfPsi)=\set{\bfv+\calL: \bfv\in\calV,\,\calL\subseteq\calV,\,\calL=\calL^\perp}.
\end{equation*}

Lemma~\ref{lemma.binder of symplectic} moreover guarantees that Lagrangian subspaces $\calL$ of $\calV$ exist.
In fact, it gives that there are exactly $\prod_{j=1}^{J}(P^j+1)$ distinct such $\calL$.
Since $\bfPsi$ is a DTETF~\cite{IversonM22} and $\bin(\bfPsi)$ is nonempty,
Theorem~\ref{theorem.binders of DTETFs} guarantees that $\bin(\bfPsi)$ forms a BIBD,
that is, that $(\calV,(\calK_b)_{b\in\calB})$ is a BIBD where $(\calK_b)_{b\in\calB}$ is any enumeration of $\bin(\bfPsi)$.
The vertex set $\calV$ of this BIBD has cardinality $V=P^{2J}$.
Every block of it is of the form $\bfv+\calL$ and so has cardinality $K=\#(\calL)=P^J$.
Moreover, recall that in a BIBD, the number $R$ of blocks that contain a given vertex is independent of one's choice of vertex.
In particular, $R$ is the number of affine Lagrangian subspaces of $\calV$ that contain $\bfzero$, that is, the number of Lagrangian subspaces of $\calV$, that is, \smash{$R=\prod_{j=1}^{J}(P^j+1)$}.
Since $\frac{V}{K}=P^J$,
\eqref{eq.BIBD parameter relations} gives that \smash{$B=\frac{V}{K}R=P^J\prod_{j=1}^{J}(P^j+1)$}.
This is consistent with the fact that each of the $R$ Lagrangian subspaces of $\calV$ has codimension $J$ and so has $P^J$ distinct cosets.
Similarly, \smash{$\frac{V-1}{K-1}=\frac{P^{2J}-1}{P^J-1}=P^J+1$} and so~\eqref{eq.BIBD parameter relations} gives
\smash{$\Lambda
=\frac{K-1}{V-1}R
=(P^J+1)^{-1}\prod_{j=1}^{J}(P^j+1)
=\prod_{j=1}^{J-1}(P^j+1)$}.
From this analysis it is clear that when $J=1$,
this empty product should be regarded as $1$,
as is the standard convention.

All that remains to be shown is that a subset of $\bin(\bfPsi)$ forms an affine plane of order \smash{$Q:=P^J$}.
Here, recall that any symplectic form on a finite-dimensional vector space has a symplectic basis~\eqref{eq.symplectic basis},
and that this implies that up to isomorphism, there is exactly one symplectic form on spaces of a given even dimension $2J$ over a given field $\bbF$.
Without loss of generality we can thus choose $\calV$ to be $\bbF_Q^2$,
and let $\rmB$ be the composition of the field trace $\tr:\bbF_Q\rightarrow\bbF_P$ (or any other such nontrivial linear functional) with the canonical $\bbF_Q$-valued symplectic form $\rmB_1$ on this space~\eqref{eq.canonical symplectic form}:
\begin{equation*}
\rmB=\tr\circ\,\rmB_1:\bbF_Q\times\bbF_Q\rightarrow\bbF,
\quad
\rmB((x_1,x_2),(y_1,y_2))=\tr(x_1y_2-x_2y_1).
\end{equation*}
Note that the $\bbF_Q$-multiples of any $(x_1,x_2)\neq(0,0)$ form a Lagrangian subspace of $\bbF_Q^2$ with respect to $\rmB$:
clearly, $\set{c(x_1,x_2): c\in\bbF_Q}$ has cardinality $Q=P^J$,
and this set is totally orthogonal since
\begin{equation*}
\rmB(c_1(x_1,x_2),c_2(x_1,x_2))
=\rmB((c_1x_1,c_1x_2),(c_2x_1,c_2x_2))
=\tr(c_1x_1c_2x_2-c_1x_2c_2x_1)
=\tr(0)
=0
\end{equation*}
for any $c_1,c_2\in\bbF_Q$.
Since $\bin(\bfPsi)$ is closed under translation,
it thus contains all affine lines (one-dimensional subspaces) in $\bbF_Q^2$,
regarded as a plane (two-dimensional space) over $\bbF_Q$,
namely the classical (Desarguesian) affine plane of order $Q$.
\end{proof}

\begin{example}
\label{example.binder of 6x16}
From Example~\ref{example.symplectic ETFs},
recall that when $P=2$ and $J=2$,
the canonical symplectic form~\eqref{eq.canonical symplectic form} on $\bbF_2^4$ is
$\rmB_2((x_1,x_2,x_3,x_4),(y_1,y_2,y_3,y_4))=(x_1y_2+x_2y_1)+(x_3y_4+x_4y_3)$,
and that the dual $\bfPsi$ of the corresponding symplectic ETF $\bfPhi$ has Gram matrix $4\brI-\bfGamma_2$ where $\bfGamma_2$ is given by~\eqref{eq.Gamma 16x16}.
By Theorem~\ref{theorem.binder of dual symplectic},
$\bin(\bfPsi)$ consists of all affine Lagrangian subspaces of \smash{$\bbF_2^4$},
namely all sets of the form $\bfv+\calL$ where \smash{$\calL\subseteq\bbF_2^4$} satisfies $\calL^\perp=\calL$,
and moreover, that such sets form a BIBD with parameters
$(V,K,\Lambda,R,B)
=(2^{2(2)},2^{2},2+1,(2^2+1)(2+1),2^2(2^2+1)(2+1))
=(16,4,3,15,60)$.
In particular, there are exactly $15$ Lagrangian subspaces of \smash{$\bbF_2^4$}, namely:
\begin{equation}
\label{eq.Lagrangian subspaces in F_2^4}
\begin{tiny}\begin{array}{ccccc}
\set{0000,0010,1000,1010},&
\set{0000,0011,1100,1111},&
\set{0000,0111,1001,1110},&
\set{0000,0110,1011,1101},&
\set{0000,0001,0100,0101},\\
\set{0000,0001,1000,1001},&
\set{0000,0001,1100,1101},&
\set{0000,0010,0100,0110},&
\set{0000,0010,1100,1110},&
\set{0000,0011,0100,0111},\\
\set{0000,0011,1000,1011},&
\set{0000,0101,1010,1111},&
\set{0000,0101,1011,1110},&
\set{0000,0110,1001,1111},&
\set{0000,0111,1010,1101}.
\end{array}\end{tiny}
\end{equation}
These can be obtained, for example, via the iterative construction used in Lemma~\ref{lemma.totally orthogonal}.
To explain in detail,
there are $(16-1)(8-2)=90$ ordered bases $(\bfv_1,\bfv_2)$ for some Lagrangian subspace of $\bbF_2^4$,
each consisting of an arbitrary $\bfv_1\neq\bfzero$ followed by some $\bfv_2\in\set{\bfv_1}^\perp\backslash\set{\bfzero,\bfv_1}$.
Any two such ordered bases span the same Lagrangian subspace if and only if they are equivalent via one of $(4-1)(4-2)=6$ invertible matrices in \smash{$\bbF_2^2$}.
Since each of these $\frac{90}{6}=15$ Lagrangian subspaces is a subgroup of \smash{$\bbF_2^4$} of index $4$,
it has $4$ distinct cosets, yielding $4(15)=60$ affine Lagrangian subspaces altogether.
Since \smash{$\binom{16}{4}=1820$},
the probability of randomly selecting a $4$-vector subset of this $16$-vector ETF for $\bbR^6$ whose vectors are linearly dependent is thus $\frac{60}{1820}=\frac{3}{91}\approx0.033$.

Since $\bin(\bfPsi)$ forms a BIBD,
we can phase (sign) the entries of the corresponding $60\times16$ incidence matrix \`{a} la~\eqref{eq.phased incidence} to embed its dual $\bfPhi$ into a $10$-dimensional subspace of $\bbR^{60}$.
In fact, Theorem~\ref{theorem.binder of dual symplectic} guarantees that a subset of $\bin(\bfPsi)$ forms an affine plane,
namely a BIBD with $(V,K,\Lambda,R,B)=(16,4,1,5,20)$,
and phasing its $20\times16$ incidence matrix yields an equally relatively sparse---yet more concise---embedding of $\bfPhi$ into a $10$-dimensional subspace of $\bbR^{20}$.
To see how such an affine plane arises,
let \smash{$\bbF_4=\set{0,1,\alpha,\beta}$} where $\beta=\alpha^2=1+\alpha$,
and regard \smash{$\bbF_4^2$} as a two-dimensional space over \smash{$\bbF_4$} that is equipped with the canonical symplectic form~\eqref{eq.canonical symplectic form}.
Since \smash{$\bbF_4^2$} is two-dimensional,
each of its one-dimensional subspaces is Lagrangian.
These spaces are:
\begin{equation}
\label{eq.binder of 6x16 1}
\set{00,10,\alpha 0,\beta 0},\quad
\set{00,11,\alpha\alpha,\beta\beta},\quad
\set{00,1\alpha,\alpha\beta,\beta1},\quad
\set{00,1\beta,\alpha 1,\beta\alpha},\quad
\set{00,01,0\alpha,0\beta}.
\end{equation}
Now consider the mapping \smash{$(x_1,x_2,x_3,x_4)\in\bbF_2^4\mapsto(x_1+x_3\alpha,x_2+x_4\alpha)\in\bbF_4^2$}.
Under it, the canonical symplectic form on $\bbF_2^4$ arises as the composition of that on \smash{$\bbF_4^2$} and the nontrivial linear functional \smash{$\bfL:\bbF_4\rightarrow\bbF_2$} defined by $\bfL(0)=\bfL(\alpha)=0$ and $\bfL(1)=\bfL(\alpha^2)=1$:
\begin{equation*}
(x_1+x_3\alpha)(y_2+y_4\alpha)+(x_2+x_4\alpha)(y_1+y_3\alpha)
%=(x_1y_2+x_2y_1)+(x_1y_4+x_3y_2+x_2y_3+x_4y_1)\alpha+(x_3y_4+x_4y_3)\alpha^2
\overset{\bfL}{\longmapsto}(x_1y_2+x_2y_1)+(x_3y_4+x_4y_3).
\end{equation*}
In particular,
it maps Lagrangian subspaces of \smash{$\bbF_4^2$} over \smash{$\bbF_4$} to Lagrangian subspaces of \smash{$\bbF_2^4$} over \smash{$\bbF_2$}.
Applying it to the five sets in~\eqref{eq.binder of 6x16 1} yields those given in the first row of~\eqref{eq.Lagrangian subspaces in F_2^4}.
Since all cosets of the former form a BIBD with $\Lambda=1$,
the same is true of the latter.

Since the cosets of the $15$ Lagrangian subspaces in \smash{$\bbF_2^4$} form a BIBD with $(V,K,\Lambda,R,B)$ parameters $(16,4,3,15,60)$,
and the cosets of these particular $5$ subspaces form one with parameters $(16,4,1,5,20)$,
the cosets of the remaining $10$ Lagrangian subspaces---those given in the second and third rows of~\eqref{eq.Lagrangian subspaces in F_2^4}---form a BIBD with parameters $(16,4,2,10,40)$.
We claim that the latter does not resolve into two affine planes.
Indeed, if it did, the nonzero members of these $10$ subspaces must resolve into two partitions of \smash{$\bbF_2^4\backslash\set{0000}$}.
One of these must contain $\set{0001,1000,1001}$,
and so the other must contain $\set{0011,1000,1011}$ and $\set{0110,1001,1111}$,
which in turn implies that the intersecting sets $\set{0101,1011,1110}$ and $\set{0101,1010,1111}$ belong to the same partition, a contradiction.
\end{example}

\begin{remark}
Theorem~\ref{theorem.binder of dual symplectic} resolves some open problems of~\cite{BodmannK20}.
Notably, it disproves Conjecture~5.12 of~\cite{BodmannK20} as stated,
and moreover proves a modified version of it.
To explain,
for any odd prime $P$ and positive integer $J$,
\cite{BodmannK20} constructs a $P^{2J}$-vector Gabor--Steiner ETF \smash{$(\bfpsi_{(\bfx,\bfy)})_{(\bfx,\bfy)\in\bbF_P^J\times\bbF_P^J}$} for a complex space of dimension $\frac12 P^J(P^J-1)$.
Here since $P$ is an odd prime,
$\exp(-\frac{4\pi\rmi}{P})$ is a primitive $P$th root of unity,
and so Lemma~5.1 of~\cite{BodmannK20} gives, without loss of generality,
that the inner product of any two distinct members of this Gabor--Steiner ETF is
\begin{align*}
\ip{\bfpsi_{(\bfx_1,\bfy_1)}}{\bfpsi_{(\bfx_2,\bfy_2)}}
&=-\prod_{j=1}^{J}\exp(\tfrac{2\pi\rmi}{P}[\bfy_2(j)-\bfy_1(j)][\bfx_2(j)+\bfx_1(j)-1])\\
%&=-\exp(\tfrac{2\pi\rmi}{P}\sum_{j=1}^J
%\bigset{[\bfx_1(j)\bfy_2(j)-\bfx_2(j)\bfy_1(j)]-[\bfx_1(j)-1]\bfy_1(j)+[\bfx_2(j)-1]\bfy_2(j)}\\
&=-\overline{z_{(\bfx_1,\bfy_1)}}z_{(\bfx_2,\bfy_2)}\exp(\tfrac{2\pi\rmi}{P}\rmB((\bfx_1,\bfy_1),(\bfx_2,\bfy_2)))
\end{align*}
where $z_{(\bfx,\bfy)}:=\exp(\tfrac{2\pi\rmi}{P}\sum_{j=1}^J[\bfx(j)-1]\bfy(j))$ is a unimodular scalar for any $(\bfx,\bfy)\in\bbF_P^J\times\bbF_P^J$,
and
\begin{equation*}
\rmB((\bfx_1,\bfy_1),(\bfx_2,\bfy_2))
:=\sum_{j=1}^J[\bfx_1(j)\bfy_2(j)-\bfx_2(j)\bfy_1(j)]
\end{equation*}
defines a symplectic form on $\bbF_P^J\times\bbF_P^J$.
In fact, this $\rmB$ is the canonical symplectic form~\eqref{eq.canonical symplectic form} on $\bbF_P^{2J}$,
provided we identify any $(\bfx,\bfy)\in\bbF_P^J\times\bbF_P^J$ with $(\bfx(1),\bfy(1),\bfx(2),\bfy(2),\dotsc,\bfx(J),\bfy(J))\in\bbF_P^{2J}$.
As such, as briefly mentioned in Example~5.10 of~\cite{IversonM22},
this type of Gabor--Steiner ETF is thus equivalent to a dual of an ETF that is symplectic in the sense of Definition~\ref{def.symplectic ETF}.
Since triple products---and so also binders---are preserved by equivalence,
Theorem~\ref{theorem.binder of dual symplectic} guarantees that $\bin(\bfPsi)$ forms a BIBD with parameters $(V,K,\Lambda)=(P^{2J},P^J,\prod_{j=1}^{J-1}(P^j+1))$.
In particular, when $J=1$ and $J=2$,
$\bin(\bfPsi)$ forms a BIBD with $(V,K,\Lambda)$ being $(P^2,P,1)$ and $(P^4,P^2,P+1)$, respectively, as previously noted in Theorems~5.6 and~5.8 of~\cite{BodmannK20}, respectively.
Based on this data, Conjecture~5.12 of~\cite{BodmannK20} predicts that, for any positive integer $J$, $\bin(\bfPsi)$ forms a BIBD with $(V,K,\Lambda)=(P^{2J},P^J,\frac{P^J-1}{P-1})$.
By Theorem~\ref{theorem.binder of dual symplectic} this conjecture is false whenever \smash{$\frac{P^J-1}{P-1}<\prod_{j=1}^{J-1}(P^j+1)$}, namely for all $J\geq 3$.
In fact, we see that $\bin(\bfPsi)$ actually forms a BIBD that is (much) larger than \cite{BodmannK20} predicts whenever $J\geq 3$.

Notably, Theorem~5.9 of~\cite{BodmannK20} further shows that when $P$ is odd and $J=2$,
$\bin(\bfPsi)$ can not be resolved into $P+1$ affine planes.
Moreover, the paragraph of~\cite{BodmannK20} that precedes Corollary~5.11 seemingly doubts whether $\bin(\bfPhi)$ contains even a single affine plane, again when $P$ is odd and $J=2$.
Theorem~\ref{theorem.binder of dual symplectic} shows that this is unfounded: regardless of $P$ and $J$,
the binder of the dual of any symplectic ETF always contains an affine plane.
That said, Example~\ref{example.binder of 6x16} suggests that $\bin(\bfPsi)$ does not resolve completely into affine planes even when $P=2$ and $J=2$.
This begs the question,
For any finite vector space that is equipped with a symplectic form,
when can the set of all affine Lagrangian subspaces of it, less one affine plane, be resolved into proper subsets, each of which forms a BIBD?
Being a natural question of finite symplectic geometry,
the answer to it might be found in the classical literature.
We leave it for future work.
\end{remark}

Having fully characterized the binder of the dual $\bfPsi$ of any symplectic ETF $\bfPhi$, we devote the remainder of this section to the binder of $\bfPhi$ itself.
By Lemma~\ref{lemma.binder of symplectic}, we already know that $\bin(\bfPhi)$ is empty unless $P=2$.
As we now explain, $\bin(\bfPhi)$ is also empty when $P=2$ except in two cases, namely when either $J=1$ or $J=2$.
When $P=2$ and $J=1$, this is not surprising:
here, $\bfPhi$ is a $4$-vector ETF for a space of dimension $3$,
and so as explained in Example~\ref{example.binder of simplex}, its binder consists of a single copy of $\calV$.
When instead $P=2$ and $J=2$, $\bin(\bfPhi)$ has a remarkable structure.

\begin{theorem}
\label{theorem.binder of symplectic}
Let $\rmB$ be a symplectic form on a vector space $\calV$ over $\bbF_2$ of dimension $2J$,
and let $\calS\subseteq\calV$ satisfy $\rmB(\bfv_1,\bfv_2)=1$ for all $\bfv_1,\bfv_2\in\calS$ with $\bfv_1\neq\bfv_2$.
Then $\#(\calS)\leq 2J+1$.
When $P=2$, the binder of the corresponding symplectic ETF $\bfPhi$ (Definition~\ref{def.symplectic ETF}) is thus empty unless $J\in\set{1,2}$.

When $P=2$ and $J=2$, $\bin(\bfPhi)$ is nontrivial and so forms a BIBD.
It consists of all shifts of any subset $\calV$ of the following form,
where $(\bfv_1,\bfv_2,\bfv_3,\bfv_4)$ is any symplectic basis~\eqref{eq.symplectic basis} for $\calV$:
\begin{equation}
\label{eq.proto binder of 10x16}
\set{\
\bfzero,\quad
\bfv_1,\quad
\bfv_2,\quad
\bfv_1+\bfv_2+\bfv_3,\quad
\bfv_1+\bfv_2+\bfv_4,\quad
\bfv_1+\bfv_2+\bfv_3+\bfv_4\ }.
\end{equation}
\end{theorem}

\begin{proof}
We prove the first claim by induction on $J$.
When $J=1$,
$\calV$ contains $2^{2J}=4$ vectors,
and so $\calS$ necessarily satisfies $\#(\calS)\leq 2J+1=3$ or else $\calS=\calV$ contains both $\bfzero$ and $\bfv\neq\bfzero$ where $\rmB(\bfzero,\bfv)=0$.
Now suppose this claim holds for a given positive integer $J-1$,
that $\rmB$ is a symplectic form on a vector space $\calV$ over $\bbF_2$ of dimension $2J$ and that $\calS\subseteq\calV$ satisfies $\rmB(\bfv_1,\bfv_2)=1$ for all distinct $\bfv_1,\bfv_2\in\calS$.
We show that $\#(\calS)\leq 2J+1$.
This obviously holds if $\#(\calS)\leq 2$.
When instead $\#(\calS)>2$,
take any $\bfv_1,\bfv_2\in\calS$, $\bfv_1\neq\bfv_2$.
Since $\rmB(\bfv_1,\bfv_2)=1$, $(\bfv_1,\bfv_2)$ is a symplectic pair,
and so $\calV$ is an internal direct sum of $\calV_0:=\set{\bfv_1,\bfv_2}^\perp$ and
$\calV_0^\perp=\Span\set{\bfv_1,\bfv_2}$.
The set $\calS_0:=\bfv_1+\bfv_2+\calS\backslash\set{\bfv_1,\bfv_2}$ has $\#(\calS_0)=\#(\calS)-2$.
Moreover, $\calS_0\subseteq\calV_0$ since
\begin{equation*}
\rmB(\bfv_1,\bfv_1+\bfv_2+\bfv)
=0+1+1
=0
=1+0+1
=\rmB(\bfv_2,\bfv_1+\bfv_2+\bfv),
\quad
\forall\,\bfv\in\calS\backslash\set{\bfv_1,\bfv_2}.
\end{equation*}
In particular,
$\rmB(\bfv_3,\bfv_1+\bfv_2)=\rmB(\bfv_4,\bfv_1+\bfv_2)=0$ for any $\bfv_3,\bfv_4\in\calS_0$.
At the same time, when such $\bfv_3,\bfv_4$ are distinct then $\bfv_1+\bfv_2+\bfv_3$ and $\bfv_1+\bfv_2+\bfv_4$ are distinct members of $\calS$,
implying
\begin{equation*}
1
=\rmB(\bfv_1+\bfv_2+\bfv_3,\bfv_1+\bfv_2+\bfv_4)
=0+0+0+\rmB(\bfv_3,\bfv_4)
=\rmB(\bfv_3,\bfv_4).
\end{equation*}
Since $\dim(\calV_0)=\dim(\calV)-2=2J-2=2(J-1)$,
applying the inductive hypothesis to $\calS_0$ gives $\#(\calS_0)\leq 2(J-1)+1=2J-1$,
and so $\#(\calS)=\#(\calS_0)+2\leq 2J+1$.

Now recall from Lemma~\ref{lemma.binder of symplectic} that $\bin(\bfPhi)$ consists of sets of the form $\bfv+(\set{\bfzero}\cup\calS)$ where $\#(\calS)=2^J+1$ and $\rmB(\bfv_1,\bfv_2)=1$ for all distinct $\bfv_1,\bfv_2\in\calS$.
From the first claim, such a set $\calS$ can only exist if $2^J+1\leq 2J+1$,
namely $J\in\set{1,2}$.
As such, $\bin(\bfPhi)$ is empty when $J>2$.

When $J=2$,
$\bin(\bfPhi)$ consists of all sets of the form $\bfv+(\set{\bfzero}\cup\calS)$ where $\#(\calS)=5$ and $\rmB(\bfv_1,\bfv_2)=1$ for all distinct $\bfv_1,\bfv_2\in\calS$.
For any distinct $\bfv_1,\bfv_2\in\calS$,
the inductive argument above gives that $\calS_0:=\bfv_1+\bfv_2+\calS\backslash\set{\bfv_1,\bfv_2}$ consists of the three nonzero members of $\calV_0:=\set{\bfv_1,\bfv_2}^\perp$.
That is, $\calS_0=\set{\bfv_3,\bfv_4,\bfv_3+\bfv_4}$ where $(\bfv_1,\bfv_2,\bfv_3,\bfv_4)$ is a symplectic basis for $\calV$,
implying that indeed $\set{\bfzero}\cup\calS$ is of form~\eqref{eq.proto binder of 10x16}.
In particular, $\bfPhi$ is a DTETF whose binder is nonempty,
implying by Theorem~\ref{theorem.binders of DTETFs} that it forms a BIBD.
\end{proof}

To be clear,
when $P=2$ and $J=2$,
one may of course use a rigorous computer algebra system to explicitly compute the binder of the symplectic ETF $\bfPhi$.
For example, recalling the matrix $\bfGamma_2$ of~\eqref{eq.Gamma 16x16},
one may simply compute which of the $\binom{16}{6}=8008$ principal submatrices of $4\brI+\bfGamma_2$ of size $6$ are singular.
Doing so reveals just $16$ sets, all translates of a single $6$-element subset of $\calV$.
Together, these sets form a BIBD with parameters $(V,K,\Lambda,R,B)=(16,6,2,6,16)$ with a $\calV$-circulant incidence matrix, and so arise from a \textit{difference set} for $\calV$.
Notably, this requires that any set of the form~\eqref{eq.proto binder of 10x16} for one choice of symplectic basis is simply a shift of the analogous set for any other choice of such a basis.
In the next section, we give an elementary proof of this fact that is an offshoot of our investigation into the binders of certain sub-ETFs of a symplectic ETF and its dual in the special case where $P=2$.

%%%%%%%%%%%%%%%%%%%%%%%%%%%%%%%%%%%%%%%%%%%%%%%%%%%%%%%%%%%%%%%%
\section{The binders of equiangular tight frames arising from quadratic forms}
%%%%%%%%%%%%%%%%%%%%%%%%%%%%%%%%%%%%%%%%%%%%%%%%%%%%%%%%%%%%%%%%

In this section, we calculate the binder of any member of four particular infinite families of real DTETFs that arise from quadratic forms over the binary field $\bbF_2:=\set{0,1}$.
ETFs in one of these four families consist of $2^{J-1}(2^J-1)$ vectors that span a space of dimension $\tfrac13(2^{J-1}-1)(2^J-1)$, where $J$ is some positive integer,
e.g., $28$ vectors in $\bbR^7$.
ETFs in another consist of $2^{J-1}(2^J+1)$ vectors that span a space of dimension $\tfrac13(2^{J-1}+1)(2^J+1)$, e.g., $36$ vectors in $\bbR^{15}$.
The remaining two families consist of their duals, which notably both span spaces of dimension $\frac13(2^{2J}-1)$, e.g., $28$ and $36$ vectors in $\bbR^{21}$.
The existence of real ETFs with such parameters has been known for some time,
including for example \textit{Steiner}~\cite{FickusMT12} and \textit{Tremain}~\cite{FickusJMP18} ETFs arising from a \textit{Steiner triple system} on $2^J-1$ vertices, respectively.
Recently, it was shown that ETFs with these parameters arise as sub-ETFs (or duals of sub-ETFs) of symplectic ETFs and their duals over $\bbF_2$~\cite{FickusIJK21}.
This leads to an infinite family~\cite{DeregowskaFFL22} of real \textit{mutually unbiased} ETFs~\cite{FickusM21}.
Moreover, real DTETFs with these parameters exist~\cite{IversonM24,DempwolffK23},
arising from actions of certain symplectic groups on certain sets of quadratic forms.
By exploiting some known but obscure connections~\cite{Kantor75,Wertheimer86} between such actions and the ETFs that were constructed in~\cite{FickusIJK21},
we verify that the latter are indeed DTETFs.
We then exploit their connections to concepts from Section~4 in conjunction with Theorem~\ref{theorem.binders of DTETFs} to calculate their binders.
Specifically, we exploit Theorem~\ref{theorem.binder of dual symplectic} to show that ETFs in two of these families have nonempty binders whose corresponding BIBDs relate to each other---and to the corresponding BIBD of affine Lagrangian subspaces---in interesting ways,
and use Theorem~\ref{theorem.binder of symplectic} to show that the binders of ETFs in the remaining two families are empty except in a finite number of cases.
We now review the parts of the theory of quadratic forms over $\bbF_2$ that we shall need.

\subsection{Quadratic forms over $\bbF_2$}
\label{subsection.quadratic forms}

Let $\calV$ be a vector space over a field $\bbF$.
The definition of a \textit{quadratic form} on $\calV$ depends on whether $\operatorname{char}(\bbF)=2$, i.e., if $1+1=0\in\bbF$.
When $\operatorname{char}(\bbF)\neq2$, quadratic forms equate to symmetric bilinear forms.
There, a function $\rmQ:\calV\rightarrow\bbF$ is a quadratic form if there exists a bilinear form $\rmB$ on $\calV$ for which
$\rmQ(x_1\bfv_1+x_2\bfv_2)=x_1^2\rmQ(\bfv_1)+2x_1x_2\rmB(\bfv_1,\bfv_2)+x_2^2\rmQ(\bfv_2)$ for all $x_1,x_2\in\bbF$ and $\bfv_1,\bfv_2\in\calV$.
This $\rmB$ is unique and symmetric,
having \smash{$\rmB(\bfv_1,\bfv_2)=\tfrac12[\rmQ(\bfv_1+\bfv_2)-\rmQ(\bfv_1)-\rmQ(\bfv_2)]$} for any $\bfv_1,\bfv_2\in\calV$.
Conversely, any symmetric bilinear form $\rmB$ on $\calV$ arises from the quadratic form $\rmQ(\bfv):=\rmB(\bfv,\bfv)$.
This $\rmQ$ is unique since necessarily
\smash{$\rmQ(\bfv)
=\rmQ(\tfrac12\bfv+\tfrac12\bfv)
=\tfrac14\rmQ(\bfv)+\tfrac12\rmB(\bfv,\bfv)+\tfrac14\rmQ(\bfv)$}
for all $\bfv\in\calV$.
When instead $\operatorname{char}(\bbF)=2$,
it is common practice to modify this definition in a seemingly minor way:
a function $\rmQ:\calV\rightarrow\bbF$ is a quadratic form if there exists a bilinear form $\rmB$ on $\calV$ for which
\begin{equation}
\label{eq.quadratic form over characteristic two}
\rmQ(x_1\bfv_1+x_2\bfv_2)
=x_1^2\rmQ(\bfv_1)+x_1x_2\rmB(\bfv_1,\bfv_2)+x_2^2\rmQ(\bfv_2)
\end{equation}
for all $x_1,x_2\in\bbF$, $\bfv_1,\bfv_2\in\calV$.
As before, $\rmQ$ uniquely determines $\rmB$:
letting $x_1=x_2=1$ here gives
\begin{equation}
\label{eq.quadratic symplectic relation}
\rmB(\bfv_1,\bfv_2)=\rmQ(\bfv_1+\bfv_2)+\rmQ(\bfv_1)+\rmQ(\bfv_2),
\quad\forall\,\bfv_1,\bfv_2\in\calV.
\end{equation}
%When this occurs, we say $\rmQ$ \textit{yields} $\rmB$.
Any such $\rmB$ is necessarily alternating and $\rmQ(\bfzero)=0$:
note
$\rmB(\bfv,\bfv)
=\rmQ(\bfv+\bfv)+\rmQ(\bfv)+\rmQ(\bfv)
=\rmQ(\bfzero)$ for all $\bfv\in\calV$,
and so
$\rmB(\bfv,\bfv)=\rmQ(\bfzero)=\rmB(\bfzero,\bfzero)=0$ for all $\bfv\in\calV$.
Since $-1=1$, such an alternating form $\rmB$ is also symmetric.
Remarkably, in this setting, the quadratic form $\rmQ$ that gives rise to a bilinear form $\rmB$ is not necessarily unique.
For example,
since $(x+y)^2=x^2+y^2$ for any $x,y\in\bbF$,
the canonical symplectic form
$((x_1,x_2),(y_1,y_2))\mapsto x_1y_2+x_2y_1$ on $\bbF^2$
arises from two distinct quadratic forms, namely $(x_1,x_2)\mapsto x_1x_2$ and $(x_1,x_2)\mapsto x_1^2+x_1x_2+x_2^2$:
\begin{multline}
\label{eq.two forms on two-dimensional space}
[(x_1+y_1)^2+(x_1+y_1)(x_2+y_2)+(x_2+y_2)^2]+(x_1^2+x_1x_2+x_2^2)+(y_1^2+y_1y_2+y_2^2)\\
=(x_1+y_1)(x_2+y_2)+x_1x_2+y_1y_2
=x_1y_2+x_2y_1.
\end{multline}

This theory simplifies when $\bbF$ is $\bbF_2$.
A function $\rmQ:\calV\rightarrow\bbF_2$ is a quadratic form if and only if~\eqref{eq.quadratic symplectic relation} defines a bilinear form $\rmB$,
as this $\rmB$ satisfies~\eqref{eq.quadratic form over characteristic two} for any $\bfv_1,\bfv_2\in\calV$ when $x_1=x_2=1$,
and any bilinear form satisfies~\eqref{eq.quadratic form over characteristic two} for any $\bfv_1,\bfv_2\in\calV$ when either $x_1=0$ or $x_2=0$.
When this occurs, we say that $\rmQ$ \textit{yields} $\rmB$.
If $\rmQ_1,\rmQ_2:\calV\rightarrow\bbF$ give rise to the same bilinear form $\rmB$ then $\rmQ_1+\rmQ_2$ is linear:
$\rmQ_1+\rmQ_2$ distributes over multiplication by scalars in $\bbF_2=\set{0,1}$ since $(\rmQ_1+\rmQ_2)(\bfzero)=0+0=0$, and distributes over addition:
\begin{align*}
(\rmQ_1+\rmQ_2)(\bfv_1+\bfv_2)
&=\rmB(\bfv_1,\bfv_2)+\rmQ_1(\bfv_1)+\rmQ_1(\bfv_2)
+\rmB(\bfv_1,\bfv_2)+\rmQ_2(\bfv_1)+\rmQ_2(\bfv_2)\\
&=(\rmQ_1+\rmQ_2)(\bfv_1)+(\rmQ_1+\rmQ_2)(\bfv_2),
\quad\forall\,\bfv_1,\bfv_2\in\calV.
\end{align*}
Conversely, if $\rmQ:\calV\rightarrow\bbF_2$ is a quadratic form that yields $\rmB$ and $\bfL:\calV\rightarrow\bbF_2$ is linear, then $\rmQ_\bfL:=\rmQ+\bfL$ is a quadratic form that also yields $\rmB$ since
\begin{equation*}
\rmQ_\bfL(\bfv_1+\bfv_2)+\rmQ_\bfL(\bfv_1)+\rmQ_\bfL(\bfv_2)
=\rmQ(\bfv_1+\bfv_2)+\rmQ(\bfv_1)+\rmQ(\bfv_2)+\bfL(2\bfv_1+2\bfv_2)
=\rmB(\bfv_1,\bfv_2)
\end{equation*}
for any $\bfv_1,\bfv_2\in\calV$.
Thus, when $\rmQ$ yields $\rmB$,
the set of all quadratic forms that yield $\rmB$ is the affine linear subspace $\rmQ+\set{\bfL:\calV\rightarrow\bbF_2\ |\ \bfL\ \text{is linear}}$ of $\bbF_2^\calV$,
a set whose cardinality equals that of $\calV$.

For the remainder of this section,
let $\calV$ be a vector space over $\bbF_2$ of dimension $2J$,
and let $\rmQ$ be a quadratic form on $\calV$ whose bilinear form~\eqref{eq.quadratic symplectic relation} is nondegenerate and so symplectic.
Forms of this type arise in the constructions of the maximum number of real mutually unbiased bases in spaces whose dimension is a power of $4$~\cite{CameronS73,CalderbankCKS97}.
Here, any linear functional $\bfL:\calV\rightarrow\bbF_2$ has a Riesz representation---
there exists a unique $\bfv_\bfL\in\calV$ such that $\bfL(\bfv)=\rmB(\bfv_\bfL,\bfv)$ for all $\bfv\in\calV$---implying
\begin{equation}
\label{eq.other quadratic forms}
\rmQ_\bfL(\bfv)
=(\rmQ+\bfL)(\bfv)
=\rmQ(\bfv)+\rmB(\bfv_\bfL,\bfv)
=\rmQ(\bfv+\bfv_\bfL)+\rmQ(\bfv_\bfL),\quad\forall\,\bfv\in\calV.
\end{equation}
Saying that a given $\bfv\in\calV$ is \textit{singular} (with respect to $\rmQ$) when $\rmQ(\bfv)=0$,
we thus see that the quadratic forms that yield $\rmB$ partition into two types:
shifts of $\rmQ$ by singular $\bfv$,
and shifts of $\rmQ+\bfone$ by nonsingular $\bfv$.
The number of the former is thus the size of the \textit{quadric} $\calQ$ of $\rmQ$,
defined as
\begin{equation}
\label{eq.quadric}
\calQ:=\set{\bfv\in\calV: \rmQ(\bfv)=0}.
\end{equation}
Since $\rmQ$ is $\bbF_2$-valued, it can be recovered from $\calQ$.
As noted above, $\calQ$ necessarily contains $\bfzero$.
We moreover claim that $\calQ\neq\set{\bfzero}$ whenever $J\geq 2$.
Indeed, for any $\bfv_1\in\calV\backslash\set{\bfzero}$,
$\set{\bfv_1}^\perp$ has dimension $2J-1\geq 3$,
and so there exists $\bfv_2\in\set{\bfv_1}^\perp\backslash\set{\bfzero,\bfv_1}$.
Since~\eqref{eq.quadratic symplectic relation} gives
$0
=\rmB(\bfv_1,\bfv_2)
=\rmQ(\bfv_1)+\rmQ(\bfv_2)+\rmQ(\bfv_1+\bfv_2)$,
either exactly one or exactly three of the three distinct vectors $\bfv_1$, $\bfv_2$ and $\bfv_1+\bfv_2$ are singular, giving the claim.

Now recall that if $(\bfv_1,\bfv_2)$ is any symplectic pair in $\calV$ then $\calV$ is the internal direct sum of $\calV_0:=\set{\bfv_1,\bfv_2}^\perp$ and $\calV_0^\perp=\Span\set{\bfv_1,\bfv_2}$,
and also that the restriction $\rmB_0$ of $\rmB$ to $\calV_0$ is a symplectic form.
The restriction $\rmQ_0:\calV_0\rightarrow\bbF_2$ of $\rmQ:\calV\rightarrow\bbF_2$ to $\calV_0$ is thus a quadratic form that yields $\rmB_0$.
Here, writing any $\bfv\in\calV$ as $\bfv=x_1\bfv_1+x_2\bfv_2+\bfv_0$ where $x_1,x_2\in\bbF_2$ and $\bfv_0\in\calV_0$,
\eqref{eq.quadratic form over characteristic two} gives
\begin{align}
\nonumber
\rmQ(x_1\bfv_1+x_2\bfv_2+\bfv_0)
&=\rmQ(x_1\bfv_1+x_2\bfv_2)+\rmB(x_1\bfv_1+x_2\bfv_2,\bfv_0)+\rmQ(\bfv_0)\\
\nonumber
&=x_1^2\rmQ(\bfv_1)+x_1x_2\rmB(\bfv_1,\bfv_2)+x_2^2\rmQ(\bfv_2)+0+\rmQ_0(\bfv_0)\\
\label{eq.quadratic form recursion in general}
&=x_1^2\rmQ(\bfv_1)+x_1x_2+x_2^2\rmQ(\bfv_2)+\rmQ_0(\bfv_0).
\end{align}

A pair $(\bfv_1,\bfv_2)$ of vectors in $\calV$ is \textit{hyperbolic} if $\rmB(\bfv_1,\bfv_2)=1$ and $\rmQ(\bfv_1)=0=\rmQ(\bfv_2)$,
that is, if it is a singular symplectic pair.
We claim that any nonzero $\bfv_1\in\calQ$ belongs to such a pair.
Indeed, for any $\bfv\in\calV\backslash\set{\bfv_1}^\perp$,
we have $\rmB(\bfv_1,\bfv_1+\bfv)=\rmB(\bfv_1,\bfv)=1$ where~\eqref{eq.quadratic symplectic relation} gives
$1=\rmB(\bfv_1,\bfv)=0+\rmQ(\bfv)+\rmQ(\bfv_1+\bfv)$,
implying one member of $\set{\bfv,\bfv_1+\bfv}$ is singular while the other is not.
Letting $\bfv_2$ be the former, the latter is $\bfv_1+\bfv_2$.
Having the claim, we next recall that if $J\geq 2$ then $\calQ\neq\set{\bfzero}$ and so it contains such a pair.
When instead $J=1$, $\calV$ either contains such a pair or it does not,
and so either $\calQ=\set{\bfzero,\bfv_1,\bfv_2}$ or $\calQ=\set{\bfzero}$, respectively.
In the former case, \eqref{eq.quadratic form over characteristic two} gives $\rmQ(x_1\bfv_1+x_2\bfv_2)=x_1x_2$ for any $x_1,x_2\in\bbF_2$.
In the latter case, \eqref{eq.quadratic form over characteristic two} instead gives
$\rmQ(x_1\bfv_1+x_2\bfv_2)=x_1^2+x_1^{}x_2^{}+x_2^2$ for any $x_1,x_2\in\bbF_2$,
where $\bfv_1$ and $\bfv_2$ are any distinct nonzero vectors in $\calV$.
In either case, when $J=1$,
\eqref{eq.two forms on two-dimensional space} implies that the resulting bilinear form is the sole symplectic form on $\calV$, having $\rmB(x_1\bfv_1+x_2\bfv_2,y_1\bfv_1+y_2\bfv_2)
=x_1y_2+x_2y_1$ for any \smash{$x_1,x_2,y_1,y_2\in\bbF_2$},
meaning any distinct nonzero vectors in $\calV$ are nonorthogonal.

If $(\bfv_1,\bfv_2)$ is a hyperbolic pair in $\calV$ then~\eqref{eq.quadratic form recursion in general} simplifies to
\begin{equation}
\label{eq.quadratic form recursion}
\rmQ(x_1\bfv_1+x_2\bfv_2+\bfv_0)
=x_1x_2+\rmQ_0(\bfv_0)
\end{equation}
for any $x_1,x_2\in\bbF_2$ and $\bfv_0\in\calV_0$.
In particular, $\bfv\in\calQ$ if and only if either $x_1x_2=0=\rmQ_0(\bfv_0)$ or $x_1x_2=1=\rmQ(\bfv_0)$, that is,
if and only if either $(x_1,x_2)\in\set{(0,0),(1,0),(0,1)}$ and
$\bfv_0\in\calQ_0:=\set{\bfv_0'\in\calV_0: \rmQ_0(\bfv_0')=0}$
or $(x_1,x_2)=(1,1)$ and $\bfv_0\in\calV_0\backslash\calQ_0$.
Thus, $\calQ$ partitions as
\begin{equation}
\label{eq.quadric relationship}
\calQ=\calQ_0\sqcup(\bfv_1+\calQ_0)\sqcup(\bfv_2+\calQ_0)\sqcup(\bfv_1+\bfv_2+\calV_0\backslash\calQ_0).
\end{equation}
The sizes of $\calQ$ and $\calQ_0$ are thus related by
$\#(\calQ)=3\#(\calQ_0)+[\#(\calV_0)-\#(\calQ_0)]=2^{2J-2}+2\#(\calQ_0)$.
As such, $\#(\calQ)=2^{J-1}(2^J\pm1)$
if and only if $\#(\calQ_0)$ is the analogous value when ``$J$'' is $J-1$, i.e.,
\begin{equation*}
\#(\calQ_0)
=\tfrac12\#(\calQ)-2^{2J-3}
=2^{J-2}(2^J\pm1)-2^{2J-3}
=2^{(J-1)-1}(2^{J-1}\pm1).
\end{equation*}
Moreover, such an equation holds when $J=1$:
here, recall $\calQ$ is either $\set{\bfzero}$ or of the form $\set{\bfzero,\bfv_1,\bfv_2}$ for some hyperbolic pair $(\bfv_1,\bfv_2)$,
and so $\#(\calQ)$ is either $1=2^{1-1}(2^1-1)$ or $3=2^{1-1}(2^1+1)$, respectively.
As such, any quadratic form $\rmQ:\calV\rightarrow\bbF_2$ has a \textit{sign},
denoted $\sign(\rmQ)$, such that
\begin{equation}
\label{eq.sign of quadratic}
\sign(\rmQ)
:=2^{-(J-1)}[\#(\calQ)-2^{2J-1}]\in\set{1,-1},
\qquad
\#(\calQ)=2^{J-1}[2^J+\sign(\rmQ)],
\end{equation}
and it is either $1$ or $-1$ when,
after iteratively choosing $J-1$ hyperbolic pairs from $\calV$,
each lying in the orthogonal complement of those previously chosen,
the restriction of $\rmQ$ to the remaining two-dimensional space has either exactly three singular points or exactly one such point, respectively.
Iteratively choosing hyperbolic pairs in this way for as long as possible yields a symplectic basis~\eqref{eq.symplectic basis} $(\bfv_m)_{m=1}^{2J}$ for $\calV$ such that either $\rmQ(\bfv_m)=0$ for all $m\in[2J]$---a \textit{hyperbolic basis}---or $\rmQ(\bfv_m)=0$ for all $m\in[2J-2]$ but $\rmQ(\bfv_{2J-1})=\rmQ(\bfv_{2J})=1$,
depending on whether $\rmQ$ is \textit{positive} or \textit{negative}, respectively,
that is, whether $\sign(\rmQ)$ is $1$ or $-1$, respectively.
Here, iteratively applying~\eqref{eq.quadratic form recursion} gives $\rmQ(\bfV\bfx)$ is either $\rmQ_{J,+}(\bfx)$ or $\rmQ_{J,-}(\bfx)$ for any $\bfx\in\bbF_2^{2J}$, respectively,
where $\bfV$ is the synthesis map of this basis and $\rmQ_{J,+}$ and $\rmQ_{J,-}$ are the \textit{canonical positive} and \textit{negative} quadratic forms on $\bbF_2^{2J}$, defined as
\begin{equation}
\label{eq.canonical quadratic forms}
\rmQ_{J,+}(\bfx)
:=\sum_{j=1}^{J}\bfx(2j-1)\bfx(2j),
\qquad
\rmQ_{J,-}(\bfx)
:=\rmQ_{J,+}(\bfx)+[\bfx(2J-1)]^2+[\bfx(2J)]^2,
\end{equation}
respectively.
Both of these forms yield the canonical symplectic form~\eqref{eq.canonical symplectic form} on $\bbF_2^{2J}$.
Thus, up to isomorphism, there is a unique symplectic form on $\calV$,
and a unique quadratic form of positive sign,
as well as a unique one of negative sign, that give rise to it.
Now recall that there are exactly $\#(\calV)=2^{2J}$ distinct quadratic forms on $\calV$ that yield $\rmB$,
each corresponding to a particular choice of $\bfv_\bfL$ in~\eqref{eq.other quadratic forms}.
The quadric of such a form $\rmQ_\bfL$ is related to that of $\rmQ$ as follows:
\begin{equation}
\label{eq.quadrics of other quadratic forms}
\calQ_\bfL
:=\set{\bfv\in\calV: \rmQ_\bfL(\bfv)=\rmQ(\bfv+\bfv_\bfL)+\rmQ(\bfv_\bfL)=0}
%=\set{\bfv+\bfv_\bfL\in\calV: \rmQ(\bfv)=\rmQ(\bfv_\bfL)}
=\left\{\begin{array}{ll}
\bfv_\bfL+\calQ,&\ \bfv_\bfL\in\calQ,\\
\bfv_\bfL+\calQ^\rmc,&\ \bfv_\bfL\in\calQ^\rmc.
\end{array}\right.
\end{equation}
In particular, shifting $\calQ$ by a member of itself yields the quadric of a quadratic form with the same sign as that of $\rmQ$,
while shifting it by a member of $\calQ^\rmc$ yields the complement of a quadric of a quadratic form whose sign is the opposite of that of $\rmQ$.
There are thus exactly $2^{J-1}(2^J+1)$ positive quadratic forms on $\calV$ that yield $\rmB$, and exactly $2^{J-1}(2^J-1)$ negative such forms.

\subsection{Doubly transitive equiangular tight frames that arise from quadratic forms over $\bbF_2$}

We now explain the construction of~\cite{FickusIJK21},
and then use ideas from~\cite{Kantor75,Wertheimer86} to show that the resulting ETFs are DTETFs.
By Theorem~1.3 of~\cite{IversonM24}, DTETFs with these parameters are known to exist and are unique up to equivalence,
so the contribution of this subsection is the fact that they can be constructed in the manner of~\cite{FickusIJK21}.
In the next subsection, we use this fact to calculate their binders and moreover relate these binders to those of the corresponding symplectic ETF and its dual.
It turns out that a key concept of this work is how the quadrics of various quadratic forms are related to each other, and to the complements of each other:

\begin{definition}
\label{def.affine quadric}
Let $\rmB$ be a symplectic form on a vector space $\calV$ over $\bbF_2$ of dimension $2J$.
We say that a subset $\calD$ of $\calV$ is an \textit{affine quadric} if either it or its complement is the quadric~\eqref{eq.quadric} of some quadratic form $\rmQ$ that yields $\rmB$.
\end{definition}

Later on, we justify this terminology by verifying that any affine quadric is a shift of a quadric.
By~\eqref{eq.sign of quadratic}, the size of any affine quadric $\calD$ is either $2^{J-1}(2^J+1)$ or $2^{J-1}(2^J-1)$.
In either case, $\calD^\rmc$ is an affine quadric of the other size.
We use this concept to streamline Theorems~4.4 and~4.5 of~\cite{FickusIJK21}:
if $\calD$ is any affine quadric,
then the $\calD\times\calV$, $\calD^\rmc\times\calV$, $\calD\times\calD^\rmc$ and $\calD^\rmc\times\calD$ submatrices of the character table $\bfGamma$ are synthesis maps of ETFs,
and moreover the duals of the latter two are yet other ETFs.
A proof of this reimagining of ideas from~\cite{FickusIJK21} is given in the appendix%
%\ of~\cite{FickusL23}% COMMENT
.

\begin{lemma}[\cite{FickusIJK21}]
\label{lemma.symplectic sub-ETFs}
Let $\rmB$ be a symplectic form on a vector space $\calV$ over $\bbF_2$ of dimension $2J\geq 4$,
and let $\bfGamma\in\bbR^{\calV\times\calV}$, $\bfGamma(\bfv_1,\bfv_2):=(-1)^{\rmB(\bfv_1,\bfv_2)}$ be its corresponding character table.
Three distinct pairs of dual ETFs arise from any affine quadric $\calD$ in $\calV$ (Definition~\ref{def.affine quadric}) with $\#(\calD)=2^{J-1}(2^J+1)$:
\begin{enumerate}
\renewcommand{\labelenumi}{(\alph{enumi})}
\item
$2^{2J}$-vector ETFs for spaces of dimension $2^{J-1}(2^J+1)$ and $2^{J-1}(2^J-1)$:
the finite sequences
$\bfPhi=(\bfphi_\bfv)_{\bfv\in\calV}$ in $\bbR^{\calD}$ and $\bfPsi=(\bfpsi_\bfv)_{\bfv\in\calV}$ in $\bbR^{\calD^\rmc}$ that are defined by
\begin{equation*}
\bfphi_{\bfv_2}(\bfv_1):=\frac1{\sqrt{2^{J-1}}}(-1)^{\rmB(\bfv_1,\bfv_2)},
\quad
\bfpsi_{\bfv_2}(\bfv_1):=\frac1{\sqrt{2^{J-1}}}(-1)^{\rmB(\bfv_1,\bfv_2)},
\end{equation*}
are dual ETFs for $\bbR^{\calD}$ and $\bbR^{\calD^\rmc}$, respectively, that are equivalent to the corresponding instances of the symplectic ETF and its dual of Definition~\ref{def.symplectic ETF}, respectively.\smallskip

\item
$2^{J-1}(2^J-1)$-vector ETFs for spaces of dimension $\tfrac13(2^{2J}-1)$ and $\tfrac13(2^{J-1}-1)(2^J-1)$, namely the subsequence \smash{$\bfPhi_{\calD^\rmc}:=(\bfphi_\bfv)_{\bfv\in\calD^\rmc}$} of $\bfPhi$ and its dual \smash{$\widehat{\bfPhi}_{\calD^\rmc}:=(\widehat{\bfphi}_\bfv)_{\bfv\in\calD^\rmc}$}, respectively, for which
\begin{equation*}
\ip{\bfphi_{\bfv_1}}{\bfphi_{\bfv_2}}
=\left\{\begin{array}{cl}
2^{J}+1,&\ \bfv_1=\bfv_2,\\
(-1)^{\rmB(\bfv_1,\bfv_2)},&\ \bfv_1\neq \bfv_2,
\end{array}\right.
\quad
\ip{\widehat{\bfphi}_{\bfv_1}}{\widehat{\bfphi}_{\bfv_2}}
=\left\{\begin{array}{cl}
2^{J-1}-1,&\ \bfv_1=\bfv_2,\\
-(-1)^{\rmB(\bfv_1,\bfv_2)},&\ \bfv_1\neq \bfv_2.
\end{array}\right.
\end{equation*}

\item
$2^{J-1}(2^J+1)$-vector ETFs for spaces of dimension $\tfrac13(2^{2J}-1)$ and $\tfrac13(2^{J-1}+1)(2^J+1)$,
namely the subsequence \smash{$\bfPsi_\calD:=(\bfpsi_\bfv)_{\bfv\in\calD}$} of $\bfPsi$ and its dual \smash{$\widehat{\bfPsi}_\calD:=(\widehat{\bfpsi}_\bfv)_{\bfv\in\calD}$}, respectively, for which
\begin{equation*}
\ip{\bfpsi_{\bfv_1}}{\bfpsi_{\bfv_2}}
=\left\{\begin{array}{cl}
2^{J}-1,&\ \bfv_1=\bfv_2,\\
-(-1)^{\rmB(\bfv_1,\bfv_2)},&\ \bfv_1\neq \bfv_2,
\end{array}\right.
\quad
\ip{\widehat{\bfpsi}_{\bfv_1}}{\widehat{\bfpsi}_{\bfv_2}}
=\left\{\begin{array}{cl}
2^{J-1}+1,&\ \bfv_1=\bfv_2,\\
(-1)^{\rmB(\bfv_1,\bfv_2)},&\ \bfv_1\neq \bfv_2.
\end{array}\right.
\end{equation*}
\end{enumerate}
\end{lemma}

In general, an ETF is \textit{harmonic} when its synthesis map arises by extracting rows from the character table of a finite abelian group $\calG$.
This occurs if and only if the indices of these rows form a difference set $\calD$ for $\calG$~\cite{Turyn65,Konig99,StrohmerH03,XiaZG05,DingF07},
namely when there exists $\Lambda\geq0$ such that $\#\set{(d_1,d_2)\in\calD\times\calD: g=d_1-d_2}=\Lambda$ for all $g\in\calG\backslash\set{0}$.
This is the case for the ETFs of Lemma~\ref{lemma.symplectic sub-ETFs}(a), for example~\cite{FickusIJK21}.
Many harmonic ETFs have irrational Welch bounds and thus empty binders~\cite{FickusJKM18}.
That said, every harmonic ETF has a group-circulant Gram matrix.
For example, as seen in the proof of Lemma~\ref{lemma.symplectic sub-ETFs},
the finite sequence $(\bfphi_\bfv)_{\bfv\in\calV}$ of (a) has $\ip{\bfphi_{\bfv_1}}{\bfphi_{\bfv_2}}=(-1)^{\rmQ(\bfv_1+\bfv_2)}$ whenever $\bfv_1\neq\bfv_2$, where $\rmQ$ is some quadratic form,
and so $\ip{\bfphi_{\bfv+\bfv_1}}{\bfphi_{\bfv+\bfv_2}}=\ip{\bfphi_{\bfv_1}}{\bfphi_{\bfv_2}}$ for all $\bfv,\bfv_1,\bfv_2\in\calV$.
The symmetry group of a harmonic ETF thus contains all translations.
In particular, the binder of any harmonic ETF is translation-invariant:
when a harmonic ETF contains a regular simplex,
any shift of its indices yields another.
The binders of some harmonic ETFs seemingly consist of all cosets of a single subgroup of the underlying group~\cite{FickusJKM18,FickusS20}.
Such incidence structures are not BIBDs.
It is therefore remarkable that the binder of the harmonic ETF $(\bfpsi_\bfv)_{\bfv\in\calV}$ of Lemma~\ref{lemma.symplectic sub-ETFs}(a) does form a BIBD:
it is equivalent to the dual of the corresponding symplectic ETF,
and so the two share the same binder which,
by Theorem~\ref{theorem.binder of dual symplectic},
forms a BIBD and consists of all affine Lagrangian subspaces of $\calV$.
(As expected, this binder is translation-invariant.)
Apart from these ETFs, the $J=2$ instance of $(\bfphi_\bfv)_{\bfv\in\calV}$ of Lemma~\ref{lemma.symplectic sub-ETFs}(a)---see Theorem~\ref{theorem.binder of symplectic}---and trivial harmonic ETFs that arise from singleton difference sets of their complements, we know of no other harmonic ETFs whose binders happen to form BIBDs.
This raises the question as to whether such objects exist.
We leave this problem for future work.

Since any symplectic ETF and its dual are DTETFs~\cite{IversonM22},
the equivalent ETFs of Lemma~\ref{lemma.symplectic sub-ETFs}(a) are too.
We show that the ETFs of Lemma~\ref{lemma.symplectic sub-ETFs}(b) and (c) are DTETFs as well.
The key idea is a known but obscure~\cite{Kantor75,Wertheimer86} way in which the symplectic group~\eqref{eq.symplectic group} here acts, in a doubly-transitive but not straightforward way, on the members of any affine quadric.
To understand that action, it helps to better understand affine quadrics themselves:

\begin{lemma}
\label{lemma.affine quadric}
Let $\rmB$ be a symplectic form on a vector space $\calV$ over $\bbF_2$ of dimension $2J$.
Then the size of any affine quadric is either $2^{J-1}(2^J+1)$ or $2^{J-1}(2^J-1)$, and there are exactly $2^{2J}$ affine quadrics of each such size,
being the necessarily-distinct shifts of the quadric~\eqref{eq.quadric} of any particular quadratic form $\rmQ$ that satisfies~\eqref{eq.quadratic symplectic relation} and whose sign~\eqref{eq.sign of quadratic} is positive or negative, respectively.
\end{lemma}

\begin{proof}
By Definition~\ref{def.affine quadric}, if $\calD$ is an affine quadric then either $\calD=\calQ$ or $\calD=\calQ^\rmc$ where $\calQ$ is the quadric~\eqref{eq.quadric} of some quadratic form $\rmQ$ that satisfies~\eqref{eq.quadratic symplectic relation}.
By~\eqref{eq.sign of quadratic},
we thus have either $\#(\calD)=\#(\calQ)=2^{J-1}[2^J+\sign(\rmQ)]$ or
$\#(\calD)=\#(\calQ^\rmc)=2^{2J}-\#(\calQ)=2^{J-1}[2^J-\sign(\rmQ)]$.

We claim that every shift of the quadric $\calQ$ of any particular such $\rmQ$ is an affine quadric.
Indeed, for any $\bfv_\bfL\in\calV$,
either $\bfv_\bfL\in\calQ$ or $\bfv_\bfL\in\calQ^\rmc$.
In the former case, \eqref{eq.quadrics of other quadratic forms} gives that
$\bfv_\bfL+\calQ
=\calQ_\bfL$ is the quadric of the quadratic form $\rmQ_\bfL$ given in~\eqref{eq.other quadratic forms}.
In the latter case, \eqref{eq.quadrics of other quadratic forms} instead gives
$\bfv_\bfL+\calQ
=(\bfv_\bfL+\calQ^\rmc)^\rmc=\calQ_\bfL^\rmc$,
giving the claim.
We moreover claim that any two such shifts are necessarily distinct.
Indeed, if $\bfv_\bfL+\calQ=\calQ$ then
$\bfv_\bfL=\bfv_\bfL+\bfzero\in\bfv_\bfL+\calQ=\calQ$,
implying by~\eqref{eq.quadrics of other quadratic forms} that $\calQ_\bfL=\bfv_\bfL+\calQ=\calQ$.
Since any $\bbF_2$-valued quadratic form is uniquely determined by its quadric---mapping members of its quadric to $0$ and all other vectors to $1$---this implies $\rmQ_\bfL=\rmQ$, implying by~\eqref{eq.other quadratic forms} that $\rmB(\bfv_\bfL,\bfv)=0$ for all $\bfv\in\calV$ and thus $\bfv_\bfL=\bfzero$,
giving the claim.

Since $\rmB$ arises from quadratic forms of both possible signs,
there are thus at least $2^{2J}$ distinct affine quadrics of each of the two possible sizes.
To see that these are the only affine quadrics in $\calV$,
now let $\calD$ be any affine quadric whose size equals that of $\calQ$.
Then either $\calD$ or $\calD^\rmc$ is the quadric $\calQ'$ of a quadratic form $\rmQ'$ that satisfies~\eqref{eq.quadratic symplectic relation}.
Since every such quadratic form that satisfies~\eqref{eq.quadratic symplectic relation} is of the form given in~\eqref{eq.other quadratic forms},
we thus have $\rmQ'=\rmQ_\bfL$ for some $\bfv_\bfL\in\calV$.
As such, $\calD$ is either $\calQ_\bfL^{}$ or $\calQ_\bfL^\rmc$ where~\eqref{eq.quadrics of other quadratic forms} gives that $\calQ_\bfL$ is either $\bfv_\bfL+\calQ$ or $\bfv_\bfL+\calQ^\rmc$.
Thus, $\calD$ itself is either $\bfv_\bfL+\calQ$ or $\bfv_\bfL+\calQ^\rmc$.
Since $\calD$ and $\calQ$ have the same size,
this implies $\calD=\bfv_\bfL+\calQ$ is a shift of $\calQ$.
\end{proof}

As we now explain, in this context---where $\rmB$ is a symplectic form on a finite-dimensional vector space $\calV$ over the field $\bbF_2$---the symplectic group $\Sp(\rmB)$ of~\eqref{eq.symplectic group} acts in a doubly transitive way on the set of all quadratic forms of a particular sign that yield $\rmB$, and moreover this action equates to one by this same group on any affine quadric.
To be precise, for any $\bfA\in\Sp(\rmB)$ and any quadratic from $\rmQ$ that yields $\rmB$,
the mapping $\bfA\cdot\rmQ:\calV\rightarrow\bbF_2$,
$(\bfA\cdot\rmQ)(\bfv):=\rmQ(\bfA^{-1}\bfv)$ is itself a quadratic form that yields $\rmB$:
since any $\bfA\in\Sp(\rmB)$ is an invertible linear operator on $\calV$ that preserves $\rmB$,
\begin{align*}
(\bfA\cdot\rmQ)(\bfv_1+\bfv_2)+(\bfA\cdot\rmQ)(\bfv_1)+(\bfA\cdot\rmQ)(\bfv_2)
&=\rmQ(\bfA^{-1}\bfv_1+\bfA^{-1}\bfv_2)+\rmQ(\bfA^{-1}\bfv_1)+\rmQ(\bfA^{-1}\bfv_2)\\
&=\rmQ(\bfA^{-1}(\bfv_1+\bfv_2))+\rmQ(\bfA^{-1}\bfv_1)+\rmQ(\bfA^{-1}\bfv_2)\\
&=\rmB(\bfA^{-1}\bfv_1,\bfA^{-1}\bfv_2)\\
&=\rmB(\bfv_1,\bfv_2),
\end{align*}
for any $\bfv_1,\bfv_2\in\calV$.
It is now straightforward to show that
\begin{equation}
\label{eq.Sp action on quadratic forms}
(\bfA,\rmQ)\mapsto\bfA\cdot\rmQ,
\quad(\bfA\cdot\rmQ)(\bfv):=\rmQ(\bfA^{-1}\bfv),
\end{equation}
defines an action of $\Sp(\rmB)$ on the set of all quadratic forms that yield $\rmB$.
Here, the quadric of $\bfA\cdot\rmQ$ is obtained by applying $\bfA$ to the elements of the quadric $\calQ$ of $\rmQ$:
\begin{equation*}
\bfA\calQ
:=\set{\bfA\bfv: \bfv\in\calQ}
%=\set{\bfA\bfv: \rmQ(\bfv)=0}.
=\set{\bfv'\in\calV: \rmQ(\bfA^{-1}\bfv')=0}
=\set{\bfv'\in\calV: (\bfA\cdot\rmQ)(\bfv')=0}.
\end{equation*}
Since $\bfA$ is invertible and so $\#(\bfA\calQ)=\#(\calQ)$,
this implies that the sign~\eqref{eq.sign of quadratic} of $\bfA\cdot\rmQ$ equals that of $\rmQ$.
As such, the aforementioned action of $\Sp(\rmB)$ on the set of all quadratic forms that yield $\rmB$ is not transitive,
and in fact restricts to an action on the set of all such forms of either given sign.
Remarkably, when restricted in this way, this action becomes doubly transitive.
This fact is briefly mentioned in~\cite{Cameron81}.
See~\cite{Wertheimer86} for an overview of it, its relation to earlier similar results~\cite{Kantor75},
and additional context about quadratic forms over $\bbF_2$ in general.
For the sake of completeness, an elementary, self-contained proof of this fact is given in the appendix%
%\ of~\cite{FickusL23}% COMMENT
:

\begin{lemma}[\cite{Kantor75,Cameron81}]
\label{lemma.DT of Sp}
Let $\rmB$ be a symplectic form on a vector space $\calV$ over $\bbF_2$ of dimension $2J$,
and let $\Sp(\rmB)$ be the corresponding symplectic group~\eqref{eq.symplectic group}.
Then~\eqref{eq.Sp action on quadratic forms} is a well-defined doubly transitive action of $\Sp(\rmB)$ on the set of all quadratic forms $\rmQ$ that yield $\rmB$ of either given sign~\eqref{eq.sign of quadratic}.
\end{lemma}

We exploit this fact in the manner of~\cite{Wertheimer86}.
To explain, let $\calQ$ be the quadric of any quadratic form $\rmQ$ that yields $\rmB$.
For any $\bfA\in\Sp(\rmB)$,
we have that $\bfA\calQ$ is the quadric of the quadratic form $\bfA\cdot\rmQ$ which also yields $\rmB$ and has $\sign(\bfA\cdot\rmQ)=\sign(\rmQ)$.
By~\eqref{eq.other quadratic forms} and~\eqref{eq.quadrics of other quadratic forms},
there thus exists a unique $\bfv_\bfA\in\calQ$ such that $\bfA\calQ=\calQ+\bfv_\bfA$,
that is, such that $\bfA\calQ+\bfv_\bfA=\calQ$.
Since $\rmQ$ is $\bbF_2$-valued and $\bfv\in\bfA\calQ+\bfv_\bfA$ if and only if
$\rmQ(\bfA^{-1}(\bfv+\bfv_\bfA))=0$, this equates to having $\rmQ(\bfA^{-1}(\bfv+\bfv_\bfA))=\rmQ(\bfv)$ for all $\bfv\in\calV$,
that is, to having $\rmQ(\bfv)=\rmQ(\bfA\bfv+\bfv_\bfA)$ for all $\bfv\in\calV$.
As such, the mapping
\begin{equation}
\label{eq.Sp action on vectors}
(\bfA,\bfv)\mapsto \bfA\cdot\bfv:=\bfA\bfv+\bfv_\bfA
\end{equation}
is a bijection from $\calV$ onto itself that restricts to bijections from $\calQ$ and $\calQ^\rmc$ onto themselves.
We claim that the mapping~\eqref{eq.Sp action on vectors} is in fact the action~\eqref{eq.Sp action on quadratic forms} under the bijection $\bfv_\bfL\mapsto\rmQ_\bfL$ from $\calV$ onto the set of all quadratic forms that yield $\rmB$ that is described in~\eqref{eq.other quadratic forms}.
Indeed, under this bijection, any $\bfv_\bfL\in\calV$ maps to the quadratic form $\bfv\mapsto\rmQ_\bfL(\bfv)
:=\rmQ(\bfv)+\rmB(\bfv_\bfL,\bfv)$ for all $\bfv\in\calV$.
Acting on $\rmQ_\bfL$ by $\bfA\in\Sp(\rmB)$ gives
\begin{equation*}
\bfv\mapsto(\bfA\cdot\rmQ_\bfL)(\bfv)
=\rmQ_\bfL(\bfA^{-1}\bfv)
=\rmQ(\bfA^{-1}\bfv)+\rmB(\bfv_\bfL,\bfA^{-1}\bfv).
\end{equation*}
To simplify further,
recall that $\rmQ(\bfA^{-1}\bfv)=\rmQ(\bfv+\bfv_\bfA)$ for all $\bfv\in\calV$, that $\rmB(\bfA\bfv_1,\bfA\bfv_2)=\rmB(\bfv_1,\bfv_2)$ for all $\bfv_1,\bfv_2\in\calV$,
that $\bfv_\bfA\in\calQ$, and moreover recall~\eqref{eq.quadratic symplectic relation}:
\begin{align*}
\rmQ(\bfA^{-1}\bfv)+\rmB(\bfv_\bfL,\bfA^{-1}\bfv)
&=\rmQ(\bfv+\bfv_\bfA)+\rmB(\bfA\bfv_\bfL,\bfv)\\
&=\rmB(\bfv,\bfv_\bfA)+\rmQ(\bfv)+\rmQ(\bfv_\bfA)+\rmB(\bfA\bfv_\bfL,\bfv)\\
&=\rmQ(\bfv)+\rmB(\bfA\bfv_\bfL+\bfv_\bfA,\bfv).
\end{align*}
Pulling $\bfv\mapsto(\bfA\cdot\rmQ_\bfL)(\bfv)=\rmQ(\bfv)+\rmB(\bfA\bfv_\bfL+\bfv_\bfA,\bfv)$
back through the bijection $\bfv_\bfL\mapsto\rmQ_\bfL$ gives $\bfA\cdot\bfv_\bfL=\bfA\bfv_\bfL+\bfv_\bfA$,
namely the result of acting on $\bfv_\bfL$ by $\bfA$ via~\eqref{eq.Sp action on vectors}, as claimed.
Since~\eqref{eq.Sp action on quadratic forms} is a group action,
the equivalent mapping~\eqref{eq.Sp action on vectors} is as well.
Moreover, further recalling that~\eqref{eq.sign of quadratic} and~\eqref{eq.quadrics of other quadratic forms} give $\sign(\rmQ_\bfL)=\sign(\rmQ)$ if and only if $\bfv_\bfL\in\calQ$,
the restrictions of the action~\eqref{eq.Sp action on vectors} to $\calQ$ and $\calQ^\rmc$ equate to the restrictions of~\eqref{eq.Sp action on quadratic forms} to the sets of quadratic forms on $\calV$ that yield $\rmB$ whose signs equal $\sign(\rmQ)$ and $-\sign(\rmQ)$, respectively.
Since Lemma~\ref{lemma.DT of Sp} gives that the latter are doubly transitive,
the former are as well.
We summarize these facts from~\cite{Wertheimer86} as follows:

\begin{lemma}[\cite{Wertheimer86}]
\label{lemma.DT of Sp on quadric}
Let $\rmB$ be a symplectic form on a vector space $\calV$ over $\bbF_2$ of dimension $2J$,
and let $\Sp(\rmB)$ be the corresponding symplectic group~\eqref{eq.symplectic group}.
Then letting $\calQ$ be the quadric of any quadratic form $\rmQ$ on $\calV$ that yields $\rmB$, for any $\bfA\in\Sp(\rmB)$ there exists a unique $\bfv_\bfA\in\calQ$ such that
\begin{equation}
\label{eq.symplectic of quadric is shift}
\bfA\calQ+\bfv_\bfA=\calQ,
\qquad\text{i.e.,}\qquad
\rmQ(\bfv)=\rmQ(\bfA\bfv+\bfv_\bfA),\quad\forall\,\bfv\in\calV.
\end{equation}
Moreover, \eqref{eq.Sp action on vectors} defines an action of $\Sp(\rmB)$ on $\calV$ whose restrictions to $\calQ$ and $\calQ^\rmc$ are doubly transitive.
\end{lemma}

These facts provide a quick proof of the following result:

\begin{theorem}
\label{theorem.sub-ETFs are DT}
The ETFs of Lemma~\ref{lemma.symplectic sub-ETFs}(b) and (c) are DTETFs.
\end{theorem}

\begin{proof}
By Definition~\ref{def.affine quadric},
there exists a quadratic form $\rmQ$ that yields $\rmB$ whose quadric $\calQ$ is either $\calD$ or $\calD^\rmc$.
By Lemma~\ref{lemma.DT of Sp on quadric},
for every $\bfA\in\Sp(\rmB)$ there exists $\bfv_\bfA\in\calQ$ such that~\eqref{eq.symplectic of quadric is shift} holds,
and moreover~\eqref{eq.Sp action on vectors} defines a doubly transitive action on both $\calQ$ and $\calQ^\rmc$,
that is, on both $\calD$ and $\calD^\rmc$.
Identifying any $\bfA\in\Sp(\rmB)$ with the restriction of the mapping $\bfv\mapsto\bfA\cdot\bfv=\bfA\bfv+\bfv_\bfA$ to $\calD$ and $\calD^\rmc$ thus yields subgroups of $\Sym(\calD)$ and $\Sym(\calD^\rmc)$ that are isomorphic to $\Sp(\rmB)$ and whose natural actions on $\calD$ and $\calD^\rmc$, respectively, are doubly transitive.
They are the \textit{affine symplectic groups} of~\cite{Wertheimer86}.
Since the ETFs of Lemma~\ref{lemma.symplectic sub-ETFs}(c) and (b) are indexed by $\calD$ and $\calD^\rmc$ respectively,
it thus suffices to show that the members of these subgroups of $\Sym(\calD)$ and $\Sym(\calD^\rmc)$ belong to the symmetry groups of these ETFs, respectively.
To show this,
note that for any $\bfA\in\Sp(\rmB)$,
and any $\bfv_1,\bfv_2\in\calV$ that both lie in either $\calD$ or $\calD^\rmc$,
\begin{align*}
\rmB(\bfA\cdot\bfv_1,\bfA\cdot\bfv_2)
=\rmB(\bfA\bfv_1+\bfv_\bfA,\bfA\bfv_2+\bfv_\bfA)
=\rmB(\bfv_1,\bfv_2)+\rmB(\bfA\bfv_1,\bfv_\bfA)+\rmB(\bfv_\bfA,\bfA\bfv_2)
\end{align*}
and so defining unimodular scalars $(z_\bfv)_{\bfv\in\calV}$ by
$z_\bfv:=(-1)^{\rmB(\bfv_\bfA,\bfA\bfv)}$ for all $\bfv\in\calV$, we have
\begin{equation*}
(-1)^{\rmB(\bfA\cdot\bfv_1,\bfA\cdot\bfv_2)}
=(-1)^{\rmB(\bfA\bfv_1,\bfv_\bfA)+\rmB(\bfv_\bfA,\bfA\bfv_2)+\rmB(\bfv_1,\bfv_2)}
=\overline{z_{\bfv_1}}z_{\bfv_2}(-1)^{\rmB(\bfv_1,\bfv_2)}.
\end{equation*}
When combined with Lemma~\ref{lemma.symplectic sub-ETFs}(c) and (b),
this gives
$\ip{\bfpsi_{\bfA\cdot\bfv_1}}{\bfpsi_{\bfA\cdot\bfv_2}}
=\overline{z_{\bfv_1}}z_{\bfv_2}\ip{\bfpsi_{\bfv_1}}{\bfpsi_{\bfv_2}}$
for all $\bfv_1,\bfv_2\in\calD$, and
$\ip{\bfphi_{\bfA\cdot\bfv_1}}{\bfphi_{\bfA\cdot\bfv_2}}
=\overline{z_{\bfv_1}}z_{\bfv_2}\ip{\bfphi_{\bfv_1}}{\bfphi_{\bfv_2}}$
for all $\bfv_1,\bfv_2\in\calD^\rmc$, respectively.
\end{proof}

Since the ETFs of Lemma~\ref{lemma.symplectic sub-ETFs}(b) and (c) are DTETFs,
and since DTETFs with their parameters are unique up to equivalence~\cite{IversonM24},
their symmetry groups are known to be isomorphic to $\Sp(\rmB)$~\cite{IversonM24}.
Their symmetry groups are thus the affine symplectic groups of~\cite{Wertheimer86} that were discussed in the proof of Theorem~\ref{theorem.sub-ETFs are DT} above.
That is, no other symmetries of them remain to be discovered.

\subsection{The binders of the doubly transitive equiangular tight frames of Lemma~\ref{lemma.symplectic sub-ETFs}}

By Theorem~\ref{theorem.sub-ETFs are DT},
the ETFs of Lemma~\ref{lemma.symplectic sub-ETFs}(b) and (c) are DTETFs.
By Theorem~\ref{theorem.binders of DTETFs}, the binders of these ETFs are thus either empty or form BIBDs.
For any ETF that belongs to one of these four families,
we now determine which of these two phenomena occur,
and moreover the parameters of any BIBD that does arise in this fashion.
As we shall see, the binders of $\bfPhi_{\calD^\rmc}$ and
\smash{$\widehat{\bfPsi}_\calD$} are empty except in a finite number of cases.
In contrast, the binders of their respective duals
\smash{$\widehat{\bfPhi}_{\calD^\rmc}$} and $\bfPsi_\calD$ always form BIBDs, and these BIBDs are related to the binder of $\bfPsi$---the set of all affine Lagrangian subspaces of $\calV$---in nontrivial ways.

Recall that in general, we say that a finite sequence $(\bfphi_n)_{n\in\calK}$ of vectors in a Hilbert space is a regular simplex if it is an ETF for a space of dimension $\#(\calK)-1$.
Notably, whether this occurs is independent of what if any set $\calN$ we regard to be a superset of $\calK$.
As such, if $\bfPhi_0=(\bfphi_n)_{n\in\calN_0}$ is an ETF that is a proper subsequence of another ETF $\bfPhi=(\bfphi_n)_{n\in\calN}$---as is the case for the ETFs $(\bfphi_\bfv)_{\bfv\in\calD^\rmc}$ and $(\bfpsi_\bfv)_{\bfv\in\calD}$ of Lemma~\ref{lemma.symplectic sub-ETFs}---then a subset $\calK$ of $\calN$ belongs to
$\bin(\bfPhi_0)$ if and only if it both belongs to $\bin(\bfPhi)$ and is contained in $\calN_0$.
In particular, $\bin(\bfPhi_0)\subseteq\bin(\bfPhi)$, and if $\bfPhi$ has an empty binder then $\bfPhi_0$ does as well.
Combining this observation with Theorem~\ref{theorem.binder of symplectic} makes short work of the problem of calculating the binders of some of the ETFs of Lemma~\ref{lemma.symplectic sub-ETFs}:

\begin{theorem}
\label{theorem.binder of sub-ETF of symplectic}
Under the assumptions of Lemma~\ref{lemma.symplectic sub-ETFs},
$\bin(\bfPhi_{\calD^\rmc})$ is empty unless $J=2$.
When $J=2$, $\bfPhi_{\calD^\rmc}$ is a $6$-vector regular simplex and so
$\bin(\bfPhi_{\calD^\rmc})=\set{\calD^\rmc}$,
while $\bin(\bfPhi)$ consists of all shifts of $\calD^\rmc$, which forms a BIBD with parameters $(V,K,\Lambda,R,B)=(16,6,2,6,16)$.
\end{theorem}

\begin{proof}
By Lemma~\ref{lemma.symplectic sub-ETFs},
$J\geq 2$ and $\bfPhi$ is equivalent to the symplectic ETF that arises as the $P=2$ case of Definition~\ref{def.symplectic ETF},
and so these two ETFs have the same binder.
When $J>2$, Theorem~\ref{theorem.binder of symplectic} gives that $\bfPhi$ has an empty binder and so the same is true for its sub-ETF $\bfPhi_{\calD^\rmc}$.
So let $J=2$ for the remainder of this proof.
In that context,
Lemma~\ref{lemma.symplectic sub-ETFs}(b) gives that
$\bfPhi_{\calD^\rmc}$ is a $6$-vector ETF for a space of dimension $5$,
and so is a regular simplex.
As detailed in Example~\ref{example.binder of simplex},
its binder thus consists of the single set $\calD^\rmc$.
Here, Theorem~\ref{theorem.binder of symplectic} gives that $\bin(\bfPhi)$ forms a BIBD and consists of all shifts of every set of the form~\eqref{eq.proto binder of 10x16} where $(\bfv_1,\bfv_2,\bfv_3,\bfv_4)$ is any symplectic basis for $\calV$.
Fix any such basis, let $\bfV$ be its synthesis map,
and consider the quadratic form $\rmQ=\rmQ_{2,-}\circ\bfV^{-1}$
obtained by composing the canonical quadratic form on $\bbF_2^4$ of negative type~\eqref{eq.canonical quadratic forms} with $\bfV^{-1}$, and so
\begin{equation*}
\rmQ(x_1\bfv_1+x_2\bfv_2+x_3\bfv_3+x_4\bfv_4)
=x_1^{}x_2^{}+x_3^{}x_4^{}+x_3^2+x_4^2,
\quad\forall\,x_1,x_2,x_3,x_4\in\bbF_2.
\end{equation*}
The quadric $\calQ$ of $\rmQ$ is the set given in~\eqref{eq.proto binder of 10x16}.
Since any symplectic basis for $\calV$ is the image of $(\bfv_1,\bfv_2,\bfv_3,\bfv_4)$ under some $\bfA\in\Sp(\rmB)$, every set of the form~\eqref{eq.proto binder of 10x16} can be expressed as
\begin{equation*}
\set{\bfA\bfzero,
\bfA\bfv_1,
\bfA\bfv_2,
\bfA\bfv_1+\bfA\bfv_2+\bfA\bfv_3,
\bfA\bfv_1+\bfA\bfv_2+\bfA\bfv_4,
\bfA\bfv_1+\bfA\bfv_2+\bfA\bfv_3+\bfA\bfv_4}\\
=\bfA\calQ.
\end{equation*}
Since Lemma~\ref{lemma.DT of Sp on quadric} moreover gives
$\bfA\calQ=\calQ+\bfv_\bfA$, every such set is a shift of $\calQ$.
Thus, $\bin(\bfPhi)$ consists of all shifts of sets which are themselves shifts of $\calQ$, and is thus $\set{\bfv+\calQ: \bfv\in\calV}$.
As noted in Lemma~\ref{lemma.affine quadric},
these shifts are distinct,
and so $\bin(\bfPsi)$ forms a BIBD with $V=16$, $K=6$ and $B=16$,
which in turn implies via~\eqref{eq.BIBD parameter relations} that
$R=6$ and $\Lambda=2$.
Since $\bfPhi_{\calD^\rmc}$ is a regular simplex,
$\calD^\rmc$ is one of these shifts of $\calQ$,
and so $\bin(\bfPhi)$ in fact consists of all shifts of $\calD^\rmc$.
\end{proof}

To calculate the binders of \smash{$\widehat{\bfPhi}_{\calD^\rmc}$} and $\bfPsi_\calD$ we need some more facts from finite symplectic geometry.
A subset $\calS$ of $\calV$ is said to be \textit{totally singular} with respect to some quadratic form $\rmQ$ that yields $\rmB$ when $\rmQ(\bfv)=0$ for all $\bfv\in\calS$,
that is, when $\calS$ is a subset of the corresponding quadric $\calQ$.
Any totally singular subspace $\calS$ is totally orthogonal:
for any $\bfv_1,\bfv_2\in\calS$, combining~\eqref{eq.quadratic symplectic relation} with the fact that $\bfv_1+\bfv_2\in\calS$ gives
$\rmB(\bfv_1,\bfv_2)=\rmQ(\bfv_1+\bfv_2)+\rmQ(\bfv_1)+\rmQ(\bfv_2)=0+0+0=0$.
We caution however that, in contrast with Lemma~\ref{lemma.totally orthogonal}(a),
a totally singular set might contain a symplectic pair, and when it does, its span is not totally singular.
Some partial converses to this fact---some from~\cite{Wertheimer86} and others folklore---are given in the following result, which we reprove in the appendix%
%\ of~\cite{FickusL23}% COMMENT
.

\begin{lemma}[\cite{Wertheimer86}]
\label{lemma.totally singular}
Let $\rmB$ be a symplectic form on a vector space $\calV$ over $\bbF_2$ of dimension $2J$, and let $\calQ$ be the quadric~\eqref{eq.quadric} of a quadratic form $\rmQ$ that yields $\rmB$.
\begin{enumerate}
\renewcommand{\labelenumi}{(\alph{enumi})}
\item
If $\calS$ is a totally orthogonal subspace of $\calV$ then $\calS\cap\calQ$ is a totally singular subspace.
Here if $\calS\nsubseteq\calQ$ then $\calS\cap\calQ$ is a subgroup of $\calS$ of index $2$, and $\calS\backslash\calQ$ is its nonidentity coset.\smallskip

\item
If $\sign(\rmQ)<0$ then no Lagrangian subspace of $\calV$ is totally singular,
and any totally singular subspace of $\calV$ of dimension $J-1$ is contained in exactly three Lagrangian subspaces of $\calV$.\smallskip

\item
If $\sign(\rmQ)>0$ then every Lagrangian subspace of $\calV$ has a unique coset that is totally singular,
and any totally singular subspace of $\calV$ of dimension $J-1$ is contained in exactly three Lagrangian subspaces of $\calV$.
Two of these three subspaces are totally singular while the third is not.
\end{enumerate}
\end{lemma}

As we now explain,
these facts imply that in the context of Lemma~\ref{lemma.symplectic sub-ETFs},
the BIBD formed by the binder of the dual $\bfPsi$ of the symplectic-equivalent ETF $\bfPhi$, namely the set of all affine Lagrangian subspaces of $\calV$, decomposes into three BIBDs,
including one whose blocks form the binder of its sub-ETF $\bfPsi_\calD$,
as well as another---with twice-repeated blocks---whose distinct blocks form the binder of the dual \smash{$\widehat{\bfPhi}_{\calD^\rmc}$} of the sub-ETF $\bfPhi_{\calD^\rmc}$ of $\bfPhi$.
Though these BIBDs are not new~\cite{Wertheimer86},
the relationships identified here between them and these ETFs seem to be novel.

\begin{theorem}
\label{theorem.binders of sub-ETFs}
Under the assumptions of Lemma~\ref{lemma.symplectic sub-ETFs},
$\bin(\bfPsi)
=\set{\bfv+\calL: \bfv\in\calV,\,\calL\subseteq\calV,\,\calL^\perp=\calL}$ is the set of all affine Lagrangian subspaces of $\calV$, which forms a BIBD with $(V,K,\Lambda,R,B)$ parameters
\begin{equation}
\label{eq.BIBD parameters of binder of dual symplectic}
2^{2J},\qquad
2^J,\qquad
\prod_{j=1}^{J-1}(2^j+1),\qquad
\prod_{j=1}^{J}(2^j+1),\qquad
2^J\prod_{j=1}^{J}(2^j+1),
\end{equation}
respectively.
Moreover,
any Lagrangian subspace here has a unique coset that lies in $\calD$, and so $\bin(\bfPsi_\calD)=\set{\bfv+\calL\in\bin(\bfPsi):\bfv+\calL\subseteq\calD}$ forms a BIBD with $(\calV_\calD,K_\calD,\Lambda_\calD,R_\calD,B_\calD)$ parameters
\begin{equation}
\label{eq.BIBD parameters of binder of sub-ETF of dual symplectic}
2^{J-1}(2^J+1),\qquad
2^J,\qquad
\prod_{j=0}^{J-2}(2^j+1),\qquad
\prod_{j=0}^{J-1}(2^j+1),\qquad
\prod_{j=1}^{J}(2^j+1),
\end{equation}
respectively.
The $\calD^\rmc$-components of the affine Lagrangian subspaces of $\calV$ that do not lie in $\calD$, i.e.,
\begin{equation}
\label{eq.D complement part}
(\bin(\bfPsi)\backslash\bin(\bfPsi_\calD))\cap\calD^\rmc
=\set{(\bfv+\calL)\cap\calD^\rmc:
\bfv+\calL\in\bin(\bfPsi),\,\bfv+\calL\nsubseteq\calD},
\end{equation}
form a BIBD with $(V_{\calD^\rmc},K_{\calD^\rmc},\Lambda_{\calD^\rmc},R_{\calD^\rmc},B_{\calD^\rmc})$ parameters
\begin{equation}
\label{eq.BIBD parameters of binder of dual of sub-ETF of symplectic}
2^{J-1}(2^J-1),\qquad
2^{J-1},\qquad
\prod_{j=2}^{J-1}(2^j+1),\qquad
\prod_{j=2}^{J}(2^j+1),\qquad
(2^J-1)\prod_{j=2}^{J}(2^j+1),
\end{equation}
respectively.
This set~\eqref{eq.D complement part} is $\bin(\widehat{\bfPhi}_{\calD^\rmc})$, and each member of it extends to exactly three members of $\bin(\bfPsi)\backslash\bin(\bfPsi_\calD)$.
Meanwhile, the $\calD$-components of these affine Lagrangian subspaces, i.e.,
\begin{equation}
\label{eq.D part}
(\bin(\bfPsi)\backslash\bin(\bfPsi_\calD))\cap\calD
=\set{(\bfv+\calL)\cap\calD:
\bfv+\calL\in\bin(\bfPsi),\,\bfv+\calL\nsubseteq\calD},
\end{equation}
form a BIBD over the vertex set $\calD$ whose parameters $(V_0,K_0,\Lambda_0,R_0,B_0)$ are
\begin{equation}
\label{eq.BIBD parameters of last BIBD}
2^{J-1}(2^J+1),\quad
2^{J-1},\quad
(2^{J-1}-1)\prod_{j=1}^{J-2}(2^j+1),\quad
(2^J-1)\prod_{j=1}^{J-1}(2^j+1),\quad
(2^J-1)\prod_{j=1}^{J}(2^j+1),
\end{equation}
respectively.
The incidence matrices $\bfX$, $\bfX_\calD$, $\bfX_{\calD^\rmc}$ and $\bfX_0$ of these four respective BIBDs satisfy
\begin{equation}
\label{eq.incidence relation}
\bfX=\left[\begin{array}{ccc}
\bfX_{\calD^\rmc}\otimes\left[\begin{smallmatrix}1\\1\\1\end{smallmatrix}\right]
&\ &\bfX_0\smallskip\\
\bfzero&\ &\bfX_\calD
\end{array}\right],
\end{equation}
provided its columns are indexed in ``$\calD^\rmc$ then $\calD$'' order,
while its rows are arranged according to the intersection of their blocks with $\calD^\rmc$,
beginning with those with nontrivial such intersections.
\end{theorem}

\begin{proof}
By Lemma~\ref{lemma.symplectic sub-ETFs}(a),
$\bfPsi$ is equivalent to the dual of the symplectic ETF that arises as the $P=2$ case of Definition~\ref{def.symplectic ETF},
and so these two ETFs have the same binder.
By Theorem~\ref{theorem.binder of dual symplectic}, this binder is the set of all affine Lagrangian subspaces of $\calV$,
which forms a BIBD whose parameters are given by~\eqref{eq.BIBD parameters of binder of dual symplectic}.
The binder of its sub-ETF $\bfPsi_\calD$ thus indeed consists of the members of $\bin(\bfPhi)$ that lie (completely) in $\calD$.
Since $\calD$ is an affine quadric of cardinality $2^{J-1}(2^J+1)$,
Lemma~\ref{lemma.affine quadric} gives that it is of the form $\bfv_\calD+\calQ_+$ where $\calQ_+$ is the quadric of some quadratic form $\rmQ_+$ that yields $\rmB$ with $\sign(\rmQ_+)>0$.
(In fact, such a $\calQ_+$ and $\rmQ_+$ exist for any of the $2^{J-1}(2^J+1)$ choices of $\bfv_\calD$ in $\calD$.)
As such, for any Lagrangian subspace $\calL$ of $\calV$,
the cosets of it that lie in $\calD$ are in one-to-one correspondence with the cosets of it that lie in $\calQ_+$.
Since Lemma~\ref{lemma.totally singular}(c) gives that exactly one coset of the latter type exists, exactly one of the former type exists as well.
In particular, the number $B_\calD$ of members of $\bin(\bfPsi_\calD)$ is the number of Lagrangian subspaces of $\calV$.
This is the number of affine Lagrangian subspaces of $\calV$ that contain $\bfzero$,
namely the ``$R$'' parameter of~\eqref{eq.BIBD parameters of binder of dual symplectic}.
Since $\bfPsi_\calD$ is a DTETF by Theorem~\ref{theorem.sub-ETFs are DT},
the fact that its binder is nonempty implies by Theorem~\ref{theorem.binders of DTETFs} that it forms a BIBD.
Here the vertex set is the index set $\calD$ of the ETF $\bfPsi_\calD$ which has size $V_\calD=2^{J-1}(2^J+1)$.
Since any member of $\bin(\bfPsi_\calD)$ is also a member of $\bin(\bfPsi)$,
the ``$K$'' parameters of these two BIBDs are identical,
and so $\calK_\calD=2^J$.
The ``$R$'' and ``$\Lambda$'' parameters of this BIBD then follow from~\eqref{eq.BIBD parameter relations}.
Indeed,
\begin{equation*}
R_\calD
=\tfrac{K_\calD}{V_\calD}B_\calD
=\tfrac{2^J}{2^{J-1}(2^J+1)}\prod_{j=1}^{J}(2^j+1)
=\prod_{j=0}^{J-1}(2^j+1).
\end{equation*}
Since $V_\calD-1=2^{J-1}(2^J+1)-1=(2^{J-1}+1)(2^J-1)$ and $J\geq 2$,
\begin{equation*}
\Lambda_\calD
=\tfrac{K_\calD-1}{V_\calD-1}R_\calD
=\tfrac{2^J-1}{(2^{J-1}+1)(2^J-1)}\prod_{j=0}^{J-1}(2^j+1)
=\prod_{j=0}^{J-2}(2^j+1).
\end{equation*}

Now consider the set of affine Lagrangian subspaces of $\calL$ that are not contained in $\calD$, namely
\begin{equation}
\label{eq.pf of binders of sub-ETFs 1}
\bin(\bfPsi)\backslash\bin(\bfPsi_\calD)
=\set{\ \bfv+\calL\ :\ \bfv\in\calV,\ \calL\subseteq\calV,\ \calL^\perp=\calL,\ \bfv+\calL\nsubseteq\calD\ }.
\end{equation}
The number of such subsets is
\begin{equation}
\label{eq.pf of binders of sub-ETFs 2}
\#(\bin(\bfPsi)\backslash\bin(\bfPsi_\calD))
=B-B_\calD
=(2^J-1)\prod_{j=1}^{J}(2^j+1).
\end{equation}
We now claim that any member $\bfv+\calL$ of $\bin(\bfPsi)\backslash\bin(\bfPsi_\calD)$ is split perfectly in half by $\calD$, having $\#((\bfv+\calL)\cap\calD^\rmc)=2^{J-1}$.
To see this, for any such set $\bfv+\calL$,
take $\bfv_{\calD^\rmc}\in(\bfv+\calL)\cap\calD^\rmc$, and note
\begin{equation}
\label{eq.pf of binders of sub-ETFs 3}
\bfzero
=\bfv_{\calD^\rmc}+\bfv_{\calD^\rmc}
\in\bfv_{\calD^\rmc}+[(\bfv+\calL)\cap\calD^\rmc]
=(\bfv_{\calD^\rmc}+\bfv+\calL)\cap(\bfv_{\calD^\rmc}+\calD^\rmc).
\end{equation}
Here since $\bfzero\in\bfv_{\calD^\rmc}+\bfv+\calL$ where $\calL$ is a subspace of $\calV$,
$\bfv_{\calD^\rmc}+\bfv+\calL=\calL$.
Moreover, since $\calD$ is an affine quadric of size $2^{J-1}(2^J+1)$,
Definition~\ref{def.affine quadric} gives that $\calD^\rmc$ is an affine quadric of size $2^{J-1}(2^J-1)$.
By Lemma~\ref{lemma.affine quadric}, $\bfv_{\calD^\rmc}+\calD^\rmc$ is thus also an affine quadric of size $2^{J-1}(2^J-1)$.
Since~\eqref{eq.pf of binders of sub-ETFs 3} gives that $\bfv_{\calD^\rmc}+\calD^\rmc$ contains $\bfzero$, Definition~\ref{def.affine quadric} implies that there exists a quadratic form $\rmQ_-$ that yields $\rmB$ with $\sign(\rmQ_-)<0$ whose quadric is $\calQ_-=\bfv_{\calD^\rmc}+\calD^\rmc$.
Lemma~\ref{lemma.totally singular}(a) then gives the claim:
\begin{equation*}
\#((\bfv+\calL)\cap\calD^\rmc)
=\#((\bfv_{\calD^\rmc}+\bfv+\calL)\cap(\bfv_{\calD^\rmc}+\calD^\rmc))
=\#(\calL\cap\calQ_-)
=\tfrac12\#(\calL)
=2^{J-1}.
\end{equation*}
Now consider the set~\eqref{eq.D complement part} of $2^{J-1}$-element subsets of $\calV$ that arise by taking intersections of the members of~\eqref{eq.pf of binders of sub-ETFs 1} with $\calD^\rmc$.
We continue the above argument to further show that the natural surjection from $\bin(\bfPsi)\backslash\bin(\bfPsi_\calD)$ onto this set---that which intersects a given set with $\calD^\rmc$---is three-to-one.
To see this, note that we have already written an arbitrary member of~\eqref{eq.D complement part} as
\begin{equation*}
(\bfv+\calL)\cap\calD^\rmc
=(\bfv_{\calD^\rmc}+\calL)\cap\calD^\rmc
=\bfv_{\calD^\rmc}+\calL\cap(\bfv_{\calD^\rmc}+\calD^\rmc)
=\bfv_{\calD^\rmc}+\calL\cap\calQ_-
=\bfv_{\calD^\rmc}+\calS,
\end{equation*}
where $\calS:=\calL\cap\calQ_-$ is some $(J-1)$-dimensional subspace of $\calV$ that is totally singular with respect to $\rmQ_-$.
By Lemma~\ref{lemma.totally singular}(b),
there are exactly three Lagrangian subspaces of $\calV$ that contain $\calS$.
Call these subspaces $\calL_1$, $\calL_2$ and $\calL_3$.
If $\bfv'+\calL'$ is any affine Lagrangian subspace of $\calV$ such that $(\bfv'+\calL')\cap\calD^\rmc=\bfv_{\calD^\rmc}+\calS$ then
\smash{$\bfv_{\calD^\rmc}
=\bfv_{\calD^\rmc}+\bfzero
\in\bfv_{\calD^\rmc}+\calS
=(\bfv'+\calL')\cap\calD^\rmc
\subseteq(\bfv'+\calL')$}
implying \smash{$\bfv'+\calL'=\bfv_{\calD^\rmc}+\calL'$} and so
\begin{equation*}
\calS
=\bfv_{\calD^\rmc}+\bfv_{\calD^\rmc}+\calS
=\bfv_{\calD^\rmc}+(\bfv'+\calL')\cap\calD^\rmc
=\calL'\cap(\bfv_{\calD^\rmc}+\calD^\rmc)
=\calL'\cap\calQ_-
\subseteq\calL'.
\end{equation*}
Thus, any such $\bfv'+\calL'$ is necessarily of the form $\bfv'+\calL'=\bfv_{\calD^\rmc}+\calL'$ where $\calL'$ is $\calL_1$, $\calL_2$ or $\calL_3$.
Conversely, if $\bfv+\calL=\bfv_{\calD^\rmc}+\calL_i$ where $i\in\set{1,2,3}$ then $\calS\subseteq\calL_i\cap\calQ_-$ where, by Lemma~\ref{lemma.totally singular}(a) and (b), $\dim(\calL_i\cap\calQ_-)=J-1=\dim(\calS)$,
implying $\calS=\calL_i\cap\calQ_-$ and so
\begin{equation*}
(\bfv+\calL)\cap\calD^\rmc
=(\bfv_{\calD^\rmc}+\calL_i)\cap(\bfv_{\calD^\rmc}+\calQ_-)
=\bfv_{\calD^\rmc}+(\calL_i\cap\calQ_-)
=\bfv_{\calD^\rmc}+\calS.
\end{equation*}

Now note that since $\bin(\bfPsi)$ forms a BIBD,
any distinct $\bfv_1,\bfv_2\in\calD^\rmc$ are contained in exactly $\Lambda$ members of $\bin(\bfPsi)$.
In fact, since every member of $\bin(\bfPsi_\calD)$ is a subset of $\calD$,
any such $\bfv_1,\bfv_2$ are contained in exactly $\Lambda$ members of $\bin(\bfPsi)\backslash\bin(\bfPsi_\calD)$.
Since the natural surjection from $\bin(\bfPsi)\backslash\bin(\bfPsi_\calD)$ onto \eqref{eq.D complement part} is three-to-one,
any distinct $\bfv_1,\bfv_2\in\calD^\rmc$ are contained in exactly $\frac13\Lambda$ members of~\eqref{eq.D complement part}.
Since every member of \eqref{eq.D complement part} is a $2^{J-1}$-element subset of the $2^{J-1}(2^J-1)$-element set $\calD^\rmc$,
this implies that~\eqref{eq.D complement part} forms a BIBD with
``$V$,'' ``$K$'' and ``$\Lambda$'' parameters of
\begin{equation*}
V_{\calD^\rmc}=2^{J-1}(2^J-1),\quad
K_{\calD^\rmc}=2^{J-1},\quad
\Lambda_{\calD^\rmc}
=\tfrac13\Lambda
=\tfrac1{2^1+1}\prod_{j=1}^{J-1}(2^j+1)
=\prod_{j=2}^{J-1}(2^j+1),
\end{equation*}
respectively.
When $J=2$, the empty product in this expression for $\Lambda_{\calD^\rmc}$ should be interpreted as $1$, as is typical.
Recalling~\eqref{eq.BIBD parameter relations} and writing
$V_{\calD^\rmc}-1=2^{J-1}(2^J-1)-1=(2^{J-1}-1)(2^J+1)$ gives what remains of~\eqref{eq.BIBD parameters of binder of dual of sub-ETF of symplectic}:
\begin{align*}
R_{\calD^\rmc}
&=\tfrac{V_{\calD^\rmc}-1}{K_{\calD^\rmc}-1}\Lambda_{\calD^\rmc}
=\tfrac{(2^{J-1}-1)(2^J+1)}{2^{J-1}-1}\prod_{j=2}^{J-1}(2^j+1)
=\prod_{j=2}^{J}(2^j+1),\\
B_{\calD^\rmc}
&=\tfrac{V_{\calD^\rmc}}{K_{\calD^\rmc}}R_{\calD^\rmc}
=\tfrac{2^{J-1}(2^J-1)}{2^{J-1}}\prod_{j=2}^{J}(2^j+1)
=(2^J-1)\prod_{j=2}^{J}(2^j+1).
\end{align*}
(Alternatively, one can argue that $R_{\calD^\rmc}$ and $B_{\calD^\rmc}$ are one-third of the number of affine Lagrangian subspaces that contain any given $\bfv\in\calD^\rmc$ and the number~\eqref{eq.pf of binders of sub-ETFs 2} of affine Lagrangian subspaces of $\calV$ that do not lie in $\calD$, respectively.)

We now use a similar approach on the set~\eqref{eq.D part} of $2^{J-1}$-element subsets of $\calV$ that arise by taking intersections of the members of~\eqref{eq.pf of binders of sub-ETFs 1} with $\calD$.
Here, we claim the natural surjection from $\bin(\bfPsi)\backslash\bin(\bfPsi_\calD)$ onto this set is one-to-one.
To see this, let $\bfv_1+\calL_1$ and $\bfv_2+\calL_2$ be affine Lagrangian subspaces of $\calV$ that do not lie in $\calD$ and assume
$(\bfv_1+\calL_1)\cap\calD=(\bfv_2+\calL_2)\cap\calD$.
Let $\bfv_\calD$ be any one of the $2^{J-1}$ members of $(\bfv_1+\calL_1)\cap\calD$.
Adding $\bfv_\calD$ to $\bfv_1+\calL_1$, $\bfv_2+\calL_2$ and $\calD$ yields $\calL_1$, $\calL_2$ and $\calQ_+:=\bfv_\calD+\calD$.
Under this shift, $(\bfv_1+\calL_1)\cap\calD=(\bfv_2+\calL_2)\cap\calD$ becomes $\calL_1\cap\calQ_+=\calL_2\cap\calQ_+$,
while the fact that $\bfv_1+\calL_1,\bfv_2+\calL_2\nsubseteq\calD$ implies that $\calL_1$ and $\calL_2$ are not totally singular with respect to $\calQ_+$.
Here, Lemma~\ref{lemma.totally singular}(a) gives that $\calS:=\calL_1\cap\calQ_+=\calL_2\cap\calQ_+$ is a totally singular subspace of dimension $J-1$, at which point Lemma~\ref{lemma.totally singular}(c) gives that $\calS$ is contained in exactly one Lagrangian subspace of $\calV$ that is not totally singular with respect to $\calQ_+$.
Thus, $\calL_1=\calL_2$.
Adding $\bfv_\calD$ to this equation gives $\bfv_1+\calL_1=\bfv_2+\calL_2$ as claimed.
As such, for any distinct $\bfv_1,\bfv_2\in\calD$,
the number of members of~\eqref{eq.D part} that contain both $\bfv_1$ and $\bfv_2$ is independent of one's choice of such $\bfv_1,\bfv_2$, being the number $\Lambda-\Lambda_\calD$ of members of $\bin(\bfPsi)$ that contain $\bfv_1$ and $\bfv_2$
less the number of members of $\bin(\bfPsi_\calD)$ that do so.
The set~\eqref{eq.D part} of $K_0:=2^{J-1}$-element subsets of $\calD$ thus forms a BIBD over the $\calV_0:=2^{J-1}(2^J+1)$-element vertex set $\calD$ whose ``$\Lambda$'', ``$R$'' and ``$B$'' parameters are obtained by subtracting those of~\eqref{eq.BIBD parameters of binder of sub-ETF of dual symplectic} from those of~\eqref{eq.BIBD parameters of binder of dual symplectic}, as summarized in~\eqref{eq.BIBD parameters of last BIBD}.

The relationship~\eqref{eq.incidence relation} between the incidence matrices of these BIBDs is merely a way to visualize many of these facts.
As such, all that remains to be shown is that the binder of \smash{$\widehat{\bfPhi}_{\calD^\rmc}$} is indeed the set
$(\bin(\bfPsi)\backslash\bin(\bfPsi_\calD))\cap\calD^\rmc$ of~\eqref{eq.D complement part}.
Here since $(-1)^{\rmB(\bfv_1,\bfv_2)+\rmB(\bfv_2,\bfv_3)+\rmB(\bfv_3,\bfv_1)}
=(-1)^{\rmB(\bfv_1+\bfv_3,\bfv_2+\bfv_3)}$ for any $\bfv_1,\bfv_2,\bfv_3\in\calV$,
applying~\eqref{eq.triple product simplex} to the expression for
\smash{$\ip{\widehat{\bfphi}_{\bfv_1}}{\widehat{\bfphi}_{\bfv_2}}$} given in Lemma~\ref{lemma.symplectic sub-ETFs}(b) gives that a subset $\calK$ of $\calD^\rmc$ belongs to $\bin(\widehat{\bfPhi}_{\calD^\rmc})$ if and only if
\begin{equation}
\label{eq.pf of binders of sub-ETFs 4}
\#(\calK)=2^{J-1},
\quad
\rmB(\bfv_1+\bfv_3,\bfv_2+\bfv_3)=0,\
\forall\,\bfv_1,\bfv_2,\bfv_3\in\calK,
\ \bfv_1\neq\bfv_2\neq\bfv_3\neq\bfv_1.
\end{equation}
If $\calK\in(\bin(\bfPsi)\backslash\bin(\bfPsi_\calD))\cap\calD^\rmc$ then
$\calK=(\bfv+\calL)\cap\calD^\rmc$ for some affine Lagrangian subspace $\bfv+\calL$ of $\calV$ that does not lie in $\calD$,
and so $\#(\calK)=2^{J-1}$.
Here, arguing similarly to before, for any $\bfv_{\calD^\rmc}\in\calK$,
we have that
$\bfv_{\calD^\rmc}+\calK=\calL\cap\calQ_-$
where $\calQ_-:=\bfv_{\calD^\rmc}+\calD^\rmc$ is the quadric of some quadratic form $\rmQ_-$ that yields $\rmB$ with $\sign(\rmQ_-)<0$.
For any $\bfv_1,\bfv_2,\bfv_3\in\calK$, we thus have $\bfv_{\calD^\rmc}+\bfv_1,\bfv_{\calD^\rmc}+\bfv_2,\bfv_{\calD^\rmc}+\bfv_3\in\calL$,
implying $\bfv_1+\bfv_3,\bfv_2+\bfv_3\in\calL$ and so
$\rmB(\bfv_1+\bfv_3,\bfv_2+\bfv_3)=0$.
Thus, if $\calK\in(\bin(\bfPsi)\backslash\bin(\bfPsi_\calD))\cap\calD^\rmc$ then $\calK$ satisfies~\eqref{eq.pf of binders of sub-ETFs 4} and so $\calK\in\bin(\widehat{\bfPhi}_{\calD^\rmc})$.
Conversely, if $\calK\in\bin(\widehat{\bfPhi}_{\calD^\rmc})$ then take any $\bfv_{\calD^\rmc}\in\calK$ and let $\calS:=\bfv_{\calD^\rmc}+\calK$.
By~\eqref{eq.pf of binders of sub-ETFs 4}, $\calS$ has cardinality $2^{J-1}$.
For any distinct $\bfv_1,\bfv_2\in\calS\backslash\set{\bfzero}$,
$\bfv_1+\bfv_{\calD^\rmc}$ and $\bfv_2+\bfv_{\calD^\rmc}$ are distinct members of $\calK\backslash\set{\bfv_{\calD^\rmc}}$,
and so~\eqref{eq.pf of binders of sub-ETFs 4} gives
\begin{equation*}
\rmB(\bfv_1,\bfv_2)
=\rmB((\bfv_1+\bfv_{\calD^\rmc})+\bfv_{\calD^\rmc},(\bfv_2+\bfv_{\calD^\rmc})+\bfv_{\calD^\rmc})
=0.
\end{equation*}
Since this clearly also holds when $\bfv_1=\bfv_2$ or when either $\bfv_1$ or $\bfv_2$ is $\bfzero$,
$\calS$ is totally orthogonal.
Here, since $\bin(\widehat{\bfPhi}_{\calD^\rmc})$ is by definition a set of subsets of $\calD^\rmc$, we also have
$\calS=\bfv_{\calD^\rmc}+\calK\subseteq\calQ_-:=\bfv_{\calD^\rmc}+\calD^\rmc$,
that is, that $\calS$ is totally singular with respect to the quadratic form that has $\rmQ_-$ as its quadric.
Let $\calL$ be any Lagrangian subspace of $\calV$ that contains $\Span(\calS)$.
Since $\sign(\rmQ_-)<0$,
Lemma~\ref{lemma.totally singular}(a) and (b) give that $\calL\cap\calQ_-$ is a subspace of $\calV$ of dimension $J-1$ that is totally singular with respect to $\rmQ_-$.
Since $\calS\subseteq\calL\cap\calQ_-$ where both $\calS$ and $\calL\cap\calQ_-$ have cardinality $2^{J-1}$, this implies $\bfv_{\calD^\rmc}+\calK=\calS=\calL\cap\calQ_-$,
and so $\calK
=\bfv_{\calD^\rmc}+(\calL\cap\calQ_-)
=(\bfv_{\calD^\rmc}+\calL)\cap\calD^\rmc
\in(\bin(\bfPsi)\backslash\bin(\bfPsi_\calD))\cap\calD^\rmc$.
Thus,
$\bin(\widehat{\bfPhi}_{\calD^\rmc})
=(\bin(\bfPsi)\backslash\bin(\bfPsi_\calD))\cap\calD^\rmc$, as claimed.
\end{proof}

\begin{example}
We now consider instances of Lemma~\ref{lemma.symplectic sub-ETFs} and Theorems~\ref{theorem.binder of sub-ETF of symplectic} and~\ref{theorem.binders of sub-ETFs} where $J=2$.
Up to isomorphism, $\bbF_2^4$ is the only vector space over $\bbF_2$ of dimension $2J=4$,
\begin{equation*}
\rmB((x_1,x_2,x_3,x_4),(y_1,y_2,y_3,y_4))
=(x_1y_2+x_2y_1)+(x_3y_4+x_4y_3)
\end{equation*}
is the symplectic form on it,
and
\begin{equation}
\label{eq.quadratic example 1}
\rmQ_+(x_1,x_2,x_3,x_4):=x_1x_2+x_3x_4,
\quad
\rmQ_-(x_1,x_2,x_3,x_4):=x_1^{}x_2^{}+x_3^{}x_4^{}+x_3^2+x_4^2,
\end{equation}
are the quadratic forms that yield this symplectic form via~\eqref{eq.quadratic symplectic relation} of positive and negative sign~\eqref{eq.sign of quadratic}, respectively.
The quadrics~\eqref{eq.quadric} of these quadratic forms are
\begin{align}
\nonumber
\calQ_+&:=\set{0000,0001,0010,0100,0101,0110,1000,1001,1010,1111},\\
\label{eq.quadratic example 2}
\calQ_-&:=\set{0000,0100,1000,1101,1110,1111},
\end{align}
respectively.
By Lemma~\ref{lemma.affine quadric},
the $16$ shifts of $\calQ_+$ and $\calQ_-$ are the $32$ affine quadrics of $\calV$.
Each is either the quadric of some quadratic form that yields $\rmB$ (that equates to one of the two forms above via a change of basis) or the complement of such a quadric.
In particular, $\calQ_-^\rmc$ is a shift of $\calQ_+$,
and so is one of the $16$ choices of $\calD$ that yield six ETFs via Lemma~\ref{lemma.symplectic sub-ETFs}.
Letting $\calD=\calQ_-^\rmc$ in particular yields the four ETFs of Figure~\ref{figure.example}, and two others arise as the dual of two of these.
\begin{figure}
\begin{equation*}
\setlength{\arraycolsep}{0pt}
\begin{array}{ccccc}
\bfPsi&=&\frac1{\sqrt{2}}
\begin{tiny}\left[\begin{array}{cccccc|cccccccccc}
+&+&+&+&+&+&+&+&+&+&+&+&+&+&+&+\\
+&+&-&-&-&-&+&+&+&+&+&+&-&-&-&-\\
+&-&+&-&-&-&+&+&+&-&-&-&+&+&+&-\\
+&-&-&+&-&-&+&-&-&-&+&+&-&+&+&+\\
+&-&-&-&+&-&-&+&-&+&-&+&+&-&+&+\\
+&-&-&-&-&+&-&-&+&+&+&-&+&+&-&+\\
\end{array}\right]\end{tiny}
&\cong&\frac1{\sqrt{2}}
\begin{tiny}\left[\begin{array}{cccccc|cccccccccc}
+&-&-&-&-&-& & & & & & & & & & \\
+&-& & & & & & & & & & &+&+&+&+\\
+& &-& & & & & & &+&+&+& & & &+\\
+& & &-& & & &+&+&+& & &+& & & \\
+& & & &-& &+& &+& &+& & &+& & \\
+& & & & &-&+&+& & & &+& & &+& \\
 &+&+& & & &+&+&+& & & & & & &-\\
 &+& &+& & &+& & & &+&+&-& & & \\
 &+& & &+& & &+& &+& &+& &-& & \\
 &+& & & &+& & &+&+&+& & & &-& \\
 & &+&+& & &+& & &-& & & &+&+& \\
 & &+& &+& & &+& & &-& &+& &+& \\
 & &+& & &+& & &+& & &-&+&+& & \\
 & & &+&+& & & &-& & &+& & &+&+\\
 & & &+& &+& &-& & &+& & &+& &+\\
 & & & &+&+&-& & &+& & &+& & &+
\end{array}\right]\end{tiny}\medskip
\\
\bfPhi&=&\frac1{\sqrt{2}}
\begin{tiny}\left[\begin{array}{cccccc|cccccccccc}
+&+&+&+&-&-&+&-&-&+&-&-&+&-&-&+\\
+&+&+&-&+&-&-&+&-&-&+&-&-&+&-&+\\
+&+&+&-&-&+&-&-&+&-&-&+&-&-&+&+\\
+&+&-&-&+&+&+&-&-&+&-&-&-&+&+&-\\
+&+&-&+&-&+&-&+&-&-&+&-&+&-&+&-\\
+&+&-&+&+&-&-&-&+&-&-&+&+&+&-&-\\
+&-&+&-&+&+&+&-&-&-&+&+&+&-&-&-\\
+&-&+&+&-&+&-&+&-&+&-&+&-&+&-&-\\
+&-&+&+&+&-&-&-&+&+&+&-&-&-&+&-\\
+&-&-&+&+&+&+&+&+&-&-&-&-&-&-&+
\end{array}\right]\end{tiny}
&\cong&\frac1{\sqrt{3}}
\begin{tiny}\left[\begin{array}{cccccc|cccccccccc}
+&+& & & & &-& & &-& & & & & & \\
+&+& & & & & &-& & &-& & & & & \\
+&+& & & & & & &-& & &-& & & & \\
+& &+& & & &-& & & & & &-& & & \\
+& &+& & & & &-& & & & & &-& & \\
+& &+& & & & & &-& & & & & &-& \\
+& & &+& & &-& & & & & & & & &-\\
+& & &+& & & & & & &-& & & &-& \\
+& & &+& & & & & & & &-& &-& & \\
+& & & &+& & &-& & & & & & & &-\\
+& & & &+& & & & &-& & & & &-& \\
+& & & &+& & & & & & &-&-& & & \\
+& & & & &+& & &-& & & & & & &-\\
+& & & & &+& & & &-& & & &-& & \\
+& & & & &+& & & & &-& &-& & & \\
 &+&-& & & & & & &-& & &+& & & \\
 &+&-& & & & & & & &-& & &+& & \\
 &+&-& & & & & & & & &-& & &+& \\
 &+& &-& & & &-& & & & & & &+& \\
 &+& &-& & & & &-& & & & &+& & \\
 &+& &-& & & & & &-& & & & & &+\\
 &+& & &-& &-& & & & & & & &+& \\
 &+& & &-& & & &-& & & &+& & & \\
 &+& & &-& & & & & &-& & & & &+\\
 &+& & & &-&-& & & & & & &+& & \\
 &+& & & &-& &-& & & & &+& & & \\
 &+& & & &-& & & & & &-& & & &+\\
 & &+&-& & & &-& & & &+& & & & \\
 & &+&-& & & & &-& &+& & & & & \\
 & &+&-& & & & & & & & &-& & &+\\
 & &+& &-& &-& & & & &+& & & & \\
 & &+& &-& & & &-&+& & & & & & \\
 & &+& &-& & & & & & & & &-& &+\\
 & &+& & &-&-& & & &+& & & & & \\
 & &+& & &-& &-& &+& & & & & & \\
 & &+& & &-& & & & & & & & &-&+\\
 & & &+&-& &-&+& & & & & & & & \\
 & & &+&-& & & & &+&-& & & & & \\
 & & &+&-& & & & & & & &+&-& & \\
 & & &+& &-&-& &+& & & & & & & \\
 & & &+& &-& & & &+& &-& & & & \\
 & & &+& &-& & & & & & &+& &-& \\
 & & & &+&-& &-&+& & & & & & & \\
 & & & &+&-& & & & &+&-& & & & \\
 & & & &+&-& & & & & & & &+&-& \\
\hline
 & & & & & &+&-& &+&-& & & & & \\
 & & & & & &+&-& & & & &+&-& & \\
 & & & & & &+& &-&+& &-& & & & \\
 & & & & & &+& &-& & & &+& &-& \\
 & & & & & &+& & & &-& & & &-&+\\
 & & & & & &+& & & & &-& &-& &+\\
 & & & & & & &+&-& &+&-& & & & \\
 & & & & & & &+&-& & & & &+&-& \\
 & & & & & & &+& &-& & & & &-&+\\
 & & & & & & &+& & & &-&-& & &+\\
 & & & & & & & &+&-& & & &-& &+\\
 & & & & & & & &+& &-& &-& & &+\\
 & & & & & & & & &+&-& &-&+& & \\
 & & & & & & & & &+& &-&-& &+& \\
 & & & & & & & & & &+&-& &-&+&
\end{array}\right]\end{tiny}
\end{array}
\end{equation*}
\caption{
\label{figure.example}
Four of the six ETFs that arise via Lemma~\ref{lemma.symplectic sub-ETFs} from the $10$-element affine quadric $\calD$ that is the complement of the $6$-element quadric~\eqref{eq.quadratic example 2} of the canonical quadratic form~\eqref{eq.quadratic example 1} on $\bbF_2^4$ of negative sign.
The $10$ rightmost columns of $\bfPsi$ and $6$ leftmost columns of $\bfPhi$ form the ETFs $\bfPsi_\calD$ and $\bfPhi_{\calD^\rmc}$, respectively.
These ETFs are DTETFs (Theorem~\ref{theorem.sub-ETFs are DT}) and have nontrivial binders (Theorems~\ref{theorem.binder of sub-ETF of symplectic} and~\ref{theorem.binders of sub-ETFs}) which necessarily form BIBDs (Theorem~\ref{theorem.binders of DTETFs}).
Phasing these BIBDs \`{a} la~\eqref{eq.phased incidence} yields ETFs equivalent to their duals.
By Theorem~\ref{theorem.binders of sub-ETFs}, $\bin(\bfPsi)$ (bottom right matrix) decomposes into three copies of $\bin(\widehat{\bfPhi}_{\calD^\rmc})$ (upper left block), $\bin(\bfPsi_\calD)$ (lower right block) and a third BIBD (upper right block).}
\end{figure}
Specifically, the dual ETFs $\bfPhi$ and $\bfPsi$ of Lemma~\ref{lemma.symplectic sub-ETFs}(a) arise by extracting the ten $\calD$-indexed and six $\calD^\rmc$-indexed rows, respectively, of the character table $\bfGamma$ of $\bbF_2^4$ with respect to $\rmB$, namely the $16\times 16$ matrix~\eqref{eq.Gamma 16x16} with its rows and columns depicted in Figure~\ref{figure.example} in ``$\calD^\rmc$ then $\calD$'' order.
These ETFs are equivalent to---and so have the same binders as---the symplectic ETF and its dual, respectively, (Definition~\ref{def.symplectic ETF}) that arise from this $\rmB$.
In particular, by Theorems~\ref{theorem.binder of dual symplectic} and~\ref{theorem.binders of sub-ETFs},
$\bfPsi$ (upper left matrix of Figure~\ref{figure.example}) is a $16$-vector ETF for $\bbR^{\calD^\rmc}\cong\bbR^6$ whose binder consists of all affine Lagrangian subspaces of $\bbF_2^4$, as detailed in Example~\ref{example.binder of 6x16}, which forms a BIBD with $(V,K,\Lambda,R,B)$ parameters $(16,4,3,15,60)$.
Phasing the incidence matrix of this BIBD \`{a} la~\eqref{eq.phased incidence} (bottom right matrix of Figure~\ref{figure.example}) yields an ETF that is equivalent to its dual $\bfPhi$.
Meanwhile $\bfPhi$ (lower left matrix of Figure~\ref{figure.example}) is a $16$-vector ETF for $\bbR^\calD\cong\bbR^{10}$.
Though a symplectic-equivalent ETF such as this usually has an empty binder (Theorem~\ref{theorem.binder of symplectic}) this $P=2$, $J=2$ case is unique:
by Theorem~\ref{theorem.binder of sub-ETF of symplectic}, its binder consists of all $\bbF_2^4$ shifts of $\calD^\rmc$, and it forms a BIBD with parameters $(16,6,2,6,16)$.
Its phased incidence matrix (upper right matrix of Figure~\ref{figure.example}) is equivalent to $\bfPsi$.
By Theorem~\ref{theorem.binders of DTETFs},
any block of either of these two BIBDs is an oval of the other,
that is, the support of any row of one incidence matrix has either $0$ or $2$ indices in common with the support of any row of the other.

By Lemma~\ref{lemma.symplectic sub-ETFs}(b),
the $\calD^\rmc$-indexed columns of $\bfPhi$ (the six leftmost columns of the lower left matrix of Figure~\ref{figure.example}) form a $6$-vector ETF $\bfPhi_{\calD^\rmc}$ for its $5$-dimensional span.
By Theorem~\ref{theorem.binder of sub-ETF of symplectic},
this is an exceptional case of an ETF that typically has an empty binder:
here since $\bfPhi_{\calD^\rmc}$ is a simplex, its binder consists of the single set $\calD^\rmc$ (the first row of the upper right matrix of Figure~\ref{figure.example}).
By Lemma~\ref{lemma.symplectic sub-ETFs}(c),
the $\calD$-index columns of $\bfPsi$ (the ten rightmost columns of the upper left matrix of Figure~\ref{figure.example}) form a $10$-vector ETF $\bfPsi_\calD$ for its $5$-dimensional span.
By Theorem~\ref{theorem.binders of sub-ETFs},
each of the $15$ Lagrangian subspaces of $\bbF_2^4$ has a unique coset that lies in $\calD$,
and are the members of the binder of $\bfPsi_\calD$,
which forms a BIBD with parameters $(10,4,2,6,15)$.
Theorem~\ref{theorem.binders of sub-ETFs} moreover gives that each of the remaining $60-15=45$ affine Lagrangian subspaces of $\bbF_2^4$ are split perfectly in half by $\calD$,
and that taking the $\calD^\rmc$-components of these subspaces yields exactly three copies of the binder of the dual \smash{$\widehat{\bfPhi}_{\calD^\rmc}$} of $\bfPhi_{\calD^\rmc}$, which in this case forms a BIBD with parameters $(6,2,1,5,15)$.
In particular, the binders of both $\bfPsi_\calD$ and $\widehat{\bfPhi}_{\calD^\rmc}$ appear in that of $\bfPsi$ (as the upper left block and lower right block, respectively, of the lower right matrix of Figure~\ref{figure.example}).
\end{example}

Having calculated the binder of any ETF that lies in five of the six families of Lemma~\ref{lemma.symplectic sub-ETFs}, we now do the same for the sixth.
Remarkably, these binders are empty except in three cases.
Our argument relies on the following technical lemma concerning quadratic forms over $\bbF_2$.
Since such forms are classical, this lemma is unlikely to be original,
and so we relegate its proof to the appendix%
%\ of~\cite{FickusL23}% COMMENT
.

\begin{lemma}
\label{lemma.counting pairs}
Let $\calQ$ be the quadric of a quadratic form $\rmQ$ that yields a symplectic form $\rmB$ on a vector space $\calV$ over $\bbF_2$ of dimension $2J$.
If $J\geq1$, the number of hyperbolic pairs $(\bfv_1,\bfv_2)$ in $\calV$ is
\begin{equation}
\label{eq.lemma.counting pairs.1}
2^{2J-2}[2^{J-1}+\sign(\rmQ)][2^{J}-\sign(\rmQ)].
\end{equation}
Moreover, for any such $(\bfv_1,\bfv_2)$,
a vector in $\calQ$ is nonorthogonal to both $\bfv_1$ and $\bfv_2$ if and only if it is of the form $\bfv_1+\bfv_2+\bfv_0$ where $\bfv_0\in\calV_0\backslash\calQ_0$,
where $\calQ_0$ is the quadric of the restriction of $\rmQ$ to $\calV_0:=\set{\bfv_1,\bfv_2}^\perp$.

If $J\geq2$, the number of ordered pairwise nonorthogonal $4$-tuples $(\bfv_1,\bfv_2,\bfv_3,\bfv_4)$ of members of $\calQ$ is
\begin{equation}
\label{eq.lemma.counting pairs.2}
2^{4J-6}[2^{2J-2}-1][2^{J-2}-\sign(\rmQ)][2^{J}-\sign(\rmQ)].
\end{equation}
Moreover, for any such $(\bfv_1,\bfv_2,\bfv_3,\bfv_4)$,
a vector in $\calQ$ is nonorthogonal to $\bfv_1$, $\bfv_2$, $\bfv_3$ and $\bfv_4$ if and only if it is of the form $\bfv_1+\bfv_2+\bfv_3+\bfv_4+\bfv_{00}$ where $\bfv_{00}$ lies in the quadric $\calQ_{00}$ of the restriction $\rmQ_{00}$ of $\rmQ$ to the space $\calV_{00}:=\set{\bfv_1,\bfv_2,\bfv_3,\bfv_4}^\perp$ of dimension $2J-4$,
which has $\sign(\rmQ_{00})=-\sign(\rmQ)$.
\end{lemma}

\begin{theorem}
\label{theorem.binder Tremain}
Under the assumptions of Lemma~\ref{lemma.symplectic sub-ETFs},
\smash{$\bin(\widehat{\bfPsi}_\calD)$} is empty when $J>4$.
When instead $J\in\set{2,3,4}$,
it forms a BIBD with the following $(V,K,\Lambda,R,B)$ parameters:
\begin{equation*}
J=2:\ (10,4,2,6,15),\qquad
J=3:\ (36,6,8,56,336),\qquad
J=4:\ (136,10,64,960,13056).
\end{equation*}
In these three cases, every member of either $\bin(\bfPsi_\calD)$ or \smash{$\bin(\widehat{\bfPsi}_\calD)$} is an oval of the other.
\end{theorem}

\begin{proof}
By Theorem~\ref{theorem.sub-ETFs are DT},
the dual ETFs $\bfPsi_\calD$ and $\widehat{\bfPsi}_\calD$ of Lemma~\ref{lemma.symplectic sub-ETFs}(c) are DTETFs.
By Theorem~\ref{theorem.binders of sub-ETFs},
$\bin(\bfPsi_\calD)$ forms a BIBD.
By Theorem~\ref{theorem.binders of DTETFs},
$\bin(\widehat{\bfPsi}_\calD)$ is either empty or forms a BIBD and, in the latter case, every member of either $\bin(\bfPsi_\calD)$ or \smash{$\bin(\widehat{\bfPsi}_\calD)$} is an oval of the other.
Here, Lemma~\ref{lemma.symplectic sub-ETFs}(c) gives
$S=\norm{\widehat{\bfpsi}_{\bfv}}^2=2^{J-1}+1$ for all $\bfv\in\calD$ and
\smash{$\ip{\widehat{\bfpsi}_{\bfv_1}}{\widehat{\bfpsi}_{\bfv_2}}=(-1)^{\rmB(\bfv_1,\bfv_2)}$}
for all distinct $\bfv_1,\bfv_2\in\calD$.
By~\eqref{eq.triple product simplex},
a subset $\calK$ of $\calD$ thus belongs to \smash{$\bin(\widehat{\bfPsi}_\calD)$} if and only if
\begin{equation}
\label{eq.pf of binder Tremain 1}
\#(\calK)=2^{J-1}+2,
\quad
\rmB(\bfv_1,\bfv_2)+\rmB(\bfv_2,\bfv_3)+\rmB(\bfv_3,\bfv_1)=1,
\ \forall\,\bfv_1,\bfv_2,\bfv_3\in\calK,
\ \bfv_1\neq\bfv_2\neq\bfv_3\neq\bfv_1.
\end{equation}
Moreover, for any $\bfv_0\in\calV$,
Lemma~\ref{lemma.affine quadric} gives that $\bfv_0+\calD$ is an affine quadric of size $2^{J-1}(2^J+1)$,
and
\begin{equation*}
\rmB(\bfv_1+\bfv_0,\bfv_2+\bfv_0)+\rmB(\bfv_2+\bfv_0,\bfv_3+\bfv_0)+\rmB(\bfv_3+\bfv_0,\bfv_1+\bfv_0)
=\rmB(\bfv_1,\bfv_2)+\rmB(\bfv_2,\bfv_3)+\rmB(\bfv_3,\bfv_1)
\end{equation*}
for any $\bfv_1,\bfv_2,\bfv_3\in\calV$.
It follows that for any $\bfv_0\in\calV$,
the mapping $\calK\mapsto\bfv_0+\calK$ is a bijection from the binder of $\widehat{\bfPsi}_\calD$ to that of the analogous ETF $\widehat{\bfPsi}_{\bfv_0+\calD}$ that arises from applying Lemma~\ref{lemma.symplectic sub-ETFs}(c) with ``$\calD$'' being $\bfv_0+\calD$.
As such, to determine when \smash{$\bin(\widehat{\bfPsi}_\calD)$} is nonempty---and if so, the parameters of the resulting BIBD---we can take $\bfv_0\in\calD$ to assume without loss of generality that $\bfzero\in\calD$.
Since $\calD$ is an affine quadric of size $2^{J-1}(2^J+1)$ that contains $\bfzero$,
it is the quadric $\calQ$ of some positive quadratic form that yields $\rmB$.
Moreover, since \smash{$\bin(\widehat{\bfPsi}_\calD)$} is a BIBD if it is nonempty, it suffices to determine the number $R$ of subsets $\calK$ of $\calD=\calQ$ that both satisfy~\eqref{eq.pf of binder Tremain 1} and contain $\bfzero$: if $R=0$, then \smash{$\bin(\widehat{\bfPsi}_\calD)$} is empty,
while if $R>0$,
then~\eqref{eq.BIBD parameter relations} gives \smash{$\Lambda=\tfrac{R(K-1)}{V-1}$} and \smash{$B=\tfrac{VR}{K}$} where $V=\#(\calD)=2^{J-1}(2^J+1)$ and $K=\#(\calK)=2^{J-1}+2$.
Note that if $\calK\subseteq\calD$ satisfies \eqref{eq.pf of binder Tremain 1},
then taking $\bfv_3=\bfzero$ in~\eqref{eq.pf of binder Tremain 1} implies that $\calS=\calK\backslash\set{\bfzero}$ satisfies
\begin{equation}
\label{eq.pf of binder Tremain 2}
\calS\subseteq\calQ,
\quad
\#(\calS)=2^{J-1}+1,
\quad
\rmB(\bfv_1,\bfv_2)=1,\ \forall\,\bfv_1,\bfv_2\in\calS,\ \bfv_1\neq\bfv_2.
\end{equation}
Conversely, if $\calS$ is any set that satisfies \eqref{eq.pf of binder Tremain 2},
then $\calK=\calS\sqcup\set{\bfzero}$ satisfies~\eqref{eq.pf of binder Tremain 1}:
clearly $\bfzero\notin\calS$,
and moreover $\rmB(\bfv_1,\bfv_2)+\rmB(\bfv_2,\bfv_3)+\rmB(\bfv_3,\bfv_1)=1$ for any pairwise distinct $\bfv_1,\bfv_2,\bfv_3\in\calK$,
being either $1+1+1$ when all three of these vectors are nonzero,
and otherwise being a sum of one $1$ and two values of $0$.
As such, $R$ is moreover the number of sets $\calS$ that satisfy~\eqref{eq.pf of binder Tremain 2}.
If $R>0$,
then such a set $\calS$ exists,
and Theorem~\ref{theorem.binder of symplectic} implies that $2^{J-1}+1=\#(\calS)\leq 2J+1$, that is, $2^{J-2}\leq J$,
namely $J\leq 4$.
By the contrapositive statement, \smash{$\bin(\widehat{\bfPsi}_\calD)$} is thus empty when $J>4$.

In summary, it thus now suffices to assume that $J$ is $2$, $3$ or $4$,
and show that the number $R$ of sets $\calS$ that satisfy \eqref{eq.pf of binder Tremain 2} is $6$, $56$ or $960$, respectively.
Theorem~4.16 of~\cite{Wertheimer86} gives the parameters of these \textit{ovoid} designs without proof.
For the sake of completeness, we count these $\calS$ explicitly.
To do so,
note that there are exactly $(2^{J-1}+1)!$ ways to enumerate the members of any such set $\calS$ to obtain an ordered finite sequence \smash{$(\bfv_i)_{i=1}^{2^{J-1}+1}$} in $\calQ$ that satisfies $\rmB(\bfv_{i_1},\bfv_{i_2})=1$ for all distinct $i_1,i_2$.
As such, we obtain $R$ by dividing the number of such sequences by $(2^{J-1}+1)!$.

When $J=2$, taking $J=2$ and $\sign(\rmQ)>0$ in~\eqref{eq.lemma.counting pairs.1} of  Lemma~\ref{lemma.counting pairs} gives $36$ hyperbolic pairs $(\bfv_1,\bfv_2)$ in $\calV$, and moreover that for any such $(\bfv_1,\bfv_2)$,
a vector $\bfv_3\in\calQ$ is nonorthogonal to both $\bfv_1$ and $\bfv_2$ if and only if $\bfv_3=\bfv_1+\bfv_2+\bfv_0$ where $\bfv_0$ belongs to the singleton set $\calV_0\backslash\calQ_0$ of vectors that are nonsingular with respect to a positive quadratic form on a symplectic space over $\bbF_2$ of dimension $2$.
There are thus $36$ ways to choose three ordered pairwise nonorthogonal members $(\bfv_1,\bfv_2,\bfv_3)$ of $\calQ$.
Dividing this number by $3!=6$ gives that there are exactly $R=6$ sets $\calS$ that satisfy~\eqref{eq.pf of binder Tremain 2} when $J=2$, as claimed.

Similarly, when $J=3$, taking $J=3$ and $\sign(\rmQ)>0$ in~\eqref{eq.lemma.counting pairs.2} of Lemma~\ref{lemma.counting pairs} gives $6720$ ways to choose four ordered pairwise nonorthogonal members $(\bfv_1,\bfv_2,\bfv_3,\bfv_4)$ of $\calQ$,
and moreover that the number of ways to append a fifth such vector $\bfv_5$ to this tuple is the number of vectors that are singular with respect to a negative quadratic form over a symplectic space over $\bbF_2$ of dimension $2$, namely $1$.
As such, dividing $(6720)(1)=6720$ by $5!=120$ gives that there are exactly $R=56$ sets $\calS$ that satisfy~\eqref{eq.pf of binder Tremain 2} when $J=3$, as claimed.

When $J=4$, taking $J=4$ and $\sign(\rmQ)>0$ in~\eqref{eq.lemma.counting pairs.2} of Lemma~\ref{lemma.counting pairs} gives $2903040$ ways to choose four ordered pairwise nonorthogonal members $(\bfv_1,\bfv_2,\bfv_3,\bfv_4)$ of $\calQ$.
Moreover, for any $(\bfv_1,\bfv_2,\bfv_3,\bfv_4)$,
taking $J=2$ and $\sign(\rmQ)<0$ in~\eqref{eq.lemma.counting pairs.2} of Lemma~\ref{lemma.counting pairs} gives $120$ ways to choose $(\bfv_5,\bfv_6,\bfv_7,\bfv_8)$ so that the members of $(\bfv_1,\bfv_2,\bfv_3,\bfv_4,\bfv_5,\bfv_6,\bfv_7,\bfv_8)$ lie in $\calQ$ and are pairwise nonorthogonal.
For any such $(\bfv_1,\bfv_2,\bfv_3,\bfv_4,\bfv_5,\bfv_6,\bfv_7,\bfv_8)$,
Lemma~\ref{lemma.counting pairs} further gives that the number of ways to append a ninth such vector $\bfv_9$ to this tuple is the number of vectors that are singular with respect to a positive quadratic form over a trivial space, namely $1$.
As such, dividing $(2903040)(120)(1)=348364800$ by $9!=362880$ gives that there are exactly $R=960$ sets $\calS$ that satisfy~\eqref{eq.pf of binder Tremain 2} when $J=4$, as claimed.
\end{proof}

We now summarize the parameters of the BIBDs formed by the binders of the six DTETFs of Lemma~\ref{lemma.symplectic sub-ETFs} when $J\in\set{2,3,4}$.
Here ``$D$'' is the dimension of the ETF's span,
``$V$'' is the number of vectors in the ETF,
``$K$'' is the necessary number of vectors in any regular simplex that it contains,
``$\Lambda$'' is the number of such regular simplices that contain any given pair of distinct ETF vectors,
``$R$'' is the number of such regular simplices that contain any given ETF vector,
and ``$B$'' is the total number of such regular simplices.
In particular, we denote that a binder is empty here by giving $\Lambda$, $R$ and $B$ to be $0$.
Letting $J$ be $2$, $3$ and $4$ in
Lemma~\ref{lemma.symplectic sub-ETFs} and
Theorems~\ref{theorem.binder of sub-ETF of symplectic}, \ref{theorem.binders of sub-ETFs} and~\ref{theorem.binder Tremain} yields the following three tables, respectively:
\begin{equation*}
\begin{array}{lcccccc}
\text{ETF}&
D&V&K&\Lambda&R&B\\
\bfPhi&
10&16&6&2&6&16\\
\bfPsi&
6&16&4&3&15&60\\
\bfPhi_{\calD^\rmc}&
5&6&6&1&1&1\\
\widehat{\bfPhi}_{\calD^\rmc}&
1&6&2&1&5&15\\
\bfPsi_\calD&
5&10&4&2&6&15\\
\widehat{\bfPsi}_\calD&
5&10&4&2&6&15
\end{array}\qquad
\begin{array}{lcccccc}
\text{ETF}&
D&V&K&\Lambda&R&B\\
\bfPhi&
36&64&10&0&0&0\\
\bfPsi&
28&64&8&15&135&1080\\
\bfPhi_{\calD^\rmc}&
21&28&10&0&0&0\\
\widehat{\bfPhi}_{\calD^\rmc}&
7&28&4&5&45&315\\
\bfPsi_\calD&
21&36&8&6&30&135\\
\widehat{\bfPsi}_\calD&
15&36&6&8&56&336
\end{array}\qquad
\begin{array}{lcccccc}
\text{ETF}&
D&V&K&\Lambda&R&B\\
\bfPhi&
136&256&18&0&0&0\\
\bfPsi&
120&256&16&135&2295&36720\\
\bfPhi_{\calD^\rmc}&
85&120&18&0&0&0\\
\widehat{\bfPhi}_{\calD^\rmc}&
35&120&8&45&765&11475\\
\bfPsi_\calD&
85&136&16&30&270&2295\\
\widehat{\bfPsi}_\calD&
51&136&10&64&960&13056
\end{array}
\end{equation*}
For all larger values of $J$,
the binders of $\bfPhi$, $\bfPhi_{\calD^\rmc}$ and $\widehat{\bfPsi}_\calD$ are all empty,
while the binders of $\bfPsi$, $\widehat{\bfPhi}_{\calD^\rmc}$ and $\bfPsi_\calD$ form BIBDs whose parameters are given in Theorem~\ref{theorem.binders of sub-ETFs}.
A number of these entries, such as those for $\bfPsi$ and $\widehat{\bfPsi}_\calD$ when $J=2$, and those for $\widehat{\bfPhi}_{\calD^\rmc}$ and $\widehat{\bfPsi}_\calD$ when $J=3$,
are consistent with those previously published in the literature~\cite{FickusJKM18} based upon numerical evidence.
That said, the BinderFinder algorithm of~\cite{FickusJKM18} is prohibitively expensive for the instances of these ETFs where $J\geq 4$,
and so for example, it is a novel observation that this $120$-vector DTETF for $\bbR^{35}$ has the property that any two of its vectors are contained in exactly $45$ regular simplices.

For $J\in\set{2,3,4}$, any member of the binder of either $\bfPsi_\calD$ or $\widehat{\bfPsi}_\calD$ is an oval of the BIBD formed by the other.
This is true even when $J=2$ and the two resulting BIBDs have the same parameters $(10,4,2,6,15)$.
This can be explicitly verified, for example, by noting that the support of any row of the right block of the upper right matrix of Figure~\ref{figure.example} has either exactly two or zero vertices in common with those of the lower right block of the lower right matrix there.
(We caution here that $\bfPhi_\calD:=(\bfphi_\bfv)_{\bfv\in\calD}$,
the $J=2$ instance of which appears in Figure~\ref{figure.example} as the $10$ rightmost columns of $\bfPhi$, is not an ETF.
In general, the dual of a sub-ETF does not arise as a submatrix of a dual ETF.
That said, the Gram matrices of two such finite sequences of vectors agree in their off-diagonal entries, and so every member of $\bin(\bfPhi)\cap\calD$ of the appropriate cardinality is also a member of $\bin(\widehat{\bfPsi}_\calD)$.)
When $J=2$, $\bin(\bfPhi)$ and $\bin(\bfPsi)$ also consist of ovals of the other,
as do $\bin(\bfPhi_{\calD^\rmc})$ and $\bin(\widehat{\bfPhi}_{\calD^\rmc})$.

To date, we have not discovered an infinite family of DTETFs that have the property that the binders of both them and their duals are nonempty.
By Theorem~\ref{theorem.binders of DTETFs}, any member of either BIBD formed by these binders is necessarily an oval of the other.
In light of the $J=2$ instance of $\bfPsi_\calD$ or $\widehat{\bfPsi}_\calD$,
one possible candidate for such a family consists of $(Q^2+1)$-vector DTETFs for real spaces of dimension $\frac12(Q^2+1)$, where $Q$ is any odd prime power~\cite{IversonM24}.
We leave a deeper investigation of the binders of such DTETFs,
as well as those of others~\cite{IversonM24,DempwolffK23}, for future work.

From the point of view of deterministic compressed sensing,
perhaps the most exciting contribution of this work is the revelation of Theorem~\ref{theorem.binder Tremain} that the binder of the DTETF $\widehat{\Psi}_\calD$ of Lemma~\ref{lemma.symplectic sub-ETFs}(c) is empty whenever $J>4$.
Here recall that the spark of an ETF outperforms Donoho and Elad's spark bound~\eqref{eq.Welch and spark bounds} precisely when its binder is empty.
Though, as noted above, this applies to $\bfPhi$, $\bfPhi_{\calD^\rmc}$ and $\widehat{\bfPsi}_\calD$ whenever $J$ is not small,
the ETFs $\bfPhi$ and $\bfPhi_{\calD^\rmc}$ are of limited applicability to compressed sensing since the dimension of their spans is more than half of the number of their vectors: why bother with the complexities of compressed sensing when one can simply perform traditional sensing at no more than twice the cost?
In contrast, the parameters of $\widehat{\bfPsi}_\calD$ match those of certain Tremain ETFs~\cite{FickusJMP18},
and as $J$ grows large, the dimension of its span $\frac13(2^{J-1}+1)(2^J+1)$ is about one-third of the number $2^{J-1}(2^J+1)$ of its vectors.
This begs the question, What is the spark of \smash{$\widehat{\bfPsi}_\calD$}?
More importantly, to what degree does \smash{$\widehat{\bfPsi}_\calD$} satisfy the restricted isometry property?
Though such problems seem to be computationally intractable in general~\cite{AlexeevCM12,BandeiraDMS13},
we have already seen, in Theorem~\ref{theorem.binders of sub-ETFs} for example,
that it is possible to exactly compute the spark of any ETF that belongs to certain infinite families of them by showing that their binders are nonempty.
Here though
\smash{$\widehat{\bfPsi}_\calD$} has an empty binder, it is highly symmetric---being a DTETF---and its Gram matrix (Lemma~\ref{lemma.symplectic sub-ETFs}(c)) is elegantly expressed in terms of a quadratic form over $\bbF_2$.
By using sufficiently sophisticated algebra and finite geometry, can one for example also exactly compute the spark of \smash{$\widehat{\bfPsi}_\calD$} for all $J>4$?

\section*{Acknowledgments}
We thank Dr.~Joseph~W.~Iverson of the Department of Mathematics of Iowa State University for the helpful comments he provided while visiting us as part of the Air Force Office of Scientific Research's Summer Faculty Fellowship Program.
The views expressed in this article are those of the authors and do not reflect the official policy or position of the United States Air Force, Department of Defense, or the U.S.~Government.

\setlength{\bibsep}{0pt}

%\end{document}

\appendix

\section{Proof of Lemma~\ref{lemma.DT BIBD}}

(a) For any $v_1,v_2,v_3,v_4\in\calV$ such that $v_1\neq v_2$ and $v_3\neq v_4$,
taking $\sigma\in\Aut(\calV,(\calK_b)_{b\in\calB})$ such that $\sigma(v_1)=v_3$ and $\sigma(v_2)=v_4$, the corresponding permutation $\tau$ of $\calB$ maps $\set{b\in\calB: v_1,v_2\in\calK_b}$ onto $\set{b'\in\calB: v_3,v_4\in\calK_{b'}}$,
implying that the number of blocks which contain $v_1$ and $v_2$ is the same as that which contain $v_3$ and $v_4$.

(b) Let $\calL$ be an $L$-element arc of a BIBD $(\calV,(\calK_b)_{b\in\calB})$.
Recall that the incidence matrix $\bfX\in\bbR^{\calB\times\calV}$ of this BIBD satisfies~\eqref{eq.BIBD incidence matrix properties}.
Letting $\bfchi_{\calL}$ be the characteristic function of $\calL$,
we also have $(\bfX\chi_\calL)(b)=\#(\calK_b\cap\calL)\in\set{0,1,2}$ for all $b\in\calB$.
Moreover, letting $N_1$ and $N_2$ be the number of blocks $\calK_b$ that intersect $\calL$ in exactly one or two vertices, respectively,
\eqref{eq.BIBD incidence matrix properties} implies
\begin{align*}
N_1+2N_2
&=\sum_{b\in\calB}\#(\calK_b\cap\calL)
=\ip{\bfone_\calB}{\bfX\bfchi_\calL}
=\ip{\bfX^\rmT\bfone_\calB}{\bfchi_\calL}
=\ip{R\bfone_\calV}{\bfchi_\calL}
=RL,\\
N_1+4N_2
&=\sum_{b\in\calB}[\#(\calK_b\cap\calL)]^2
=\norm{\bfX\bfchi_\calL}^2
=\ip{\bfchi_\calL}{[(R-\Lambda)\brI+\Lambda\bfone_{\calV\times\calV}]\bfchi_\calL}
=(R-\Lambda)L+\Lambda L^2.
\end{align*}
Subtracting the second equation here from twice the first and then recalling~\eqref{eq.BIBD parameter relations} gives
\begin{equation}
\label{eq.pf of DT BIBD 1}
0\leq N_1=2RL-[(R-\Lambda)L+\Lambda L^2]
=\Lambda L[(\tfrac{R}{\Lambda}+1)-L]
=\Lambda L[(\tfrac{V-1}{K-1}+1)-L].
\end{equation}
In particular, $L\leq\frac{V-1}{K-1}+1$.
Moreover, if $L=\frac{V-1}{K-1}+1$ then $L>1$ and~\eqref{eq.pf of DT BIBD 1} gives $N_1=0$, that is, $\#(\calK_b\cap\calL)\in\set{0,2}$ for all $b\in\calB$.
Conversely, if $\calL$ is any nonempty subset of $\calV$ such that $\#(\calK_b\cap\calL)\in\set{0,2}$ for all $b\in\calB$ then $\calL$ is an arc of this BIBD with $L>0$ and $N_1=0$, implying by~\eqref{eq.pf of DT BIBD 1} that $L=\frac{V-1}{K-1}+1$.
When this occurs, $K-1$ necessarily divides $V-1$.
See~\cite{Andriamanalimanana79} for an alternative upper bound on the size of an arc that holds when $K-1\nmid V-1$.

(c) Assume that $(\calV,(\calK_b)_{b\in\calB})$ is doubly transitive---implying by (a) that it is a BIBD---and moreover that this BIBD has an oval,
namely that there exists a nonempty subset $\calL$ of $\calV$ such that $\#(\calK_b\cap\calL)\in\set{0,2}$ for all $b\in\calB$.
Any automorphism $\sigma$ of $(\calV,(\calK_b)_{b\in\calB})$ maps any such oval $\calL$ to another: if $\tau$ is the associated permutation of $\calB$,
\begin{equation*}
\#(\calK_b\cap\sigma(\calL))
=\#(\sigma(\sigma^{-1}(\calK_b)\cap\calL))
=\#(\calK_{\tau^{-1}(b)}\cap\calL)
\in\set{0,2},
\quad\forall\,b\in\calB.
\end{equation*}
Letting $(\calL_a)_{a\in\calA}$ be any enumeration of the set of all ovals of $(\calV,(\calK_b)_{b\in\calB})$,
this implies that $\Aut(\calV,(\calK_b)_{b\in\calB})$ is a subgroup of $\Aut(\calV,(\calL_a)_{a\in\calA})$.
Since $(\calV,(\calK_b)_{b\in\calB})$ is doubly transitive this implies that $(\calV,(\calL_a)_{a\in\calA})$ is as well, and so by (a) is itself a BIBD.
Since $\#(\calK_b\cap\calL_a)\in\set{0,2}$ for all $a\in\calA$ and $b\in\calB$,
any member of either BIBD is an oval of the other.

\section{Proof of Lemma~\ref{lemma.DT implies ETF}}

Consider the matrix $\bfA\in\bbF^{\calN\times\calN}$ defined by
\begin{equation}
\label{eq.pf of DT implies ETF 1}
\bfA(n_1,n_3)
:=\sum_{n_2\in\calN}\ip{\bfphi_{n_1}}{\bfphi_{n_2}}\ip{\bfphi_{n_2}}{\bfphi_{n_3}}\ip{\bfphi_{n_3}}{\bfphi_{n_1}}
=(\bfPhi^*\bfPhi)^2(n_1,n_3)\overline{(\bfPhi^*\bfPhi)(n_1,n_3)}.
\end{equation}
By assumption,
for any $n_1,n_3,n_4,n_6\in\calN$ with $n_1\neq n_3$ and $n_4\neq n_6$, there exists $\sigma\in\Sym(\bfPhi)$ such that $\sigma(n_1)=n_4$ and $\sigma(n_3)=n_6$.
Since $\sigma\in\Sym(\bfPhi)$, there exists a finite sequence of unimodular scalars $(z_n)_{n\in\calN}$ such that
$\ip{\bfphi_{\sigma(n_7)}}{\bfphi_{\sigma(n_8)}}
=\overline{z_{n_7}}z_{n_8}\ip{\bfphi_{n_7}}{\bfphi_{n_8}}$ for all $n_7,n_8\in\calN$.
In particular, $\abs{\ip{\bfphi_{n_4}}{\bfphi_{n_6}}}
=\abs{\ip{\bfphi_{n_1}}{\bfphi_{n_3}}}$
for all such $n_1,n_3,n_4,n_6$,
implying there exists $C\geq0$ such that $\abs{\ip{\bfphi_{n_1}}{\bfphi_{n_3}}}=C$ when $n_1\neq n_3$.
In fact, $C>0$ since $N>D$, and we without loss of generality scale $\bfPhi$ so that $C=1$.
Similarly, \eqref{eq.triple product invariance} implies that $\bfA(n_4,n_6)=\bfA(n_1,n_3)$ for all such $n_1,n_3,n_4,n_6$,
and so there exists $A\in\bbF$ such that $\bfA(n_1,n_3)=A$ when $n_1\neq n_3$.
Multiplying~\eqref{eq.pf of DT implies ETF 1} by $(\bfPhi^*\bfPhi)(n_1,n_3)$ thus gives
$(\bfPhi^*\bfPhi)^2(n_1,n_3)=A(\bfPhi^*\bfPhi)(n_1,n_3)$ when $n_1\neq n_3$.
When instead $n_1=n_3$, $(\bfPhi^*\bfPhi)^2(n_1,n_3)=S^2+(N-1)$ while $(\bfPhi^*\bfPhi)(n_1,n_3)=S$.
Altogether, $(\bfPhi^*\bfPhi)^2=A(\bfPhi^*\bfPhi)+[S(S-A)+(N-1)]\brI$.
That said,
$\ker(\bfPhi^*\bfPhi)$ contains a nonzero vector since $N>D$,
and applying this above equation to it gives $S(S-A)+(N-1)=0$.
Thus, $(\bfPhi^*\bfPhi)^2=A(\bfPhi^*\bfPhi)$, implying $\bfPhi$ is a tight frame and so also an ETF.

\section{Proof of Lemma~\ref{lemma.totally orthogonal}}

(a)
For any finite linear combinations of members of $\calS$,
\begin{equation*}
\rmB\biggparen{\,
\sum_{m_1\in\calM_1}\bfx_1(m_1)\bfv_{m_1},\sum_{m_2\in\calM_2}\bfx_2(m_2)\bfv_{m_2}}
=\sum_{m_1\in\calM_1}\sum_{m_2\in\calM_2}\bfx_1(m_1)\bfx_2(m_2)\,0
=0.
\end{equation*}

(b) If $\calL\subseteq\calL^\perp$ then
$2\dim(\calL)
\leq\dim(\calL)+\dim(\calL^\perp)
=\dim(\calV)
=2J$ and so $\dim(\calL)\leq J$.
Here, $\dim(\calL)=J$ if and only if $\dim(\calL)=\dim(\calL^\perp)$.
Since $\calL\subseteq\calL^\perp$, this equates to having $\calL=\calL^\perp$.

(c) If $(\bfv_m)_{m=1}^M$ is a basis for a totally orthogonal space
then $\rmB(\bfv_{m_1},\bfv_{m_2})=0$ for any $m_1,m_2$,
and so in particular $\bfv_m\in\calV_m^\perp$ for any $m$,
where $\calV_m:=\Span(\bfv_{m'})_{m'=1}^{m-1}$.
Moreover, $\bfv_m\notin\calV_m^{}$ for any $m$ since $(\bfv_m)_{m=1}^M$ is linearly independent.
Conversely, if $(\bfv_m)_{m=1}^M$ is any finite sequence of vectors in $\calV$ with the property that $\bfv_m\in\calV_m^{\perp}\backslash\calV_m^{}$ for all $m$,
then it is linearly independent since $\bfv_m\notin\calV_m$ for all $m$.
Moreover, $\rmB(\bfv_{m_1},\bfv_{m_2})=0$ for any $m_1,m_2$:
this holds when $m_1=m_2$ since $\rmB$ is alternating,
holds when $m_1<m_2$ since $\bfv_{m_2}\in\calV_{m_2-1}^\perp\supseteq\calV_{m_1}^\perp$,
and similarly holds when $m_1>m_2$.
Thus, $(\bfv_m)_{m=1}^M$ is totally orthogonal,
implying $\Span(\bfv_m)_{m=1}^M$ is as well.

(d)
By (b) and (c),
any Lagrangian subspace is the image of the synthesis map $\bfV:\bbF_Q^J\rightarrow\calV$ of a finite sequence $(\bfv_j)_{j=1}^{J}$ of vectors in $\calV$ with $\bfv_j\in\calV_j^{\perp}\backslash\calV_j^{}$ for all $j\in[J]$ where
\smash{$\calV_j:=\Span(\bfv_{j'})_{j'=1}^{j-1}$}.
When chosen in this way, (c) gives that each $\calV_j$ is a totally orthogonal subspace of dimension $j-1$,
and so lies in $\calV_j^{\perp}$, which has dimension $2J-(j-1)$.
As such, for each $j\in[J]$, there are exactly
\smash{$\#(\calV_j^{\perp}\backslash\calV_j^{})
=\#(\calV_j^{\perp})-\#(\calV_j^{})
=Q^{2J-(j-1)}-Q^{j-1}$} choices for $\bfv_j$.
The total number of such $\bfV$ is thus $\prod_{j=1}^{J}(Q^{2J-(j-1)}-Q^{j-1})$.
For any such $\bfV$ and invertible matrix $\bfA\in\bbF_Q^{J\times J}$,
$\bfV\bfA$ is of this same type, having $\im(\bfV\bfA)=\im(\bfV)$.
Conversely, for any \smash{$\bfV_1,\bfV_2:\bbF_Q^J\rightarrow\calV$} of this type with $\im(\bfV_1)=\im(\bfV_2)$, we have that $\bfV_1,\bfV_2:\bbF_Q^J\rightarrow\im(\bfV_1)=\im(\bfV_2)$ are invertible,
and so $\bfV_2=\bfV_1\bfA$ where $\bfA:=\bfV_1^{-1}\bfV_2$ is an invertible matrix.
As such, the total number of Lagrangian subspaces of $\calV$ is obtained by dividing the number $\prod_{j=1}^{J}(Q^{2J-(j-1)}-Q^{j-1})$ of such $\bfV$ by the order of $\GL(\bbF_Q^J)$.
The latter is \smash{$\prod_{j=1}^{J}(Q^J-Q^{j-1})$}:
a matrix \smash{$\bfA\in\bbF_Q^{J\times J}$} is invertible if and only if,
for each $j\in[J]$, its $j$th column does not belong to the $(j-1)$-dimensional span of its prior columns.
Altogether, the total number of Lagrangian subspaces of $\calV$ is thus
\begin{equation*}
\frac{\prod_{j=1}^{J}(Q^{2J-(j-1)}-Q^{j-1})}{\prod_{j=1}^{J}(Q^J-Q^{j-1})}
=\frac{\prod_{j=1}^{J}(Q^{2(J-j+1)}-1)}{\prod_{j=1}^{J}(Q^{J-j+1}-1)}
=\prod_{j=1}^{J}(Q^{J-j+1}+1)
=\prod_{j=1}^{J}(Q^j+1).
\end{equation*}

\section{Proof of Lemma~\ref{lemma.symplectic sub-ETFs}}

For the moment, let $\calD$ be any affine quadric in $\calV$,
and so $\#(\calD)$ may be either $2^{J-1}(2^J+1)$ or $2^{J-1}(2^J-1)$.
Then either $\calD$ or $\calD^\rmc$ is the quadric $\calQ$ of a quadratic form $\rmQ:\calV\rightarrow\bbF_2$ that satisfies~\eqref{eq.quadratic symplectic relation}.
Letting $\bfc\in\bbR^\calV$, $\bfc(\bfv):=(-1)^{\rmQ(\bfv)}$ be the associated \textit{chirp} function,
the characteristic function of $\calD$ is thus
$\bfchi_\calD=\bfchi_\calQ=\tfrac12(\bfone+\bfc)$ when $\calD=\calQ$ and $\bfchi_\calD=\bfchi_{\calQ^\rmc}=\tfrac12(\bfone-\bfc)$ when $\calD=\calQ^\rmc$.
In either case, $\bfchi_\calD=\tfrac12(\bfone+\varepsilon\bfc)$ where $\varepsilon:=1$ when $\calD=\calQ$ and $\varepsilon:=-1$ when $\calD=\calQ^\rmc$.
Thus $\bfchi_{\calD^\rmc}=\tfrac12(\bfone-\varepsilon\bfc)$.
We caution that this $\varepsilon$ is distinct from the sign~\eqref{eq.sign of quadratic} of $\rmQ$, which arises here too:
\begin{equation}
\label{eq.pf of symplectic sub-ETFs 0}
\sum_{\bfv\in\calV}(-1)^{\rmQ(\bfv)}
=\sum_{\bfv\in\calV}[2\bfchi_\calQ(\bfv)-1]
=2\#(\calQ)-2^{2J}
=\sign(\rmQ)2^{J}.
\end{equation}
Remarkably, $\rmc$ is an eigenvector of $\bfGamma$ with this number as its eigenvalue: for any $\bfv_1\in\calV$, \eqref{eq.quadratic symplectic relation} gives
\begin{equation*}
(\bfGamma\bfc)(\bfv_1)
=\sum_{\bfv_2\in\calV}(-1)^{\rmB(\bfv_1,\bfv_2)}(-1)^{\rmQ(\bfv_2)}
=\sum_{\bfv_2\in\calV}(-1)^{\rmQ(\bfv_1+\bfv_2)+\rmQ(\bfv_1)}
=\sign(\rmQ)2^{J}\bfc(\bfv_1).
\end{equation*}
Since $\bfGamma$ is a Hadamard matrix of size $2^{2J}$ with an all-ones $\bfzero$th row, $\bfGamma\bfone=2^{2J}\bfdelta_\bfzero$, and so
\begin{equation}
\label{eq.pf of symplectic sub-ETFs 1}
\bfGamma\bfchi_\calD
=\bfGamma\tfrac12(\bfone+\varepsilon\bfc)
=\tfrac12(\bfGamma\bfone+\varepsilon\bfGamma\bfc)
=\tfrac12[2^{2J}\bfdelta_{\bfzero}+\varepsilon\sign(\rmQ)2^{J}\bfc]
=2^{J-1}[2^{J}\bfdelta_{\bfzero}+\varepsilon\sign(\rmQ)\bfc].
\end{equation}
Evaluating this at $\bfv=\bfzero$ gives
\begin{equation}
\label{eq.pf of symplectic sub-ETFs 2}
\#(\calD)
=2^{J-1}[2^J+\varepsilon\sign(\rmQ)].
\end{equation}
Now define $(\bfzeta_{\bfv})_{\bfv\in\calV}$ in $\bbR^\calD$ by
$\bfzeta_{\bfv_2}(\bfv_1)
:=2^{-\frac12(J-1)}(-1)^{\rmB(\bfv_1,\bfv_2)}$.
For any $\bfv_1,\bfv_2\in\calV$,
\begin{equation*}
2^{J-1}\ip{\bfzeta_{\bfv_1}}{\bfzeta_{\bfv_2}}
=\sum_{\bfv\in\calD}(-1)^{\rmB(\bfv,\bfv_1)}(-1)^{\rmB(\bfv,\bfv_2)}
=\sum_{\bfv\in\calV}(-1)^{\rmB(\bfv,\bfv_1+\bfv_2)}\bfchi_\calD(\bfv)
=(\bfGamma\bfchi_\calD)(\bfv_1+\bfv_2).
\end{equation*}
For any $\bfv_1,\bfv_2\in\calV$,
combining this with~\eqref{eq.pf of symplectic sub-ETFs 1} gives
\begin{equation}
\label{eq.pf of symplectic sub-ETFs 3}
\ip{\bfzeta_{\bfv_1}}{\bfzeta_{\bfv_2}}
=[2^{J}\bfdelta_{\bfzero}+\varepsilon\sign(\rmQ)\bfc](\bfv_1+\bfv_2)
=\left\{\begin{array}{cl}
2^J+\varepsilon\sign(\rmQ),&\ \bfv_1=\bfv_2,\\
\varepsilon\sign(\rmQ)(-1)^{\rmQ(\bfv_1+\bfv_2)},&\ \bfv_1\neq\bfv_2.
\end{array}\right.
\end{equation}
As such, $(\bfzeta_{\bfv})_{\bfv\in\calV}$ is equiangular.
It is moreover a tight frame for $\bbR^\calD$ since for any $\bfv_1,\bfv_2\in\calD$,
\begin{equation*}
\biggparen{\,\sum_{\bfv\in\calV}\bfzeta_{\bfv}^{}\bfzeta_{\bfv}^*}(\bfv_1,\bfv_2)
=\tfrac1{2^{J-1}}\sum_{\bfv\in\calV}(-1)^{\rmB(\bfv_1,\bfv)}(-1)^{\rmB(\bfv_2,\bfv)}
=\tfrac1{2^{J-1}}(\bfGamma\bfGamma^*)(\bfv_1,\bfv_2)
=2^{J+1}\brI(\bfv_1,\bfv_2).
\end{equation*}

Now consider the subsequence \smash{$(\bfzeta_{\bfv})_{\bfv\in\calD^\rmc}$} of $(\bfzeta_{\bfv})_{\bfv\in\calV}$ that has the $\calQ\times\calQ^\rmc$ or $\calQ^\rmc\times\calQ$ submatrix of \smash{$2^{-\frac12(J-1)}\bfGamma$} as its synthesis map when $\varepsilon$ is $1$ or $-1$, respectively.
Regardless, $\rmQ(\bfv_1)=\rmQ(\bfv_2)$ for all $\bfv_1,\bfv_2\in\calD^\rmc$ and so also $\rmQ(\bfv_1+\bfv_2)=\rmB(\bfv_1,\bfv_2)$ for all such $\bfv_1,\bfv_2$.
By~\eqref{eq.pf of symplectic sub-ETFs 3}, the Gram matrix of $(\bfzeta_{\bfv})_{\bfv\in\calD^\rmc}$ is thus $2^J\brI+\varepsilon\sign(\rmQ)\bfGamma_{\calD^\rmc}$ where $\bfGamma_{\calD^\rmc}$ is the $\calD^\rmc\times\calD^\rmc$ submatrix of $\bfGamma$:
\begin{equation}
\label{eq.pf of symplectic sub-ETFs 4}
\ip{\bfzeta_{\bfv_1}}{\bfzeta_{\bfv_2}}
=\left\{\begin{array}{cl}
2^J+\varepsilon\sign(\rmQ),&\ \bfv_1=\bfv_2\\
\varepsilon\sign(\rmQ)(-1)^{\rmB(\bfv_1,\bfv_2)},&\ \bfv_1\neq\bfv_2
\end{array}\right\}
%=[2^J\brI+\varepsilon\sign(\rmQ)\bfGamma](\bfv_1,\bfv_2)
=[2^J\brI+\varepsilon\sign(\rmQ)\bfGamma_{\calD^\rmc}](\bfv_1,\bfv_2),
\end{equation}
for any $\bfv_1,\bfv_2\in\calD^\rmc$.
To compute the spectrum of $\bfGamma_{\calD^\rmc}$ and thus also of this Gram matrix,
note that since $\bfchi_{\calD^\rmc}=\frac12(\bfone-\varepsilon\bfc)$,
the $\calV\times\calV$ diagonal projection matrix whose diagonal entries indicate $\calD^\rmc$ can be expressed as $\bfP:=\tfrac12(\brI-\varepsilon\bfDelta)$ where $\bfDelta\in\bbR^{\calV\times\calV}$ is diagonal with $\bfDelta(\bfv,\bfv):=\rmc(\bfv)=(-1)^{\rmQ(\bfv)}$ for all $\bfv\in\calV$.
This matrix $\bfDelta$ has an unusual relationship with $\bfGamma$:
for any $\bfv_1,\bfv_3\in\calV$,
\begin{equation*}
(\bfGamma\bfDelta\bfGamma)(\bfv_1,\bfv_3)
=\sum_{\bfv_2\in\calV}\bfGamma(\bfv_1,\bfv_2)\bfDelta(\bfv_2,\bfv_2)\bfGamma(\bfv_2,\bfv_3)
=\sum_{\bfv_2\in\calV}(-1)^{\rmB(\bfv_1,\bfv_2)+\rmQ(\bfv_2)+\rmB(\bfv_2,\bfv_3)},
\end{equation*}
where~\eqref{eq.quadratic symplectic relation} gives,
for any $\bfv_1,\bfv_2,\bfv_3\in\calV$, that
\begin{align*}
\rmB(\bfv_1,\bfv_2)+\rmQ(\bfv_2)+\rmB(\bfv_2,\bfv_3)
&=\rmQ(\bfv_1+\bfv_2+\bfv_3)+\rmQ(\bfv_1+\bfv_3)\\
&=\rmQ(\bfv_1+\bfv_2+\bfv_3)+\rmQ(\bfv_1)+\rmB(\bfv_1,\bfv_3)+\rmQ(\bfv_3),
\end{align*}
and moreover~\eqref{eq.pf of symplectic sub-ETFs 0} gives
$\sum_{\bfv_2\in\calV}(-1)^{\rmQ(\bfv_1+\bfv_2+\bfv_3)}=\sign(\rmQ)2^{J}$ for any $\bfv_1,\bfv_3\in\calV$,
implying
\begin{equation*}
(\bfGamma\bfDelta\bfGamma)(\bfv_1,\bfv_3)
=\sum_{\bfv_2\in\calV}(-1)^{\rmQ(\bfv_1+\bfv_2+\bfv_3)+\rmQ(\bfv_1)+\rmB(\bfv_1,\bfv_3)+\rmQ(\bfv_3)}
=\sign(\rmQ)2^{J}(\bfDelta\bfGamma\bfDelta)(\bfv_1,\bfv_3).
\end{equation*}
Thus,
$\bfGamma\bfDelta\bfGamma
=\sign(\rmQ)2^{J}\bfDelta\bfGamma\bfDelta$,
that is, the \textit{chirp--Fourier} transform $\bfDelta\bfGamma$
cubes to $\sign(\rmQ)2^{3J}\brI$.
Combining this with the facts that $\bfGamma^2=2^{2J}\brI$, $\bfDelta^2=\brI$ and $(\brI-\varepsilon\bfDelta)^2=2(\brI-\varepsilon\bfDelta)$ gives
\begin{align*}
&\tfrac12[(\brI-\varepsilon\bfDelta)\bfGamma(\brI-\varepsilon\bfDelta)]^2\\
&\qquad=(\brI-\varepsilon\bfDelta)\bfGamma(\brI-\varepsilon\bfDelta)\bfGamma(\brI-\varepsilon\bfDelta)\\
&\qquad=
\brI\bfGamma\brI\bfGamma\brI
-\varepsilon\brI\bfGamma\brI\bfGamma\bfDelta
-\varepsilon\brI\bfGamma\bfDelta\bfGamma\brI
-\varepsilon\bfDelta\bfGamma\brI\bfGamma\brI
+\brI\bfGamma\bfDelta\bfGamma\bfDelta
+\bfDelta\bfGamma\brI\bfGamma\bfDelta
+\bfDelta\bfGamma\bfDelta\bfGamma\brI
-\varepsilon\bfDelta\bfGamma\bfDelta\bfGamma\bfDelta\\
&\qquad=2^{2J+1}(\brI-\varepsilon\bfDelta)
+\sign(\rmQ)2^{J}
[-\varepsilon\bfDelta\bfGamma\bfDelta+(\bfDelta\bfGamma+\bfGamma\bfDelta)-\varepsilon\bfGamma]\\
&\qquad=2^{2J+1}(\brI-\varepsilon\bfDelta)
-\varepsilon\sign(\rmQ)2^{J}
(\brI-\varepsilon\bfDelta)\bfGamma(\brI-\varepsilon\bfDelta).
\end{align*}
Dividing this equation by $2^3$ and recalling that $\bfP=\tfrac12(\brI-\varepsilon\bfDelta)$ gives
\begin{equation*}
(\bfP\bfGamma\bfP)^2
=2^{2J-1}\bfP-\varepsilon\sign(\rmQ)2^{J-1}\bfP\bfGamma\bfP.
\end{equation*}
Since $\bfP$ is the diagonal projection onto
$\Span(\bfdelta_\bfv)_{\bfv\in\calD^\rmc}$ and $\bfGamma_{\calD^\rmc}$ is the $\calD^\rmc\times\calD^\rmc$ submatrix of $\bfGamma$, this implies
$\bfGamma_{\calD^\rmc}^2
=2^{2J-1}\brI-\varepsilon\sign(\rmQ)2^{J-1}\bfGamma_{\calD^\rmc}$.
Thus,
\begin{equation*}
[\varepsilon\sign(\rmQ)\bfGamma_{\calD^\rmc}-2^{J-1}\brI][\varepsilon\sign(\rmQ)\bfGamma_{\calD^\rmc}+2^{J}\brI]
=\bfGamma_{\calD^\rmc}^2+\varepsilon\sign(\rmQ)2^{J-1}\bfGamma_{\calD^\rmc}-2^{2J-1}\brI
=\bfzero,
\end{equation*}
implying that every eigenvalue of $\varepsilon\sign(\rmQ)\bfGamma_{\calD^\rmc}$ is either $2^{J-1}$ or $-2^J$.
Letting $M$ be the multiplicity of the former,
the multiplicity of the latter is $N-M$ where $N:=\#(\calD^\rmc)=2^{J-1}[2^J-\varepsilon\sign(\rmQ)]$,
cf.\ \eqref{eq.pf of symplectic sub-ETFs 2}.
Since every diagonal entry of $\bfGamma_{\calD^\rmc}$ is $1$,
taking the trace of $\varepsilon\sign(\rmQ)\bfGamma_{\calD^\rmc}$ gives
\begin{equation*}
\varepsilon\sign(\rmQ)N
=\Tr(\varepsilon\sign(\rmQ)\bfGamma_{\calD^\rmc})
=2^{J-1}M-2^J(N-M)
=3(2^{J-1})M-2^{J}N.
\end{equation*}
Solving for $M$ reveals that it is notably invariant of both $\sign(\rmQ)$ and $\varepsilon$:
\begin{equation*}
M
=\tfrac13[2^J+\varepsilon\sign(\rmQ)]2^{-(J-1)}N
=\tfrac13[2^J+\varepsilon\sign(\rmQ)][2^J-\varepsilon\sign(\rmQ)]
=\tfrac13(2^{2J}-1).
\end{equation*}
Recalling~\eqref{eq.pf of symplectic sub-ETFs 4},
the Gram matrix $2^{J}\brI+\varepsilon\sign(\rmQ)\bfGamma_{\calD^\rmc}$ of \smash{$(\bfzeta_{\bfv})_{\bfv\in\calD^\rmc}$}
thus has eigenvalues $3(2^{J-1})$ and $0$, the former with multiplicity $\tfrac13(2^{2J}-1)$.
Thus, \smash{$(\bfzeta_{\bfv})_{\bfv\in\calD^\rmc}$} is an ETF for a space of dimension $\tfrac13(2^{2J}-1)$.
The Gram matrix of its dual \smash{$(\widehat{\bfzeta}_{\bfv})_{\bfv\in\calD^\rmc}$} is
\begin{equation}
\label{eq.pf of symplectic sub-ETFs 5}
3(2^{J-1})\brI-[2^{J}\brI+\varepsilon\sign(\rmQ)\bfGamma_{\calD^\rmc}]
=2^{J-1}\brI-\varepsilon\sign(\rmQ)\bfGamma_{\calD^\rmc},
\end{equation}
which has eigenvalues $3(2^{J-1})$ and $0$, the former having multiplicity
\begin{align}
\nonumber
N-M
&=2^{J-1}[2^J-\varepsilon\sign(\rmQ)]-\tfrac13[2^{J}+\varepsilon\sign(\rmQ)][2^J-\varepsilon\sign(\rmQ)]\\
\label{eq.pf of symplectic sub-ETFs 6}
&=\tfrac13[2^{J-1}-\varepsilon\sign(\rmQ)][2^J-\varepsilon\sign(\rmQ)].
\end{align}
Thus, \smash{$(\widehat{\bfzeta}_{\bfv})_{\bfv\in\calD^\rmc}$} is an ETF for a space of this dimension.

We now require that $\#(\calD)=2^{J-1}(2^J+1)$ and let $(\bfphi_\bfv)_{\bfv\in\calV}$ and $(\bfpsi_\bfv)_{\bfv\in\calV}$ be the instances of $(\bfzeta_{\bfv})_{\bfv\in\calV}$
where ``$\calD$'' is $\calD$ and $\calD^\rmc$, respectively.
By Definition~\ref{def.affine quadric},
the latter is itself an affine quadric of cardinality $\#(\calD^\rmc)=2^{J-1}(2^J-1)$.
In these two instances,
\eqref{eq.pf of symplectic sub-ETFs 2} implies that ``$\varepsilon\sign(\rmQ)$'' is $1$ and $-1$, respectively, and so~\eqref{eq.pf of symplectic sub-ETFs 3} gives
that these ETFs for $\bbR^\calD$ and $\bbR^{\calD^\rmc}$ satisfy
\begin{equation*}
\ip{\bfphi_{\bfv_1}}{\bfphi_{\bfv_2}}
=\left\{\begin{array}{cl}
2^J+1,&\ \bfv_1=\bfv_2,\\
(-1)^{\rmQ(\bfv_1+\bfv_2)},&\ \bfv_1\neq \bfv_2,
\end{array}\right.
\quad
\ip{\bfpsi_{\bfv_1}}{\bfpsi_{\bfv_2}}
=\left\{\begin{array}{cl}
2^J-1,&\ \bfv_1=\bfv_2,\\
-(-1)^{\rmQ(\bfv_1+\bfv_2)},&\ \bfv_1\neq \bfv_2.
\end{array}\right.
\end{equation*}
Since $\ip{\bfpsi_{\bfv_1}}{\bfpsi_{\bfv_2}}
=-\ip{\bfphi_{\bfv_1}}{\bfphi_{\bfv_2}}$ whenever $\bfv_1\neq\bfv_2$,
these ETFs are dual.
Moreover, \eqref{eq.quadratic symplectic relation} gives
\begin{equation*}
(-1)^{\rmQ(\bfv_1+\bfv_2)}
=(-1)^{\rmB(\bfv_1,\bfv_2)+\rmQ(\bfv_1)+\rmQ(\bfv_2)}
=\overline{z_{\bfv_1}}z_{\bfv_2}(-1)^{\rmB(\bfv_1,\bfv_2)}
\end{equation*}
for all $\bfv_1,\bfv_2\in\calV$ where $(z_\bfv)_{\bfv\in\calV}$, $z_{\bfv}:=(-1)^{\rmQ(\bfv)}$ is a finite sequence of unimodular scalars.
Thus $(\bfphi_\bfv)_{\bfv\in\calV}$ and $(\bfpsi_\bfv)_{\bfv\in\calV}$ are equivalent---via these scalars---to the symplectic ETF and its dual that arise as the corresponding instance of Definition~\ref{def.symplectic ETF}, respectively.
Moreover, since $(\bfphi_\bfv)_{\bfv\in\calD^\rmc}$ and $(\bfpsi_\bfv)_{\bfv\in\calD}$ are the instances of $(\bfzeta_{\bfv})_{\bfv\in\calD^\rmc}$ when ``$\calD$'' is $\calD$ and $\calD^\rmc$, respectively, both are ETFs for spaces of dimension $\frac13(2^{2J}-1)$.
Applying~\eqref{eq.pf of symplectic sub-ETFs 4}, \eqref{eq.pf of symplectic sub-ETFs 5}, \eqref{eq.pf of symplectic sub-ETFs 6} in these two instances---where again ``$\varepsilon\sign(\rmQ)$'' is $1$ and $-1$, respectively---gives the remaining claims made in (b) and (c), respectively.

\section{Proof of Lemma~\ref{lemma.DT of Sp}}

When $J=1$, $\calV$ has exactly three nonzero vectors, any two of which form a singular pair, implying $\Sp(\rmB)=\GL(\calV)$.
Here, there is just one quadratic form $\rmQ$ that yields $\rmB$ of negative sign,
having quadric $\calQ=\set{\bfzero}$,
and the action of $\Sp(\rmB)$ on the set of all such forms is doubly transitive in a vacuous sense.
There are three quadratic forms that yield $\rmB$ of positive sign,
each uniquely defined by its nonsingular vector, which may be any nonzero vector in $\calV$.
Here, \eqref{eq.Sp action on quadratic forms} is doubly transitive since any pair of distinct nonzero vectors in $\calV$ can be mapped to any other such pair via a member of
$\Sp(\rmB)=\GL(\calV)$.
As such, we assume $J\geq 2$ for the remainder of the proof.
When $J=2$, all products below of the form ``$\prod_{j=1}^{J-2}$'' should be regarded as $1$.

Let $\rmQ$ be a quadratic form that yields $\rmB$ with desired sign~\eqref{eq.sign of quadratic}.
By~\eqref{eq.other quadratic forms}, the quadratic forms that yield $\rmB$ are parameterized by $\bfv_\bfL\in\calV$,
being of the form $\bfv\mapsto\rmQ_\bfL(\bfv):=\rmQ(\bfv)+\rmB(\bfv_{\bfL},\bfv)$.
By \eqref{eq.sign of quadratic}, $\sign(\rmQ_\bfL)=\sign(\rmQ)$ if and only if the quadric $\calQ_\bfL$ of $\rmQ_\bfL$ has the same size as the quadric $\calQ$ of $\rmQ$.
By~\eqref{eq.quadrics of other quadratic forms}, this occurs if and only if $\bfv_\bfL\in\calQ$.
The number of quadratic forms that yield $\rmB$ and have the same sign as $\rmQ$ is thus $\#(\calQ)=2^{J-1}[2^J+\sign(\rmQ)]$, and so
\begin{equation}
\label{eq.pf of DT of Sp 0}
\#(\calQ)-1
=2^{J-1}[2^J+\sign(\rmQ)]-1
=[2^{J-1}+\sign(\rmQ)][2^{J}-\sign(\rmQ)].
\end{equation}
The number of ordered pairs $(\rmQ_1,\rmQ_2)$ of distinct quadratic forms that yield $\rmB$ and have the same sign as $\rmQ$ is thus
\begin{align}
\nonumber
\#(\calQ)[\#(\calQ)-1]
&=2^{J-1}[2^J+\sign(\rmQ)][2^{J-1}+\sign(\rmQ)][2^{J}-\sign(\rmQ)]\\
\label{eq.pf of DT of Sp 1}
&=2^{J-1}(2^{2J}-1)[2^{J-1}+\sign(\rmQ)].
\end{align}
To show that the aforementioned action $(\bfA,\rmQ')\mapsto\bfA\cdot\rmQ'$ is doubly transitive,
it thus suffices to show that the size of the orbit of any particular such pair $(\rmQ_1,\rmQ_2)$ under the induced action
$\bfA\cdot(\rmQ_1,\rmQ_2):=(\bfA\cdot\rmQ_1,\bfA\cdot\rmQ_2)$ is given by~\eqref{eq.pf of DT of Sp 1}.
Here, by redefining ``$\rmQ$'' to be $\rmQ_1$,
we may assume without loss of generality that $(\rmQ_1,\rmQ_2)=(\rmQ,\rmQ_\bfL)$ where $\bfv_\bfL\in\calQ$ and $\bfv_\bfL\neq\bfzero$ since $\sign(\rmQ_1)=\sign(\rmQ_2)$ and $\rmQ_1\neq\rmQ_2$, respectively.

We compute the size of this orbit via its stabilizer.
Here, recalling that $\bfA\in\Sp(\rmB)$ if and only if $\bfA:\calV\rightarrow\calV$ is linear and maps any particular symplectic basis for $\calV$ to another,
the order of $\Sp(\rmB)$ is simply the number of (ordered) symplectic bases for $\calV$.
Further recall that any such basis is constructed by iteratively choosing (ordered) symplectic pairs from the orthogonal complement of those already chosen.
When $\calV$ has dimension $2J$, any $\bfv_1\in\calV\backslash\set{\bfzero}$ can serve as the first member of such a pair, and any $\bfv_2\in\calV\backslash\set{\bfv_1}^\perp$ can serve as the second, implying the number of such pairs is $(2^{2J}-1)(2^{2J}-2^{2J-1})=2^{2J-1}(2^{2J}-1)$.
Thus, $\Sp(\rmB)$ has order
\begin{equation}
\label{eq.pf of DT of Sp 2}
\prod_{j=1}^{J}[2^{2j-1}(2^{2j}-1)]\\
=2^{4J-4}(2^{2J}-1)[2^{J-1}+\sign(\rmQ)][2^{J-1}-\sign(\rmQ)]
\prod_{j=1}^{J-2}[2^{2j-1}(2^{2j}-1)].
\end{equation}
By the orbit--stabilizer theorem, the size of
$\Orb(\rmQ,\rmQ_\bfL)=\set{(\bfA\cdot\rmQ,\bfA\cdot\rmQ_\bfL): \bfA\in\Sp(\rmB)}$ is thus given by~\eqref{eq.pf of DT of Sp 1} if and only if the order of
$\Stab(\rmQ,\rmQ_\bfL)
=\set{\bfA\in\Sp(\rmB): (\bfA\cdot\rmQ,\bfA\cdot\rmQ_\bfL)=(\rmQ,\rmQ_\bfL)}$
equals the number obtained by dividing~\eqref{eq.pf of DT of Sp 2} by~\eqref{eq.pf of DT of Sp 1}. We next simplify this stabilizer.

In general, for any quadratic form $\rmQ$ that yields $\rmB$,
the corresponding \textit{orthogonal group} $\Orth(\rmQ)$ is the set of all linear operators $\bfA$ on $\calV$ such that $\rmQ(\bfA\bfv)=\rmQ(\bfv)$ for all $\bfv\in\calV$.
Here~\eqref{eq.quadratic symplectic relation} gives
\begin{equation*}
\rmB(\bfA\bfv_1,\bfA\bfv_2)
=\rmQ(\bfA(\bfv_1+\bfv_2))+\rmQ(\bfA\bfv_1)+\rmQ(\bfA\bfv_2)
=\rmQ(\bfv_1+\bfv_2)+\rmQ(\bfv_1)+\rmQ(\bfv_2)
=\rmB(\bfv_1,\bfv_2)
\end{equation*}
for all $\bfv_1,\bfv_2\in\calV$ and so $\Orth(\rmQ)$ is a subset of $\Sp(\rmB)$, which itself is a subgroup of $\GL(\calV)$.
This makes it straightforward to show that $\Orth(\rmQ)$ is itself a subgroup of $\Sp(\rmB)$.
Notably, $\bfA\cdot\rmQ=\rmQ$ if and only if $\rmQ(\bfA^{-1}\bfv)=\rmQ(\bfv)$ for all $\bfv\in\calV$, that is, $\bfA^{-1}\in\Orth(\rmQ)$, namely $\bfA\in\Orth(\rmQ)$.
As such, $\Stab(\rmQ,\rmQ_\bfL)=\Orth(\rmQ)\cap\Orth(\rmQ_\bfL)$ is an intersection of two of these orthogonal subgroups of $\Sp(\rmB)$.
Altogether, simplifying the quotient of~\eqref{eq.pf of DT of Sp 2} and~\eqref{eq.pf of DT of Sp 1},
in order to show that $(\bfA,\rmQ')\mapsto\bfA\cdot\rmQ'$ is a doubly transitive action of $\Sp(\rmB)$ on the set of all quadratic forms of a given sign that yield $\rmB$,
it suffices to take any such $\rmQ$, any $\bfv_\bfL\in\calQ\backslash\set{\bfzero}$,
define $\rmQ_\bfL$ by~\eqref{eq.other quadratic forms}, and show that
\begin{equation}
\label{eq.pf of DT of Sp 3}
\#(\Orth(\rmQ)\cap\Orth(\rmQ_\bfL)) =2^{3J-3}[2^{J-1}-\sign(\rmQ)]\prod_{j=1}^{J-2}[2^{2j-1}(2^{2j}-1)].
\end{equation}

To count the members of this intersection,
we claim that for any symplectic basis $(\bfv_k)_{k=1}^{2J}$ and any linear operator $\bfA$ on $\calV$, we have $\bfA\in\Orth(\rmQ)$ if and only if $(\bfA\bfv_k)_{k=1}^{2J}$ is a symplectic basis $\calV$ and $\rmQ(\bfA\bfv_k)=\rmQ(\bfv_k)$ for all $k$.
Indeed,
the former clearly implies the latter since $\Orth(\rmQ)$ is a subgroup of $\Sp(\rmB)$.
Conversely,
decomposing any $\bfv\in\calV$ as $\bfv=\sum_{k=1}^{2J}x_k\bfv_k$ for some scalars $(x_k)_{k=1}^{2J}$ in $\bbF_2$,
note that by~\eqref{eq.quadratic form over characteristic two} and the orthogonality of distinct symplectic pairs,
\begin{align*}
\rmQ(\bfA\bfv)
&=\sum_{j=1}^J\rmQ(x_{2j-1}\bfA\bfv_{2j-1}+x_{2j}\bfA\bfv_{2j})\\
&=\sum_{j=1}^J [x_{2j-1}^2\rmQ(\bfA\bfv_{2j-1})
+x_{2j-1}x_{2j}\rmB(\bfA\bfv_{2j-1},\bfA\bfv_{2j})
+x_{2j}^2\rmQ(\bfA\bfv_{2j})].
\end{align*}
Disregarding $\bfA$ and recombining yields $\rmQ(\bfv)$, yielding the claim.
As such, to find the order of $\Orth(\rmQ)\cap\Orth(\rmQ_\bfL)$,
we need only choose any symplectic basis $(\bfv_k)_{k=1}^{2J}$ for $\calV$ and then count the symplectic bases $(\bfv_k')_{k=1}^{2J}$ for $\calV$ that have
$\rmQ(\bfv_k')=\rmQ(\bfv_k)$ and $\rmQ_\bfL(\bfv_k')=\rmQ_\bfL(\bfv_k)$ for all $k$.
To do so, it is convenient to design $(\bfv_k)_{k=1}^{2J}$ so that it has a simple relationship with $\bfv_\bfL$ and is also as singular as reasonably possible with respect to both $\rmQ$ and $\rmQ_\bfL$.

We now count the symplectic pairs $(\bfv_1,\bfv_2)$ that are $\calQ$-singular.
Here~\eqref{eq.pf of DT of Sp 0} gives
$[2^{J-1}+\sign(\rmQ)][2^{J}-\sign(\rmQ)]$ choices of $\bfv_1\in\calQ\backslash\set{\bfzero}$.
For any such $\bfv_1$, take any $\bfv_2\in\calQ$ such that $\rmB(\bfv_1,\bfv_2)=1$.
(As discussed in Subsection~\ref{subsection.quadratic forms}, such a $\bfv_2$ necessarily exists.)
Then uniquely decomposing any $\bfv\in\calV$ as $\bfv=x_1\bfv_1+x_2\bfv_2+\bfv_0$ where $x_1,x_2\in\bbF_2$, $\bfv_0\in\calV_0:=\set{\bfv_1,\bfv_2}^\perp$,
we have $\rmB(\bfv_1,\bfv)=1$ if and only if $x_2=1$.
Moreover,
\eqref{eq.quadratic form recursion} gives $\rmQ(x_1\bfv_1+\bfv_2+\bfv_0)=x_1+\rmQ_0(\bfv_0)$.
Letting $\calQ_0$ be the quadric of $\rmQ_0$,
we thus have that
$\set{\bfv\in\calQ: \rmB(\bfv_1,\bfv)=1}
=(\bfv_2+\calQ_0)\sqcup(\bfv_1+\bfv_2+\calV_0\backslash\calQ_0)$
is a set of cardinality $\#(\calV_0)=2^{2J-2}$.
That is, for any $\bfv_1\in\calQ\backslash\set{\bfzero}$,
\begin{equation}
\label{eq.pf of DT of Sp 4}
\#\set{\bfv\in\calQ: \rmB(\bfv_1,\bfv)=1}
=2^{2J-2}.
\end{equation}
Altogether, the number of $\calQ$-singular symplectic pairs in the $2J$-dimensional space $\calV$ is
\begin{align}
\nonumber
\#\set{(\bfv_1,\bfv_2)\in\calQ\times\calQ: \rmB(\bfv_1,\bfv_2)=1}
&=2^{2J-2}[2^{J-1}+\sign(\rmQ)][2^{J}-\sign(\rmQ)]\\
\label{eq.pf of DT of Sp 5}
&=2^{2J-2}(2^{2J-2}-1)\frac{2^{J}-\sign(\rmQ)}{2^{J-1}-\sign(\rmQ)}.
\end{align}

Now let $(\bfv_1,\bfv_2)$ be any symplectic pair with $\bfv_1=\bfv_\bfL$ that is $\rmQ$-singular,
and complete it to a symplectic basis $(\bfv_k)_{k=1}^{2J}$ for $\calV$ that is as $\rmQ$-singular as possible,
having $\rmQ(\bfv_k)=0$ for all $k\in[2J-2]$,
and either $\rmQ(\bfv_{2J-1})=\rmQ(\bfv_{2J})=0$ or $\rmQ(\bfv_{2J-1})=\rmQ(\bfv_{2J})=1$ when $\sign(\rmQ)$ is $1$ or $-1$, respectively.
Here $\rmQ_\bfL(\bfv_k)=\rmQ(\bfv_k)$ if and only if $\rmB(\bfv_1,\bfv_k)=0$,
namely if and only if $k\neq 2$.

As noted above,
the order of $\Orth(\rmQ)\cap\Orth(\rmQ_\bfL)$ is thus the number of symplectic bases $(\bfv_k')_{k=1}^{2J}$ for $\calV$ with $\rmQ(\bfv_k')=\rmQ(\bfv_k)$ and $\rmQ_\bfL(\bfv_k')=\rmQ_\bfL(\bfv_k)$ for all $k\in[2J]$.
For any such basis and any $k\in[2J]$,
we have $\rmB(\bfv_\bfL,\bfv_k')=0$ if and only if
$\rmQ(\bfv_k')+\rmB(\bfv_\bfL,\bfv_k')=\rmQ_\bfL(\bfv_k')=\rmQ_\bfL(\bfv_k)$
equals $\rmQ(\bfv_k')=\rmQ(\bfv_k)$,
namely if and only if $k\neq 2$.
Decomposing $\bfv_L$ in terms of any such basis $(\bfv_k')_{k=1}^{2J}$,
and distributing $\rmB(\bfv_\bfL,\bfv_2')=1$ and $\rmB(\bfv_\bfL,\bfv_k')=0$ over this decomposition for all $k\neq 2$ gives that necessarily $\bfv_1'=\bfv_\bfL=\bfv_1$.
For any symplectic basis $(\bfv_k')_{k=1}^{2J}$ such that $\bfv_1'=\bfv_\bfL$ and $\rmQ(\bfv_k')=\rmQ(\bfv_k)$ whenever $k\geq 2$, we claim that $\rmQ_\bfL(\bfv_k')=\rmQ_\bfL(\bfv_k)$ for all $k\in[2J]$.
Indeed, since $\bfv_1'=\bfv_\bfL=\bfv_1$, this holds for $k=1$.
For $k=2$,
summing the equations $\rmQ(\bfv_2')=\rmQ(\bfv_2)$ and
$\rmB(\bfv_\bfL,\bfv_2')
=\rmB(\bfv_1',\bfv_2')
=1
=\rmB(\bfv_1,\bfv_2)
=\rmB(\bfv_\bfL,\bfv_2)$
gives $\rmQ_\bfL(\bfv_2')=\rmQ_\bfL(\bfv_2)$.
Similarly, for any $k\geq 3$, summing $\rmQ(\bfv_k')=\rmQ(\bfv_k)$
and
$\rmB(\bfv_\bfL,\bfv_k')
=\rmB(\bfv_1',\bfv_k')
=0
=\rmB(\bfv_1,\bfv_k)
=\rmB(\bfv_\bfL,\bfv_k)$
gives $\rmQ_\bfL(\bfv_k')=\rmQ_\bfL(\bfv_k)$, yielding the claim.
In summary, the order of $\Orth(\rmQ)\cap\Orth(\rmQ_\bfL)$ is the number of symplectic bases
$(\bfv_k')_{k=1}^{2J}$ for $\calV$ such that $\bfv_1'=\bfv_\bfL$,
$\rmQ(\bfv_k')=0$ whenever $3\leq k\leq 2J-2$,
and such that either $\rmQ(\bfv_{2J-1})=\rmQ(\bfv_{2J})=0$ or $\rmQ(\bfv_{2J-1})=\rmQ(\bfv_{2J})=1$ when $\sign(\rmQ)$ is $1$ or $-1$, respectively.
By~\eqref{eq.pf of DT of Sp 4}, there are $2^{2J-2}$ choices of $\bfv_2'$.
For any $j\in\set{2,\dotsc,J-1}$,
the number of choices of suitable $(\bfv_{2j-1},\bfv_{2j})$ is given by~\eqref{eq.pf of DT of Sp 5} where ``$J$'' is $J-(j-1)=J-j+1$.
Since the final pair $(\bfv_{2J-1},\bfv_{2J})$ is chosen from a space that contains either two nonzero singular vectors or three nonsingular vectors, depending on whether $\sign(\rmQ)$ is $1$ or $-1$, respectively, there are either $(2)(1)=2$ or $(3)(2)=6$ choices for this pair in these respective cases.
In either case, there are thus $2[2-\sign(\rmQ)]$ choices for $(\bfv_{2J-1},\bfv_{2J})$.
Simplifying the resulting telescoping product reveals that the order of $\Orth(\rmQ)\cap\Orth(\rmQ_\bfL)$ satisfies~\eqref{eq.pf of DT of Sp 3},
implying that the action $(\bfA,\rmQ')\mapsto\bfA\cdot\rmQ'$ is indeed doubly transitive:
\begin{multline*}
2^{2J-2}\{2[2-\sign(\rmQ)]\}
\prod_{j=2}^{J-1}\biggbracket{2^{2j-2}(2^{2j-2}-1)\frac{2^{j}-\sign(\rmQ)}{2^{j-1}-\sign(\rmQ)}}\\
=2^{2J-1}[2^{J-1}-\sign(\rmQ)]\prod_{j=2}^{J-1}[2^{2j-2}(2^{2j-2}-1)]
%=2^{2J-1}[2^{J-1}-\sign(\rmQ)]\prod_{j=1}^{J-2}[2^{2j}(2^{2j}-1)]
=2^{3J-3}[2^{J-1}-\sign(\rmQ)]\prod_{j=1}^{J-2}[2^{2j-1}(2^{2j}-1)].
\end{multline*}

\section{Proof of Lemma~\ref{lemma.totally singular}}

(a) We claim that if $\calS$ is a totally orthogonal subspace of $\calV$ then $\calS\cap\calQ$ is a totally singular subspace of $\calV$.
Clearly, $\calS\cap\calQ$ is totally singular.
It contains $\bfzero$ since $\bfzero\in\calS$ and $\bfzero\in\calQ$.
Since the underlying field is $\bbF_2$, this immediately implies that it is closed under scalar multiplication.
It is moreover closed under addition since $\rmQ(\bfv_1+\bfv_2)=\rmB(\bfv_1,\bfv_2)+\rmQ(\bfv_1)+\rmQ(\bfv_2)=0+0+0=0$ for all $\bfv_1,\bfv_2\in\calS\cap\calQ$,
yielding the claim.
We next claim that any totally orthogonal subspace $\calS$ of $\calV$ that is not contained in $\calQ$ is split perfectly in half by it.
Indeed, if there exists $\bfv_0\in\calS\backslash\calQ$, then
$\rmQ(\bfv+\bfv_0)=\rmB(\bfv,\bfv_0)+\rmQ(\bfv)+\rmQ(\bfv_0)=0+\rmQ(\bfv)+1$ for all $\bfv\in\calS$, implying $\bfv\mapsto\bfv+\bfv_0$ is an involution from $\calS\cap\calQ$ onto $\calS\backslash\calQ$.
Since $\calS\cap\calQ$ is a subgroup of $\calS$ with
$\calS
=(\calS\cap\calQ)\sqcup(\calS\backslash\calQ)
=(\calS\cap\calQ)\sqcup(\bfv_0+\calS\cap\calQ)$,
it has index $2$ and its nonidentity coset is $\calS\backslash\calQ$, yielding the claim.

For (b) and (c),
continue to let $\sign(\rmQ)$ be arbitrary,
and let $\calS$ be any totally singular subspace of $\calV$ of dimension $J-1$.
By (a), such subspaces exist, arising for example as any (not necessarily proper) subspace of the intersection of $\calQ$ with any Lagrangian subspace $\calL$ of $\calV$.
Clearly, there exists $\bfv_1\in\calS\backslash\set{\bfzero}$ if and only if $J\geq 2$.
When this occurs, any such $\bfv_1$ belongs to $\calQ\backslash\set{\bfzero}$ and so there exists $\bfv_2\in\calV$ such that $(\bfv_1,\bfv_2)$ is a hyperbolic pair.
Here necessarily $\bfv_2\in\calQ\backslash\calS$ since $\rmQ(\bfv_2)=0$ and $\rmB(\bfv_1,\bfv_2)=1$.
Further recall that any $\bfv\in\calV$ can be uniquely decomposed as $\bfv=x_1\bfv_1+x_2\bfv_2+\bfv_0$ where $x_1,x_2\in\bbF_2$ and $\bfv_0\in\calV_0:=\set{\bfv_1,\bfv_2}^\perp$.
When $\bfv\in\calS$, we necessarily have
$0=\rmB(\bfv_1,\bfv)=\rmB(x_1\bfv_1+x_2\bfv_2+\bfv_0)=x_2$,
implying $\bfv_0=\bfv+x_1\bfv_1$ is a linear combination of two elements of $\calS$,
and so $\bfv_0\in\calS_0:=\calS\cap\calV_0$.
Thus, $\calS=\set{x_1\bfv_1+\bfv_0: x_1\in\bbF, \bfv_0\in\calS_0}$.
Here since $\bfv_1\notin\calV_0$ and so $\bfv_1\notin\calS_0$,
we have $\dim(\calS_0)=\dim(\calS)-1=J-2$.
Since $\calS$ is totally singular with respect to $\rmQ$,
$\calS_0$ is totally singular with respect to its restriction $\rmQ_0$ to $\calV_0$:
by~\eqref{eq.quadratic form recursion},
$0=\rmQ(x_1\bfv_1+\bfv_0)=x_1(0)+\rmQ_0(\bfv_0)=\rmQ_0(\bfv_0)$ for all $\bfv_0\in\calS_0$.

Iteratively applying this idea yields a symplectic basis $(\bfv_k)_{k=1}^{2J}$ for $\calV$ where $(\bfv_{2j-1})_{j=1}^{J-1}$ is a basis for $\calS$,
$\rmQ(\bfv_k)=0$ for all $k\in[2J-2]$,
and \smash{$(-1)^{\rmQ(\bfv_{k})}=\sign(\rmQ)$} for $k\in\set{2J-1,2J}$.
Writing any \smash{$\bfv\in\calV$} as
\smash{$\bfv=\sum_{k=1}^{2J}x_k\bfv_k$} for some scalars \smash{$(x_k)_{k=1}^{2J}$} in $\bbF_2$,
we thus have \smash{$\bfv\in\calS^\perp$} if and only if $0=\rmB(\bfv_{2j-1},\bfv)=x_{2j}$ for all $j\in[J-1]$.
Moreover, when this occurs,
\smash{$\rmQ(\sum_{k=1}^{2J}x_k\bfv_k)$} is either $x_{2J-1}x_{2J}$ or $x_{2J-1}^{}x_{2J}^{}+x_{2J-1}^2+x_{2J}^2$ when $\sign(\rmQ)$ is $1$ or $-1$, respectively.
Now note that since $\dim(\calS^\perp)-\dim(\calS)=2$,
$\calS$ has three nonidentity cosets in $\calS^\perp$.
In light of the observations above, these cosets are $\calS_1:=\bfv_{2J-1}+\calS$,
$\calS_2:=\bfv_{2J}+\calS$ and $\calS_{1,2}:=\bfv_{2J-1}+\bfv_{2J}+\calS$.
Moreover,
$\calS_{1,2}$ is \textit{totally nonsingular}---none of its members are singular---while $\calS_1$ and $\calS_2$ are either both totally singular or both totally nonsingular when $\sign(\rmQ)$ is $1$ or $-1$, respectively.
Thus,
\begin{equation*}
\calS^\perp\backslash\calS
=\calS_1\sqcup\calS_2\sqcup\calS_{1,2},
\qquad
(\calS^\perp\backslash\calS)\cap\calQ
=\left\{\begin{array}{cl}
\calS_1\sqcup\calS_2,&\ \sign(\rmQ)>0,\\
\emptyset,&\ \sign(\rmQ)<0.
\end{array}\right.
\end{equation*}
Now note that by Lemma~\ref{lemma.totally orthogonal},
every Lagrangian subspace $\calL$ of $\calV$ is of the form $\calL=\Span\set{\calS,\bfv}$ where $\calS$ is a totally orthogonal subspace of dimension $J-1$ and $\bfv\in\calS^\perp\backslash\calS$.
We thus have that any totally singular subspace $\calS$ of dimension $J-1$ is contained in exactly three Lagrangian subspaces, namely $\calS\sqcup\calS_1$, $\calS\sqcup\calS_2$ and $\calS\sqcup\calS_{1,2}$,
exactly two of which---$\calS\sqcup\calS_1$ and $\calS\sqcup\calS_2$---are totally singular when $\sign(\rmQ)>0$,
and none of which are totally singular when $\sign(\rmQ)<0$.

In particular, when $\sign(\rmQ)<0$, no Lagrangian subspace $\calL$ of $\calV$ is totally singular since if it were,
then any $(J-1)$-dimensional subspace $\calS$ of it would be as well, implying that $\calL$ is one of the three Lagrangian subspaces that contains $\calS$, none of which are totally singular, a contradiction.
This completes the proof of (b).

For (c), now assume that $\sign(\rmQ)>0$ and that $\calL$ is a Lagrangian subspace of $\calV$.
If $\calL$ is totally singular, then $\calL=\bfzero+\calL$ is the only totally singular coset of $\calL$:
if $\calL^\perp=\calL\subseteq\calQ$ and $\bfv_1+\calL\subseteq\calQ$ for some $\bfv_1\in\calV$ then $\bfv_1=\bfv_1+\bfzero\in\bfv_1+\calL\subseteq\calQ$,
and so applying~\eqref{eq.quadratic symplectic relation} with any $\bfv_2\in\calL\subseteq\calQ$ gives
$0=\rmQ(\bfv_1+\bfv_2)
=\rmB(\bfv_1,\bfv_2)+\rmQ(\bfv_1)+\rmQ(\bfv_2)
=\rmB(\bfv_1,\bfv_2)$,
implying $\bfv_1\in\calL^\perp=\calL$ and so $\bfv+\calL=\calL$.
When $\calL$ is not totally singular, (a) gives that $\calS:=\calL\cap\calQ$ is a totally singular subspace of $\calV$ of dimension $\dim(\calL)-1=J-1$.
Constructing a hyperbolic basis $(\bfv_k)_{k=1}^{2J}$ for $\calV$ such that $(\bfv_{2j-1})_{j=1}^{J-1}$ is a basis for $\calS=\calL\cap\calQ$ and defining $\calS_1$, $\calS_2$ and $\calS_{1,2}$ from $\calS$ as above,
the fact that $\calL$ is not totally singular gives $\calL=\calS\sqcup\calS_{1,2}$.
There is thus indeed a coset of $\calL$ that is totally singular:
\begin{equation*}
\bfv_{2J-1}+\calL
=\bfv_{2J-1}+(\calS\sqcup\calS_{1,2})
=(\bfv_{2J-1}+\calS)\sqcup(\bfv_{2J-1}+\calS_{1,2})
=\calS_1\sqcup\calS_2
\subseteq\calQ.
\end{equation*}
To see that this coset is unique, note that if $\bfv_1+\calL\subseteq\calQ$ then $\bfv_1\in\calQ$ and so
$0
=\rmQ(\bfv_1+\bfv_2)
=\rmB(\bfv_1,\bfv_2)+\rmQ(\bfv_1)+\rmQ(\bfv_2)
=\rmB(\bfv_1,\bfv_2)$ for any $\bfv_2\in\calS$,
implying $\bfv_1\in\calS^\perp$.
At the same time, $\bfv_1\notin\calL$ since otherwise $\bfv_1+\calL=\calL\nsubseteq\calQ$.
As such, if $\bfv_1+\calL\subseteq\calQ$ then $\bfv_1\in\calS^\perp\backslash\calL=\calS_1\sqcup\calS_2$.
Now note that since $\set{\bfv_1\in\calV: \bfv_1+\calL\subseteq\calQ}$ is clearly closed under shifts by members of $\bfL$ it is itself a union of cosets of $\bfL$.
Since this set is nonempty, containing $\bfv_{2J-1}$, and is contained in $\calS_1\sqcup\calS_2$, we thus have
$\set{\bfv_1\in\calV: \bfv_1+\calL\subseteq\calQ}
=\calS_1\sqcup\calS_2=\bfv_{2J-1}+\calL$.
Thus, $\bfv_1+\calL\subseteq\calQ$ if and only if $\bfv_1+\calL=\bfv_{2J-1}+\calL$.

\section{Proof of Lemma~\ref{lemma.counting pairs}}

Recall that if $(\bfv_1,\bfv_2)$ is a symplectic pair in $\calV$ then any $\bfv\in\calV$ uniquely decomposes as $\bfv=x_1\bfv_1+x_2\bfv_2+\bfv_0$ for some $x_1,x_2\in\bbF_2$ and $\bfv_0\in\calV_0:=\set{\bfv_1,\bfv_2}^\perp$.
Further recall that when this occurs,
if $\rmQ$ is a quadratic form on $\calV$ that yields $\rmB$,
then its restriction $\rmQ_0$ to $\calV_0$ is a quadratic form that yields the restriction $\rmB_0$ of $\rmB$ to $\calV_0$ and moreover satisfies~\eqref{eq.quadratic form recursion in general} for any
$x_1,x_2\in\bbF_2$ and $\bfv_0\in\calV_0$.

If $J=1$ and $\sign(\rmQ)<0$,
the quantity in~\eqref{eq.lemma.counting pairs.1} is $0$,
which is indeed the number of symplectic pairs $(\bfv_1,\bfv_2)$ of members of $\calQ=\set{\bfzero}$.
So now assume that either $J\geq2$ or $J=1$ and $\sign(\rmQ)>0$.
Here, recall that there exists $\bfv_1\in\calQ\backslash\set{\bfzero}$,
and moreover that for any such $\bfv_1$,
there exists $\bfv_2\in\calQ$ such that $\rmB(\bfv_1,\bfv_2)=1$.
For any such hyperbolic pair $(\bfv_1,\bfv_2)$,
\eqref{eq.quadratic form recursion in general} reduces to~\eqref{eq.quadratic form recursion},
implying
\begin{equation*}
\rmB(\bfv_1,x_1\bfv_1+x_2\bfv_2+\bfv_0)
=x_2,
\quad
\rmQ(x_1\bfv_1+x_2\bfv_2+\bfv_0)=x_1x_2+\rmQ_0(\bfv_0),
\end{equation*}
for all $x_1,x_2\in\bbF_2$, $\bfv_0\in\calV_0$.
In particular, for any $\bfv_1\in\calQ\backslash\set{\bfzero}$,
the pair $(\bfv_1,\bfv)$ is hyperbolic if and only if $\bfv=x_1\bfv_1+x_2\bfv_2+\bfv_0$ where $x_2=1$ and $x_1x_2+\rmQ_0(\bfv_0)=0$.
For any $\bfv_0\in\calV_0$, this holds if and only if $x_1=\rmQ_0(\bfv_0)$ and $x_2=1$.
Thus, the number of hyperbolic pairs in $\calV$ is the product of the number $\#(\calQ\backslash\set{\bfzero})$ of choices of $\bfv_1$ with the number $\#(\calV_0)$ of choices for such $\bfv$, namely the value given in~\eqref{eq.lemma.counting pairs.1}:
\begin{equation*}
\#(\calQ\backslash\set{\bfzero})\#(\calV_0)
=\bigset{2^{J-1}[2^J+\sign(\rmQ)]-1}2^{2(J-1)}
=2^{2J-2}[2^{J-1}+\sign(\rmQ)][2^{J}-\sign(\rmQ)].
\end{equation*}
For any hyperbolic pair $(\bfv_1,\bfv_2)$ in $\calV$,
now recall that $\sign(\rmQ_0)=\sign(\rmQ)$ and that~\eqref{eq.quadric relationship} holds,
namely that $\calQ$ partitions into $\calQ_0$,
$\bfv_1+\calQ_0$,
$\bfv_2+\calQ_0$ and $\bfv_1+\bfv_2+\calV_0\backslash\calQ_0$.
Here since $\rmB(\bfv_1,\bfv_2)=1$ and $\calQ_0\subseteq\calV_0=\set{\bfv_1,\bfv_2}^\perp$,
the set $\bfv_1+\bfv_2+\calV_0\backslash\calQ_0$ is the only one of these four components whose members are nonorthogonal to both $\bfv_1$ and $\bfv_2$.
Moreover, $\rmB(\bfv_1+\bfv_2+\bfv_{01},\bfv_1+\bfv_2+\bfv_{02})=\rmB_0(\bfv_{01},\bfv_{02})$ for any $\bfv_{01},\bfv_{02}\in\calV_0\backslash\calQ_0$.
As such, if additional vectors in $\calQ$ are appended to $(\bfv_1,\bfv_2)$,
the members of the resulting tuple are pairwise nonorthogonal if and only if the appended vectors are obtained by adding $\bfv_1+\bfv_2$ to pairwise nonorthogonal vectors in $\calV_0\backslash\calQ_0$.

For the remaining claims, now let $\calQ_0$ be the quadric of an arbitrary quadratic form $\rmQ_0$ that yields a symplectic form $\rmB_0$ on a vector space $\calV_0$ over $\bbF_2$ of dimension $2J_0\geq 2$.
We claim that the number of nonsingular symplectic pairs in $\calV_0$,
namely the number of pairs $(\bfv_{01},\bfv_{02})$ of members of $\calV_0\backslash\calQ_0$ such that $\rmB(\bfv_{01},\bfv_{02})=1$, is
\begin{equation}
\label{eq.pf.lemma.counting pairs.1}
2^{2J_0-2}[2^{J_0-1}-\sign(\rmQ_0)][2^{J_0}-\sign(\rmQ_0)].
\end{equation}
When $J_0=1$ and $\sign(\rmQ_0)>0$,
the above quantity is $0$,
which is indeed the number of symplectic pairs in the singleton set $\calV_0\backslash\calQ_0$.
So now assume that either $J_0\geq 2$ or $J_0=1$ and $\sign(\rmQ_0)<0$.
We claim that for any $\bfv_{01}\in\calV_0\backslash\calQ_0$,
there exists $\bfv_{02}\in\calV_0\backslash\calQ_0$ such that $\rmB_0(\bfv_{01},\bfv_{02})=1$.
When $J_0=1$ and $\sign(\rmQ_0)<0$,
recall that $\calV_0$ consists of four vectors,
and that its three nonzero members are nonsingular and pairwise nonorthogonal.
To see that such a vector $\bfv_{02}$ exists in the remaining case where $J_0\geq2$, take any $\bfv\in\calV_0\backslash\set{\bfv_{01}}^\perp$, and so $\rmB_0(\bfv_{01},\bfv)=1$.
If $\rmQ_0(\bfv)=1$, let $\bfv_{02}=\bfv$.
Otherwise, $\rmQ_0(\bfv)=0$ and
$\rmQ_0(\bfv_{01}+\bfv)
=\rmQ_0(\bfv_{01})+\rmB_0(\bfv_{01},\bfv)+\rmQ_0(\bfv)
=1+1+0
=0$,
and so $\bfv_{03}:=\bfv$ and $\bfv_{04}:=\bfv_{01}+\bfv$ form a hyperbolic pair such that $\bfv_{03}+\bfv_{04}=\bfv_{01}$.
Here since $J_0\geq2$ and so
$\#(\calV_0\backslash\calQ_0)=2^{J_0-1}[2^{J_0}-\sign(\rmQ_0)]\geq 6$,
there exists $\bfv_{00}\in\calV_{00}:=\set{\bfv_{01},\bfv_{02}}^\perp$ such that $\rmQ_0(\bfv_{00})=1$.
Letting $\bfv_{02}=\bfv_{03}+\bfv_{00}$, we thus have
\begin{align*}
\rmB_0(\bfv_{01},\bfv_{02})
&=\rmB_0(\bfv_{03}+\bfv_{04},\bfv_{03}+\bfv_{00})
=\rmB_0(\bfv_{04},\bfv_{03})
=1,\\
\rmQ_0(\bfv_{02})
&=\rmQ_0(\bfv_{03}+\bfv_{00})
=\rmQ_0(\bfv_{03})+\rmB_0(\bfv_{03},\bfv_{00})+\rmQ_0(\bfv_{00})
=0+0+1
=1,
\end{align*}
as claimed.
For any such nonsingular symplectic pair $(\bfv_{01},\bfv_{02})$,
\eqref{eq.quadratic form recursion in general} becomes
\begin{equation}
\label{eq.pf.lemma.counting pairs.2}
\rmQ_0(x_{01}\bfv_{01}+x_{02}\bfv_{02}+\bfv_{00})
=x_{01}^2+x_{01}x_{02}+x_{02}^2+\rmQ_{00}(\bfv_{00})
\end{equation}
for any $x_{01},x_{02}\in\bbF_2$ and $\bfv_{00}\in\calV_{00}:=\set{\bfv_{01},\bfv_{02}}^\perp$.
Thus, the quadrics $\calQ_0$ and $\calQ_{00}$ of $\rmQ_0$ and $\rmQ_{00}$, respectively, satisfy
\begin{equation}
\label{eq.pf.lemma.counting pairs.3}
\calQ_0=\calQ_{00}\sqcup(\bfv_{01}+\calV_{00}\backslash\calQ_{00})\sqcup(\bfv_{02}+\calV_{00}\backslash\calQ_{00})\sqcup(\bfv_{01}+\bfv_{02}+\calV_{00}\backslash\calQ_{00}).
\end{equation}
As such, $\#(\calQ_0)=3\#(\calV_{00})-2\#(\calQ_{00})$ where
$\#(\calV_{00})=2^{2J_0-2}$ and $\#(\calQ_{00})=2^{J_0-2}[2^{J_0-1}+\sign(\rmQ_{00})]$,
implying
\begin{equation*}
2^{J_0-1}[2^J_0+\sign(\rmQ_0)]
=\#(\calQ_0)
=3\#(\calV_{00})-2\#(\calQ_{00})
=2^{J_0-1}[2^J_0-\sign(\rmQ_{00})],
\end{equation*}
namely that $\sign(\rmQ_{00})=-\sign(\rmQ_0)$.
(This is notably different than the aforementioned hyperbolic (singular symplectic) case in which $\rmQ$ and $\rmQ_0$ have the same sign and their respective quadrics satisfy~\eqref{eq.quadric relationship}.)
We now count the nonsingular symplectic pairs $(\bfv_{01},\bfv)$ in $\calV_0$.
As noted above, for any $\bfv_{01}\in\calV_0\backslash\calQ_0$,
there exists $\bfv_{02}\in\calV_0\backslash\calQ_0$ such that $\rmB_0(\bfv_{01},\bfv_{02})=1$.
Uniquely decomposing any $\bfv\in\calV_0$ as $\bfv=x_{01}\bfv_{01}+x_{02}\bfv_{02}+\bfv_{00}$ where $x_{01},x_{02}\in\bbF_2$ and $\bfv_{00}\in\calV_{00}=\set{\bfv_{01},\bfv_{02}}^\perp$,
we have $\rmB_0(\bfv_{01},\bfv)=1$ if and only if $x_{02}=1$,
in which case~\eqref{eq.pf.lemma.counting pairs.2} gives
$1=\rmQ_0(\bfv)=x_{01}^2+x_{01}^{}+1+\rmQ_{00}(\bfv_{00})$ if and only if $\bfv_{00}\in\calQ_{00}$ (regardless of $x_{01}\in\bbF_2$).
As such, the number of such pairs is the product of the number $\#(\calV_0\backslash\calQ_0)=2^{J_0-1}[2^{J_0}-\sign(\rmQ_0)]$ of choices for $\bfv_{01}$ with the number
$\#(\bbF_2)\#(\calQ_{00})
=2(2^{J_0-2})[2^{J_0-1}+\sign(\rmQ_{00})]
=2^{J_0-1}[2^{J_0-1}-\sign(\rmQ_0)]$ of choices for $\bfv$, namely~\eqref{eq.pf.lemma.counting pairs.1}.
Here~\eqref{eq.pf.lemma.counting pairs.3} also immediately implies that $\calV_0\backslash\calQ_0$ partitions into $\calV_{00}\backslash\calQ_{00}$,
$\bfv_{01}+\calQ_{00}$, $\bfv_{02}+\calQ_{00}$ and $\bfv_{01}+\bfv_{02}+\calQ_{00}$.
Since $\calQ_{00}\subseteq\calV_{00}=\set{\bfv_{01},\bfv_{02}}^\perp$ where $\rmB_0(\bfv_{01},\bfv_{02})=1$,
the set $\bfv_{01}+\bfv_{02}+\calQ_{00}$ is the only one of these four components whose members are nonorthogonal to both $\bfv_{01}$ and $\bfv_{02}$.

We now combine the previous results in the special case where $J\geq 2$ and $\rmQ_0$ is the restriction of $\rmQ$ to $\calV_0:=\set{\bfv_1,\bfv_2}^\perp$ where $(\bfv_1,\bfv_2)$ is a hyperbolic pair in $\calV$.
In particular, the members of a $4$-tuple $(\bfv_1,\bfv_2,\bfv_3,\bfv_4)$ lie in $\calQ$ and are pairwise nonorthogonal if and only if $(\bfv_1,\bfv_2)$ is hyperbolic and $\bfv_3=\bfv_1+\bfv_2+\bfv_{01}$ and $\bfv_4=\bfv_1+\bfv_2+\bfv_{02}$ where $(\bfv_{01},\bfv_{02})$ is a nonsingular symplectic pair in $\calV_0$.
The number of such $4$-tuples is thus the product of the number~\eqref{eq.lemma.counting pairs.1} of such $(\bfv_1,\bfv_2)$ with the number~\eqref{eq.pf.lemma.counting pairs.1} of such $(\bfv_{01},\bfv_{02})$ in the case where $J_0=J-1$ and $\sign(\rmQ)=\sign(\rmQ_0)$, namely~\eqref{eq.lemma.counting pairs.2}.
For any such $(\bfv_1,\bfv_2,\bfv_3,\bfv_4)$,
a vector $\bfv\in\calQ$ is nonorthogonal to both $\bfv_1$ and $\bfv_2$ if and only if $\bfv=\bfv_1+\bfv_2+\bfv_0$ for some $\bfv_0\in\calV_0\backslash\calQ_0$.
Moreover, any such $\bfv$ is nonorthogonal to both $\bfv_3=\bfv_1+\bfv_2+\bfv_{01}$ and $\bfv_4=\bfv_1+\bfv_2+\bfv_{02}$ if and only if $\bfv_0$ is nonorthogonal to both $\bfv_{01}$ and $\bfv_{02}$,
namely if and only if $\bfv_0=\bfv_{01}+\bfv_{02}+\bfv_{00}$ for some $\bfv_{00}\in\calQ_{00}$.
Thus, $\bfv\in\calQ$ is nonorthogonal to $\bfv_1$, $\bfv_2$, $\bfv_3$ and $\bfv_4$ if and only if $\bfv=\bfv_1+\bfv_2+\bfv_{01}+\bfv_{02}+\bfv_{00}=\bfv_1+\bfv_2+\bfv_3+\bfv_4+\bfv_{00}$ for some $\bfv_{00}\in\calQ_{00}$.
Here, $\rmQ_0$ is the restriction of $\rmQ$ to $\calV_0=\set{\bfv_1,\bfv_2}^\perp$,
and $\rmQ_{00}$ is the restriction of $\rmQ_0$ to the orthogonal complement of $\set{\bfv_{01},\bfv_{02}}=\set{\bfv_1+\bfv_2+\bfv_3,\bfv_1+\bfv_2+\bfv_4}$ in $\calV_0$, denoted $\calV_{00}$.
We claim that $\calV_{00}=\set{\bfv_1,\bfv_2,\bfv_3,\bfv_4}^\perp$.
Indeed, if $\bfv\in\set{\bfv_1,\bfv_2,\bfv_3,\bfv_4}^\perp$ then $\bfv\in\set{\bfv_1,\bfv_2}^\perp=\calV_0$ and $\bfv\in\set{\bfv_{01},\bfv_{02}}^\perp$,
implying $\bfv\in\calV_{00}$.
Conversely, if $\bfv\in\calV_{00}$ then $\bfv\in\calV=\set{\bfv_1,\bfv_2}^\perp$ and $\bfv\in\set{\bfv_{01},\bfv_{02}}^\perp$, implying $\bfv\in\set{\bfv_3,\bfv_4}^\perp$.
Thus, $\rmQ_{00}$ is the restriction of $\rmQ$ to $\calV_{00}=\set{\bfv_1,\bfv_2,\bfv_3,\bfv_4}^\perp$.
Since $\calV_0$ has codimension $2$ in $\calV$ and $\calV_{00}$ has codimension $2$ in $\calV_0$, this space $\calV_{00}$ has codimension $4$ in the $2J$-dimensional space $\calV$.
Moreover,
$\sign(\rmQ_{00})=-\sign(\rmQ_0)=-\sign(\rmQ)$.


\begin{thebibliography}{WW}
\bibitem{AlexeevCM12}
B.~Alexeev, J.~Cahill, D.~G.~Mixon,
Full spark frames,
J.\ Fourier Anal.\ Appl.\ 18 (2012) 1167--1194.

\bibitem{Andriamanalimanana79}
B.~Andriamanalimanana,
Ovals, unitals, and codes,
Ph.D.\ Dissertation, Lehigh University, 1979.

\bibitem{BandeiraDMS13}
A.~S.~Bandeira, E.~Dobriban, D.~G.~Mixon, W.~F.~Sawin,
Certifying the restricted isometry property is hard,
IEEE Trans.\ Inform.\ Theory 59 (2013) 3448--3450.

\bibitem{BandeiraFMW13}
A.~S.~Bandeira, M.~Fickus, D.~G.~Mixon, P.~Wong,
The road to deterministic matrices with the Restricted Isometry Property,
J.\ Fourier Anal.\ Appl.\ 19 (2013) 1123--1149.

\bibitem{BodmannK20}
B.~Bodmann, E.~J.~King,
Optimal arrangements of classical and quantum states with limited purity,
J.\ Lond.\ Math.\ Soc.\ 101 (2020) 393--431.

\bibitem{BourgainDFKK11}
J.~Bourgain, S.~Dilworth, K.~Ford, S.~Konyagin, D.~Kutzarova,
Explicit constructions of RIP matrices and related problems,
Duke Math.~J.\ 159 (2011) 145--185.

\bibitem{CalderbankCKS97}
A.~R.~Calderbank, P.~J.~Cameron, W.~M.~Kantor, J.~J.~Seidel,
$\mathbb{Z}_4$-Kerdock codes, orthogonal spreads, and extremal Euclidean line-sets,
Proc.\ Lond.\ Math.\ Soc.\ 75 (1997) 436--480.

\bibitem{Cameron81}
P.~J.~Cameron,
Finite permutation groups and finite simple groups,
Bull.~London Math.~Soc.\ 13 (1981) 1--22.

\bibitem{CameronS73}
P.~J.~Cameron, J.~J.~Seidel,
Quadratic forms over $GF(2)$,
Indag.\ Math.\ 76 (1973) 1--8.

\bibitem{Candes08}
E.~J.~Cand\`{e}s,
The restricted isometry property and its implications for compressed sensing,
C.\ R.\ Math.\ Acad.\ Sci.\ Paris 346 (2008) 589--592.

\bibitem{DeregowskaFFL22}
B.~Der\c{e}gowska, M.~Fickus, S.~Foucart, B.~Lewandowska,
On the value of the fifth maximal projection constant,
J.\ Funct.\ Anal.\ 283 (2022) 109634/1--16.

\bibitem{DempwolffK23}
U.~Dempwolff, W.~M.~Kantor,
On 2-transitive sets of equiangular lines,
Bull.\ Aust.\ Math.\ Soc.\ 107 (2023) 134--145.

\bibitem{DingF07}
C.~Ding, T.~Feng,
A generic construction of complex codebooks meeting the Welch bound,
IEEE Trans.\ Inform.\ Theory 53 (2007) 4245--4250.

\bibitem{DonohoE03}
D.\ L.\ Donoho, M.\ Elad,
Optimally sparse representation in general (nonorthogonal) dictionaries via $\ell^1$ minimization,
Proc.\ Natl.\ Acad.\ Sci.\ USA 100 (2003) 2197--2202.

\bibitem{FickusIJK21}
M.~Fickus, J.~W.~Iverson, J.~Jasper, E.~J.~King,
Grassmannian codes from paired difference sets,
Des.\ Codes Cryptogr.\ 89 (2021) 2553--2576.

\bibitem{FickusJKM18}
M.~Fickus, J.~Jasper, E.~J.~King, D.~G.~Mixon,
Equiangular tight frames that contain regular simplices,
Linear Algebra Appl.\ 555 (2018) 98--138.

\bibitem{FickusJMP18}
M.~Fickus, J.~Jasper, D.~G.~Mixon, J.~D.~Peterson,
Tremain equiangular tight frames,
J.\ Combin.\ Theory Ser.~A 153 (2018) 54--66.

\bibitem{FickusJMP21}
M.~Fickus, J.~Jasper, D.~G.~Mixon, J.~D.~Peterson,
Hadamard equiangular tight frames,
Appl.\ Comput.\ Harmon.\ Anal.\ 50 (2021) 281--302.

\bibitem{FickusJMPW18}
M.~Fickus, J.~Jasper, D.~G.~Mixon, J.~D.~Peterson, C.~E.~Watson,
Equiangular tight frames with centroidal symmetry,
Appl.\ Comput.\ Harmon.\ Anal.\ 44 (2018) 476--496.

\bibitem{FickusJMPW19}
M.~Fickus, J.~Jasper, D.~G.~Mixon, J.~D.~Peterson, C.~E.~Watson,
Polyphase equiangular tight frames and abelian generalized quadrangles,
Appl.\ Comput.\ Harmon.\ Anal.\ 47 (2019) 628--661.

\bibitem{FickusL23}
M.~Fickus, E.~C.~Lake,
Doubly transitive equiangular tight frames that contain regular simplices,
arXiv:2302.08879 (2023).

\bibitem{FickusM21}
M.~Fickus, B.~R.~Mayo,
Mutually unbiased equiangular tight frames,
IEEE Trans.\ Inform.\ Theory 67 (2021) 1656--1667.

\bibitem{FickusM16}
M.~Fickus, D.~G.~Mixon,
Tables of the existence of equiangular tight frames,
arXiv:1504.00253 (2016).

\bibitem{FickusMJ16}
M.~Fickus, D.~G.~Mixon, J.~Jasper,
Equiangular tight frames from hyperovals,
IEEE Trans.\ Inform.\ Theory.\ 62 (2016) 5225--5236.

\bibitem{FickusMT12}
M.~Fickus, D.~G.~Mixon, J.~C.~Tremain,
Steiner equiangular tight frames,
Linear Algebra Appl.\ 436 (2012) 1014--1027.

\bibitem{FickusS20}
M.~Fickus, C.~A.~Schmitt,
Harmonic equiangular tight frames comprised of regular simplices,
Linear Algebra Appl.\ 586 (2020) 130--169.

\bibitem{IversonM22}
J.~W.~Iverson, D.~G.~Mixon,
Doubly transitive lines I: Higman pairs and roux,
J.\ Combin.\ Theory Ser.\ A 185 (2022) 105540.

\bibitem{IversonM24}
J.~W.~Iverson, D.~G.~Mixon,
Doubly transitive lines II: Almost simple symmetries,
Algebr.\ Comb.\ 7 (2024) 37--76.

\bibitem{Kantor75}
W.~M.~Kantor,
Symplectic groups, symmetric designs, and line ovals,
J.\ Algebra 33 (1975) 43--58.

\bibitem{King19}
E.~J.~King,
$2$- and $3$-covariant equiangular tight frames,
Proc.\ Int.\ Conf.\ Sampl.\ Theory Appl. (2019) 1--4.

\bibitem{King25}
E.~J.~King,
$k$-Homogeneous equiangular tight frames,
arXiv:2505.00160 (2025).

\bibitem{Konig99}
H.~K\"{o}nig,
Cubature formulas on spheres,
Math.\ Res.\ 107 (1999) 201--212.

\bibitem{Lake23}
E.~C.~Lake,
Regular simplices within doubly transitive equiangular tight frames,
M.S.\ Thesis, Air Force Institute of Technology, 2023.

\bibitem{StrohmerH03}
T.~Strohmer, R.~W.~Heath,
Grassmannian frames with applications to coding and communication,
Appl.\ Comput.\ Harmon.\ Anal.\ 14 (2003) 257--275.

\bibitem{Tropp04}
J.~A.~Tropp,
Greed is good: Algorithmic results for sparse approximation,
IEEE Trans.\ Inform.\ Theory.\ 50 (2004) 2231--2242.

\bibitem{Turyn65}
R.~J.~Turyn,
Character sums and difference sets,
Pacific J.\ Math.\ 15 (1965) 319--346.

\bibitem{Welch74}
L.~R.~Welch,
Lower bounds on the maximum cross correlation of signals,
IEEE Trans.\ Inform.\ Theory 20 (1974) 397--399.

\bibitem{Wertheimer86}
M.~Wertheimer,
Designs in quadrics,
Ph.D.\ Dissertation, University of Pennsylvania, 1986.

\bibitem{XiaZG05}
P.~Xia, S.~Zhou, G.~B.~Giannakis,
Achieving the Welch bound with difference sets,
IEEE Trans.\ Inform.\ Theory 51 (2005) 1900--1907.

\end{thebibliography}
\end{document}